\documentclass[11pt]{article}

\usepackage[normalem]{ulem}

\usepackage{amsfonts,amsmath,amsthm,times,fullpage,amssymb,epsfig,graphicx,enumerate,bm}

\usepackage{hyperref}

\usepackage{subfig}
\hypersetup{colorlinks=false}

\usepackage[dvipsnames]{xcolor} 

\usepackage{color}
\usepackage{dsfont}
\usepackage{wrapfig}
\usepackage{relsize}

\usepackage{url}

\usepackage[left=1.0in,top=1.0in,right=1.0in,bottom=1.0in,nohead,textheight=10in,footskip=0.3in]{geometry}

\usepackage{bbm}

\usepackage[noend]{algpseudocode} 
\usepackage[nothing]{algorithm}  
\usepackage{caption}

\newtheorem{theorem}{Theorem}[section]
\newtheorem{lemma}[theorem]{Lemma}
\newtheorem{corollary}[theorem]{Corollary}

\newtheorem{definition}[theorem]{Definition} 
\newtheorem{remark}{Remark}[section]
\newtheorem{proposition}[theorem]{Proposition}  
 
\newtheorem{conjecture}[theorem]{Conjecture}

\newcommand\ex{\mathbb{E}}

\newcounter{fooTH}
\newcounter{fooEQ}

\begin{document}

\title{Group testing and local search:\\ 
is there a computational-statistical gap?
}

\author{
Fotis Iliopoulos
\thanks{This material is based upon work directly supported by the IAS Fund for Math and indirectly supported by the National Science
Foundation Grant No. CCF-1900460. Any opinions, findings and conclusions or recommendations expressed in this material are
those of the author(s) and do not necessarily reflect the views of the National Science Foundation. This work is also supported by the Simons Collaboration on Algorithms and Geometry.} \\ 
Princeton University \\and Institute for Advanced Study \\
{\small fotios@ias.edu}
 \\
 \and 
Ilias Zadik
\thanks{I.Z. is supported by a CDS Moore-Sloan postdoctoral fellowship.} \\ 
New York University\\
{\small zadik@nyu.edu}
}

\maketitle

\begin{abstract}

Group testing is a fundamental problem in statistical inference with many real-world applications, including the need for massive group testing during the ongoing COVID-19 pandemic. 
%The goal is to detect a set of $k$ defective items out of a population of size $p$ using  $n \ll p$ tests. Towards that end, one is allowed to  utilize a  procedure that can test groups of items, in which each test is returned positive if and only if at least one item in the group is defective. 
In this paper we study the task of \emph{approximate recovery}, in which we tolerate having a small number of incorrectly classified items. 

 One of the most well-known, optimal, and easy to implement testing procedures is the \emph{non-adaptive Bernoulli group testing}, where all tests are conducted in parallel, and each  item is chosen to be part of any certain test independently  with some fixed probability. In this setting, there is an observed gap between the  number of tests above which recovery is information theoretically possible, and the number of tests required by the currently best known efficient algorithms to succeed.  In this paper we seek to understand whether this computational-statistical gap can  be closed. Our main contributions are the following:

% In this paper we seek to understand whether such a phenomenon takes place for Bernoulli group testing as well.  

 \begin{enumerate}

 \item 
 Often times such gaps are explained by a phase transition in the landscape of the solution space of the problem (an \textit{Overlap Gap Property (OGP)} phase transition). We provide first moment evidence that, perhaps  surprisingly, such a phase transition does \emph{not} take place throughout  the regime for which recovery is information theoretically possible. This fact suggests that the model is in fact amenable to local search algorithms.

 \item We prove the complete absence of ``bad" local minima for a part of the ``hard" regime, a fact which  implies  an improvement over known theoretical results on the performance of efficient algorithms for approximate recovery without false-negatives.

 \item Finally, motivated by the evidence for the absence for the OGP, we present  extensive simulations that strongly suggest  that a \emph{very simple} local algorithm known as \emph{Glauber Dynamics}  does indeed succeed, and can be used to efficiently implement the well-known (theoretically optimal) Smallest Satisfying Set (SSS) estimator.  Given that practical algorithms for this task utilize Branch and Bound  and Linear Programming relaxation techniques, our finding could potentially be of practical interest.

 \end{enumerate}

\end{abstract}

\newpage

\section{Introduction}

Group testing is a fundamental and long-studied problem in statistical inference, which  was  introduced in the 1940s by Dorfman~\cite{dorfman1943detection}. The goal is to detect a set of $k$  \emph{defective items}, for a known parameter $k$, out of population of size $p$ using $n \ll p $ tests.  To achieve this, one is allowed to utilize a procedure that is able to test items in groups.  Each test is returned positive if and only if at least one item in the group is defective. 

As an illustrative example, group testing can be applied in the context of medical testing, enabling us to efficiently screen a population for a rare disease. In this setting, we assume that we have an estimate on the number of infected individuals (who correspond to the ``defective items"), as well as a way to take a sample, of say saliva or blood, from each individual, and test it. The idea then is that the number of tests needed to identify the infected individuals can be dramatically reduced by pooling samples.  Indeed, utilizing  group testing for pooling samples in medical tests was the original idea of Dorfman, and it is also highly relevant  nowadays due to the ongoing COVID-19 pandemic~\cite{cov3,cov2,cov1}. Besides medical testing though, group testing has found several real-world applications in a wide variety of areas, including communication protocols~\cite{anta2013unbounded},  molecular biology and DNA  library screening~\cite{cheng2008new}, pattern matching~\cite{clifford2010pattern},  databases~\cite{cormode2005s} and data compression~\cite{hong2001group}.

From an information theoretic point of view, the number of tests required depends on various assumptions on the mathematical model used, and the reader is referred to the  survey of Aldridge, Johnson and Scarlett~\cite{survey} for details on the several models that have been studied.
Here we  focus on the \emph{sublinear sparse regime}, i.e., the case where $k$ scales sublinearly with $p$; $k/p \rightarrow 0, $ as $p \rightarrow +\infty$. A standard sublinear setting in the group testing literature is when $k= \lfloor p^{\alpha } \rfloor$  for some constant sparsity parameter $ \alpha \in (0,1)$ \cite{survey}. This is the most interesting regime from a mathematical perspective, but also the one that is suitable for modeling the  early stages of an epidemic  \cite{wang2011evolution} in the context of medical testing. Further, we are interested in the so-called \emph{approximate} recovery task (also known as ``partial" recovery), where we tolerate having some small number of incorrectly classified items.  Note that, since in practical settings we only expect to have an estimation on the input number of defective items $k$, incorrect classifications may be unavoidable even if we are able to guarantee  ``exact" recovery. Note also that in several applications it might not be required for both \emph{false-positive} errors (non-defective items incorrectly classified as defective) and \emph{false-negative} errors (defective items incorrectly classified as non-defective)  to be zero. As an  example, when screening for diseases, a small number of false-positive errors might arguably be a small cost to pay  compared to performing many more pooled tests. 

More formally, let $\sigma^*$ denote the set of defective items and $\hat{\sigma}$ be the output of our estimator. For a set of items $\sigma$ let $\overline{\sigma}:= [p] \setminus \sigma $  denote the complement of $\sigma$.  Given $d \in \mathbb{N}$ let 
\begin{align*}
\Pr_d[\mathrm{err} ] := \Pr\left[ |\overline{\hat{\sigma}} \cap \sigma^*  | > d \text{ or } |\hat{\sigma} \cap \overline{\sigma^*}  | >d \right],
\end{align*}
denote the probability either the number of false negatives or false positives exceeds a common threshold $d$.
\begin{definition}\label{approx_rec_def}
Fix a parameter $\gamma \in (0,1)$ and let $d = \lfloor \gamma k \rfloor$. We say that an estimator achieves \emph{$(1-\gamma)$-approximate recovery} asymptotically almost surely if $\lim_{p \rightarrow +\infty } \Pr_d[ \mathrm{err} ]  = 0$.
\end{definition} Similar to the above definition, everywhere in this work, we say than a sequence of events $(A_p)_{p \in \mathbb{N}}$ happen asymptotically almost surely \textbf{(a.a.s.)} as $p \rightarrow +\infty$ if $\lim_{p \rightarrow +\infty} \Pr\left[A_p \right]=1.$

% \begin{align*}
% \Pr_d[\mathrm{err} ] := \Pr[ | ([p] \setminus \sigma)  \cap \sigma^*  | > d \text{ and } |\hat{\sigma} \cap \overline{\sigma^*}  | >d ],
% \end{align*}

Finally, in this paper we consider the well-known and very simple to implement \emph{non-adaptive Bernoulli} group testing design, where all tests are conducted in parallel (non-adaptivity) and each item is part of any certain test with some fixed probability and independently of each other item. Despite its simplicity,   Bernoulli group testing is asymptotically information theoretically optimal  in the context of \emph{$(1-o(1) )$-approximate recovery} in the following sense. (The first part of Theorem~\ref{approx_phase} is proven in~\cite{scarlett2016phase}, while the second part in~\cite{separate_decoding}.) 
% Let $\theta^* \in \{0,1\}^p$ denote the indicator vector of the defective items, and $\hat{\theta} \in \{0,1\}^p$ the indicator vector corresponding to our estimator. Also, given $d \in \mathbb{N}$, let $\Pr_{d}[ \mathrm{err}]:=  \Pr[ d_{H}( \theta^*,  \hat{\theta} ) >  d]  $.  (Here $d_H( \cdot, \cdot) $ denotes the Hamming distance between two binary vectors.) We have the following theorem.

\begin{theorem}[\cite{scarlett2016phase,separate_decoding}]\label{approx_phase}
Let $k, p,d \in \mathbb{N}$ with $1 \le k \le p$. We assume that $k,p \rightarrow +\infty$ with $k = o(p)$.  Fix parameters $\gamma, \eta \in (0,1)$. Under  non-adaptive Bernoulli group testing in which each item participates in a test with probability $\nu/k$, and  $d = \gamma k$, we have the following:
\begin{enumerate}[(a)]
\item   With $\nu$ satisfying $\left( 1- \frac{\nu}{k} \right)^{k} = \frac{1}{2} $, there exists an estimator with the property that $\lim_{p \rightarrow +\infty} \Pr_{d}[ \mathrm{err} ] = 0$, provided that $n > (1+\eta) n^*$,
where $n^* =  k \log_2 \frac{p}{k}$.
%\begin{align*}
%n^* =  k \log_2 \frac{n}{k}.
%\end{align*}
\item For \emph{any test design} and any estimator, in order to achieve $ \lim_{p \rightarrow +\infty} \Pr_d[  \mathrm{err}] = 0 $ it is necessary that $n > (1-\eta)n_{\gamma}^*$, where $n_{\gamma}^* = (1-\gamma) k \log_2 \frac{p}{k}  = (1-\gamma) n^*.$
%\begin{align*}
%n_{\gamma}^* = (1-\gamma) k \log_2 \frac{n}{k}  = (1-\gamma) n^*.
%\end{align*}
\end{enumerate}
\end{theorem}

% \begin{remark}
% In~\cite{scarlett2016phase} Part (a) of Theorem~\ref{approx_phase} is stated with $\nu = \ln 2$
% \end{remark}

% \begin{remark}\label{den_tairiazete}

% The bounds in Theorem~\ref{approx_phase} asymptotically match  the lower bound of $ n \ge \log_2 {p \choose k  }  \approx k \log_2 \frac{ p}{ k} = n^* $ tests  which are required for exact recovery in \emph{any} test design. (This lower bound follows from a simple counting argument based on the pigeonhole principle, as we attempt to encode with $n$ bits $\binom{p}{k}$ sets. Furthermore it extends to the non-zero error probability case, see e.g~\cite{chan2011non}). Note that the approximate recovery assumption is crucial here as it is known that Bernoulli group testing is asymptotically information theoretically optimal only for relatively small values of the sparsity  parameter $\alpha$ in terms of the exact recovery task, i.e., only for $\alpha \le 1/3$, and it is suboptimal otherwise. \textcolor{red}{Ref}
% \end{remark}

\begin{remark}\label{toso_antikeimenika}
 Strictly speaking, the result in~\cite{scarlett2016phase} requires $\nu = \ln 2$ for Part (a) of Theorem~\ref{approx_phase}, which is though  equal asymptotically to the value of $\nu$ as we stated it. In Lemma~\ref{key_approx_lemma} we provide a slightly more general result which justifies this minor change.
\end{remark}

\begin{remark}\label{sou_lew} 
It is known~\cite{coja2020optimal}  that performing $n_{\mathrm{exact} }$ tests is asymptotically a necessary and sufficient condition (information theoretically)  for \emph{exact} recovery in the  non-adaptive group testing setting, where $k=p^{\alpha}, \alpha \in (0,1),$
 \begin{align}\label{bound_exact}
 n_{\mathrm{exact} } =  \max \left\{  \frac{k}{ \ln2 }  \log_2 k   , n^* \right\} = \max \left\{ \frac{\alpha}{\ln2}   , 1-\alpha \right\} p^{\alpha} \log_2 p.
 \end{align}
Observe that for $\alpha   > \frac{1}{1 + \frac{1}{ \ln2 }  } \approx  0.41 $ the first term in~\eqref{bound_exact} dominates and, therefore,  the approximate recovery bound promised by Theorem~\ref{approx_phase} provides an  improvement in the number of tests required. Note also that Coja-Oghlan et al. obtain this result  via a sophisticated test-design other than Bernoulli group-testing.
\end{remark}

We remark that the estimator promised by the first part Theorem~\ref{approx_phase} is in principle \textit{not computationally efficient}. In particular, Theorem~\ref{approx_phase} refers to exhaustive  search towards finding a set of $k$ items which (nearly) satisfies all the tests, namely none of the $k$ items participate in a negative test, and at least one of them participates in each positive test. (See Lemma~\ref{key_approx_lemma} for an exact statement and~\cite{scarlett2016phase,separate_decoding} for a relevant discussion). Note also that exhaustive search in principle takes $\binom{p}{k}$ time.  

The currently best known efficient algorithm~\cite{scarlett2018near} for $(1-o(1))$-approximate recovery under non-adaptive Bernoulli group testing requires $(\ln 2)^{-1}  \cdot n^*$  tests asymptotically.
This gap between the  number of tests above which recovery is information theoretically possible, and the number of tests required by the currently best known algorithms to succeed,   is usually referred to as a \emph{computational-statistical gap}. The extent to which this gap is of fundamental nature has drawn significant attention in the group testing community and has been explicitly posed as one of nine open problems in the recent theoretical survey on the topic \cite[Section 6, Open Problem 3]{survey}. The natural belief that  simple ``random-design" settings exhibit a potentially fundamental such computational-statistical gap, similar to the one observed in the infamous planted clique model (see e.g. the introduction on \cite{gamarnik2019landscape} and below), is one of the main reasons why researchers in the area of group testing have studied test designs other than Bernoulli testing (see e.g.~\cite{coja2020optimal, johnson2018performance, mezard2008group, wadayama2016nonadaptive,nakoni}).  Indeed,  recovery in the non-adaptive Bernoulli group testing setting can be seen as equivalently solving \emph{planted Set Cover} instances coming from a certain random family. (Recall that Set Cover, similar to Max Clique, is an NP-complete problem. We discuss this connection in Remark~\ref{set_cover}.) Notably,   in the context of exact recovery via non-adaptive testing,  the line of work going beyond the Bernoulli design  culminated  with the very recent paper of Coja-Oghlan et.al.~\cite{coja2020optimal} which shows that $n_{\mathrm{exact} }$ is the information-theoretic/algorithmic phase transition threshold.  Specifically, Coja-Oghlan et.al. provide  a sophisticated test design which is inspired by recent advances in coding theory known as spatially coupled low-density parity check codes~\cite{felstrom1999time,kudekar2011threshold},  which matches the information theoretical optimal bound, $n_{\mathrm{exact} }$. In addition, they propose an  efficient algorithm which, given as input  $n_{\mathrm{exact}}$ tests (asymptotically) from their test design,  successfully infers the set of defective items with high probability. 

In the light of Theorem~\ref{approx_phase}, Remarks~\ref{toso_antikeimenika},~\ref{sou_lew}, the simplicity of the Bernoulli group testing design, and our discussion above, a natural  and important question both from a theoretical and application point of view is the following:

\medskip
\emph{Is there a computational-statistical gap under non-adaptive Bernoulli group testing for the task of  $(1-o(1))$-approximate recovery?}
\medskip

Our main contribution in this paper  is to  provide  extensive theoretical and experimental evidence suggesting that, perhaps surprisingly, such a gap does \emph{not} exist. We do so by showing  that in some, potentially fundamental, geometric way the Bernoulli group testing behaves differently than other models exhibiting computational-statistical gaps. In particular, we provide evidence which indicate that the inference task at hand  can be performed efficiently via a very simple local search algorithm. 

We stress that, given  the recent results of~\cite{coja2020optimal} regarding exact recovery, it would not be surprising if a ``spatial coupling"-inspired   test design  also  succeeds in closing the computational-statistical gap in the approximate recovery task. After all, ``spatial coupling"-inspired models are known to not exhibit gaps that appear in simpler random models in all the contexts they have been applied, see  e.g.~\cite{achlioptas2016bounds} for a discussion of this phenomenon in the context of random constraint satisfaction problems and e.g.~\cite{DonohoSpatial,KrzakalaETAL} in the context of compressed sensing. However, the main message of this paper is that the very simple non-adaptive Bernoulli group testing design may already be sufficient for  computationally efficient $(1-o(1))$-approximate recovery.

We discuss our results informally forthwith, and we present them formally in Appendices~\ref{formal_statement} and~\ref{simulations}.  We  remark that throughout the paper we will be interested in different values of the parameter $\nu$, which recall that  dictates the probability $\nu/k$ with which each item participates in a certain test, being precise in each of our statements.

Computational gaps between what existential or brute-force/information-theoretic methods promise and what computationally efficient algorithms achieve arise in the study of several ``non-planted" models like  random constraint satisfaction problems, see e.g.~\cite{mitsaras_barriers,coja2015independent,gamarnik2014limits}, but also several ``planted" inference algorithmic tasks, like  high dimensional linear regression~\cite{gamarnik2017high, gamarnik2019sparse},  tensor PCA~\cite{arous2020algorithmic, AndreaTensor}, sparse PCA \cite{gamarnik2019overlap, ArousSparse} and, of course, the planted clique problem~\cite{jerrum1992large, BarakClique, gamarnik2019landscape}. In the latter context the gaps are commonly referred to as computational-statistical gaps, as we mentioned above.  Often times such gaps are  explained by the presence of a geometric property in the solution space of the problem known as \emph{Overlap Gap Property} (OGP), which has been repeatedly observed to appear exactly at the regime where local algorithms cannot solve efficiently the problem. Additionally, it has also been observed that in the absence of this property the problem is amenable to simple local search algorithms. OGP is a notion  originating in spin glass theory and the groundbreaking work of Talagrand~\cite{talagrand2003spin} and, in some form, it has been originally introduced in the context of computational gap for the ``non-planted model" of random $k$-SAT~\cite{achlioptas2006solution,mezard2005clustering}.  Recently it has been defined and analyzed also in the context of computational-statistical gaps for the high dimensional linear regression model \cite{gamarnik2017high,gamarnik2019sparse},  the tensor PCA model~\cite{arous2020algorithmic}, the planted clique model \cite{gamarnik2019landscape} and the sparse PCA model \cite{ArousSparse,gamarnik2019overlap}. In all such planted models, it is either proven, or suggested by evidence, that the OGP phase transition takes place exactly at the point where we observe and expect local algorithms to work. From a rigorous point of view, the existence of OGP has been proven to imply the failure of various MCMC methods \cite{gamarnik2019landscape,gamarnik2019overlap,ArousSparse}, yet a rigorous proof that at its absence local method works remains one of the important conjectures in this line of research. Interestingly, progress in the last direction has been recently made in the context of certain ``non-planted" mean field spin glass systems \cite{subag2019following, AndreaFOCS19, AlaouiSpin20} where under the conjecture of the absence of a property similar to OGP the success of a certain ``local" Approximate Message Passing method has been established. In this work  we attempt to locate exactly the OGP phase transition (if it exists at all) in order to understand if the non-adaptive Bernoulli group testing exhibits  a computational-statistical gap.

%It should be noted that from a rigorous point of view, only recently it has been established that the existence of OGP implies indeed the failure of local methods, in the context of ``non-planted models" and of ``planted models" \cite{gamarnik2019landscape, gamarnik2019sparse, ArousSparse}. Furthermore, the belief that the abscence of (a variant of) OGP implies the success of local methods has been recently establish in the context of ``non-planted" mean-field spin glass systems and it unfortunately missing in the context of ``planted settings". 

It should be noted that over the recent years researchers have approached the study of whether such gaps are fundamental from various different angles other than the OGP. For example, researchers have studied average-case reductions between various inference models (see e.g. \cite{BRred13,BBH18, brennan2020reducibility} and references therein), the performance of various restricted classes of inference algorithms, such as low-degree methods (see e.g. \cite{BarakSOS, hopkins-thesis, kunisky2019notes, schramm2020computational}), message passing algorithms such as Belief Propagation and Approximate Message Passing (see e.g. \cite[Chapters 3,4]{bandeira2018notes} and references therein), and statistical-query algorithms (see e.g. \cite{VitalySQ}).

To conclude, in line with the thread of research of studying computational-statistical gaps via their geometric phase transitions, in this paper we seek to understand whether the optimization landscape of the $(1-o(1))$-approximate inference task under Bernoulli group testing is smooth enough for local search algorithms to succeed.

\section{Contributions}

In this section we report our main results. Due to the technical nature of many of our theoretical results, we choose to do this in the main body of the paper in an informal manner and, in particular, via the informally stated Theorems~\ref{first_moment_decreasing_inf},~\ref{local_search_theorem_informal} and Corollary~\ref{no_false_negatives_inf}. In the Appendix we present formally all the statements with their proofs.
In this section we also discuss the main insights from our experimental results (see Figures~\ref{Optimization} and~\ref{Comparison}), but we give the full details in Appendix~\ref{simulations}.

\subsection{Absence of the Overlap Gap Property (OGP) }

For our first and main result we provide first moment evidence suggesting that the landscape of the optimization problem corresponding to the inference task does \emph{not} exhibit the Overlap Gap Property accross the regime where inference is information-theoretic possible. We state our result formally in Appendix~\ref{absence_ogp}, where we also give details regarding its  technical aspects. Here we informally discuss why we believe this is a potentially fundamental reason for computational tractability, and what we mean when referring to ``first moment evidence".

A first observation towards approximate recovery is  that one can remove from consideration any item that participates in a negative test, as such an item is certainly non-defective. We call the remaining items as \emph{potentially defective}. We also say that a potentially defective item \emph{explains} a certain positive test if it participates in it. The following key lemma informs us that finding a set of  $k$ potentially defective items whose elements explain all but a few positive tests  implies approximate recovery. Its proof can be found in Appendix~\ref{key_approx_lemma_proof}. 

\begin{lemma}\label{key_approx_lemma}
Let $k, p \in \mathbb{N}$ with $1 \le k \le p$. We assume that $k,p \rightarrow +\infty$ with $k = o(p)$.  Fix parameters $\delta, \epsilon \in (0,1)$. Assume we observe  $n \geq  (1+\epsilon) k \log_2 \frac{p}{k}  $ tests  under  non-adaptive Bernoulli group testing in which each item participates in a test with probability $\nu/k$, where $\nu$ satisfies $(1-\frac{\nu}{k })^{k} = \frac{1}{2}$. There exists an $\epsilon'>0$ such that every set of size $k$ whose elements explain at least $(1-\epsilon')n$ tests must contain at least $\lfloor(1-\delta) k \rfloor$ defective items, a.a.s. as $p \rightarrow +\infty.$
\end{lemma}

Based on Lemma~\ref{key_approx_lemma}, to perform optimal approximate recovery we consider the case where $\nu>0$ satisfies $(1-\frac{\nu}{k})^k=\frac{1}{2}$ and consider the task of minimizing the number of unexplained tests over the space of $k$-tuples of potentially defective items, which we denote by $\Omega_k$.  In particular, we seek to understand ``smoothness" properties of the underlying optimization landscape.   

Informally, the Overlap Gap Property (OGP) with respect to this minimization problem says that the near-optimal solutions of this problem form two disjoint and well-separated clusters, one  corresponding to sets in $\Omega_k$  which are ``close" to the set of defective items, and one corresponding to sets in $\Omega_k$ which are ``far" from it. (See  Definition~\ref{def:OGP} in the Appendix for a rigorous definition.) At the presence of such a disconnectivity property, one can rigorously prove that a class of natural local search algorithms fails (see e.g. \cite{gamarnik2019landscape}). As mentioned above a highly non-rigorous, yet surprisingly accurate in many contexts computational prediction \cite{gamarnik2017high, gamarnik2019landscape, gamarnik2019sparse, ArousSparse}  is that the OGP can be the only type of computational bottlenecks there is for local algorithms. In other words, the prediction is that, at its absence, an appropriate local search algorithm works efficiently and can solve the optimization problem to (near) optimality. In this work, we provide ``first moment evidence" that the OGP \textit{never appears} in the landscape of the mentioned minimization problem when inference is possible, i.e. when $n \geq (1+\epsilon)k \log_2 \frac{p}{k}$ for any $\epsilon>0.$

Now in order to explain what we mean by ``first moment evidence" for the absence of the OGP, let us point out that the standard way of studying the existence of OGP is by studying the monotonicity of the function $\phi(\ell)$, namely the minimum possible value of unexplained tests, which is minimized over the sets in $\Omega_k$ that contain \textit{exactly $\ell$ defective items} (has overlap with the set of defective items equal to $\ell$)  (see Definition~\ref{phidef} for a formal definition, and also~\cite{gamarnik2017high, gamarnik2019landscape}). This is because a necessary implication of the existence of OGP is that, roughly speaking, $\phi(\ell)$ is not decreasing. (See Lemma~\ref{lem:monot} for a formal statement.) Intuitively, such a connection holds since OGP roughly implies that the overlap values $\ell$ for which $\phi(\ell)$ is small (few unexplained sets) are either ``large" (corresponding to large intersection with the true defective items — the ``close" cluster) or ``small" (corresponding small intersection with the true defective items — the ``far" cluster). Hence,  such a function cannot be decreasing with $\ell$.

Towards understanding the monotonicity properties of $\phi(\ell)$ we need to understand the value $\phi(\ell)$ which is the optimal value of a restricted random combinatorial optimization problem. For this we use the moments method. In particular,  we define for $t,\ell \ge 0$, the counting random variable
\begin{align*}
    \Sigma_{\ell,t} = \{ \sigma \in \Omega_k:  | \sigma \cap \sigma^* | = \ell, H(\sigma) \le t   \},
\end{align*}
where recall that $\sigma^*$ denotes the set of defective items, and let $H(\sigma)$ denote the number of unexplained tests with respect to $\sigma$. Observe that
\begin{align*}
\phi(\ell) \le t \Leftrightarrow \Sigma_{\ell,t}\geq 1,
\end{align*}
and that, in particular,  by Markov's inequality and Paley's-Zigmund's inequality we have for all $t>0$ and $\ell \in \{0,1,2,\ldots, \lfloor(1-\epsilon) k \rfloor \}$:
\begin{align*}
\frac{\mathbb{E}\left[\Sigma_{t,\ell}\right]^2}{\mathbb{E}\left[\Sigma^2_{t,\ell}\right]} \leq \Pr\left[\phi(\ell) \le t\right] = \Pr[  \Sigma_{\ell,t } \ge 1  ]  \leq \mathbb{E}\left[\Sigma_{t,\ell}\right].
\end{align*}
Hence, if for some $t_1,t_2>0$ it holds $\mathbb{E}\left[\Sigma_{t_1,\ell}\right]=o(1)$ we have $\phi(\ell)>t_1$ a.a.s as $p \rightarrow +\infty$ and if $\frac{\mathbb{E}\left[\Sigma_{t_2,\ell}\right]^2}{\mathbb{E}\left[\Sigma^2_{t_2,\ell}\right]} =1-o(1)$ or equivalently $\frac{\mathrm{Var}\left[\Sigma_{t_2,\ell}\right]}{\mathbb{E}\left[\Sigma_{t_2,\ell}\right]^2} =o(1)$ we have $\phi(\ell) \leq t_2$ a.a.s. The employment of the first moment to get an a.a.s. lower bound is called a first moment method, and the employment of the second moment to get an a.a.s. upper bound is called the second moment method.

In many cases of sparse combinatorial optimization problems it has been established in the literature that the first and second moment methods can be proven for $t_1,t_2$ sufficiently close to each other, sometimes satisfying even $t_2 \leq t_1+2,$ a phenomenon known as 2-point concentration (see e.g. \cite{bollobas1982vertices} or the more recent \cite{balister2019dense}). Here we say that we provide first moment evidence for the OGP, because we do not check the second moment method to prove the sufficient  concentration of measure, but we only use the first moment to derive a prediction for the value on which $\phi(\ell)$ concentrates on.  Specifically, we \textit{define} the function which maps $\ell$ to the value $t=t_{\ell}>0$ for which the first moment satisfies $$\mathbb{E}\left[\Sigma_{t,\ell}\right]=1, $$ and we call it the \textit{first moment function} and denote it by $F(\ell)$. We use $F(\ell)$ as an approximation for the value of $\phi(\ell).$ Note that the first moment prediction is expected to correspond to the critical $t>0$ above of which this first moment will blow up (where the second moment method can work) and below it will shrink to zero (where Markov inequality works). The first moment prediction has been proven, admittedly via lengthy and elaborate conditional second moment arguments, to provide tight predictions for the associated function $\phi(\ell)$ in many setting similar to the group testing where OGP has been studied, such as the regression setting \cite{gamarnik2017high} and the planted clique setting \cite{gamarnik2019landscape}. Notably, in these setting the approximation is tight enough so that the existence or not of OGP is directly related to the monotonicity of the first moment function, and furthermore the transition point where this function becomes decreasing corresponds exactly to the point where local algorithms are expected to work. Finally, it is worth pointing out that a similar ``first moment", or ``annealed complexity", method has been employed for the study of the landscape of the spiked tensor PCA model \cite{MeiPCA, Biroli19}.

In our work, we explicitly derive the first moment function $F(\ell)$ for the Bernoulli group testing model and prove that it remains decreasing (in the sense of Lemma~\ref{lem:monot}) \textit{throughout} the regime $n \geq (1+\epsilon)k \log_2 \frac{p}{k}$ for any $\epsilon>0$. 
\begin{theorem}[Informal Statement]\label{first_moment_decreasing_inf}
Let $k,p \in \mathbb{N}$ with $ 1 \leq k \leq p$. We assume that $k,p \rightarrow +\infty$ with $k=o(p)$ and consider the Bernoulli group testing model with $\nu>0$ such that $(1-\frac{\nu}{k})^k=\frac{1}{2}.$ Then if $n \geq (1+\epsilon)k \log_2 \frac{p}{k}$ for any $\epsilon>0,$ (the information-theoretic threshold for the problem), the first moment function of the model is strictly decreasing.
\end{theorem} One can pictorially observe the decreasing property of the first moment curve in Figure \ref{fig:monoton} all the way to the information-theoretic threshold. To get the result into context we invite the reader to compare this behavior with the similar monotonicity property of the first moment curve in the regression setting \cite{gamarnik2017high} and its widely believed computational-statistical gap. In Figure 1 of the arXiv version of \cite{gamarnik2017high}  one can see that the first moment curve in the information-theoretic relevant regime, transitions as $n$ increases from being non-monotonic (exactly at the conjectured ``hard" regime), to being monotonic (exactly at the ``easy" regime). Our result is that such a transition \textit{never appears} in the Bernoulli group testing model.

Upon a conjectured tightness of the second moment method, we show how this implies the absence of the OGP when  $n \geq (1+\epsilon)k \log_2 \frac{p}{k}$ for any $\epsilon>0$ (see Theorem \ref{thm:OGP}).  We consider this notable evidence that a local search algorithm can succeed in this regime, suggesting that there is actually no computational-statistical gap. Remarkably this is in full agreement with our experiments section below. We formalize and further explain all the latter statements in Appendix~\ref{absence_ogp}.

\subsection{Absence of bad local minima} 

For our second theoretical contribution we study a much more strict notion of optimization landscape smoothness, namely the absence of ``bad"  local minima. This is a very stringent, but certainly sufficient, condition for the success of even  greedy local improvements algorithms. Hence, it is naturally not expected to coincide with the OGP phase transition where a potentially more elaborate local algorithm could be needed (see \cite{gamarnik2019sparse} for a similar result in the context of regression). Yet, our evidence that the OGP does not appear for the Bernoulli group testing model when $n \geq (1+\epsilon)k \log_2 \frac{p}{k}$ for any $\epsilon>0$, conceivably suggests that the landscape of the minimization problem could be smooth enough to not even contain bad local minima when $n \geq (1+\epsilon)k \log_2 \frac{p}{k}$ for some  small values of $\epsilon>0.$ 

We consider the same objective function as in the previous paragraph, but this time we study the space of $k'$-tuples of potentially defective items, where $k' = \lfloor (1+\epsilon) k \rfloor$ for any fixed $\epsilon \in [0,1)$. Crucially, we now study all the possible Bernoulli group testing designs by considering a fixed by arbitrary parameter $\nu$ (recall that for each test, each item is tested with probability $\nu/k$).  Informally, a bad local minimum $\sigma$ is a $k'$-tuple of potentially defective items that contains a non-negligible number of non-defective items,  and such that every other $k'$-tuple in Hamming distance two from $\sigma$ explains at most the same number of positive tests. We give a formal definition in Appendix~\ref{undesirable_local}, Definition~\ref{bad_local_minima}.

 Our  result  is a bound on the number of tests required for the absence of bad local minima as a function of parameters $\epsilon,\delta, \nu$, where we are interested in $(1-\epsilon- \delta)$-approximate recovery. We state the theorem informally below, and formally in Appendix~\ref{undesirable_local}, Theorem~\ref{local_search_theorem}.

\begin{theorem}[Informal Statement]\label{local_search_theorem_informal}
Let $k,p \in \mathbb{N}$ with $ 1 \leq k \leq p$. We assume that $k,p \rightarrow +\infty$ with $k=o(p)$. Fix parameters $\nu>0, R \in (0,1),\delta, \epsilon \in [0,1)$ such that $\delta \epsilon > 0$, set $ k' = \lfloor(1+\epsilon)k \rfloor$, and assume that we observe the outcome of  $n=\lfloor \frac{ \log_2 \binom{p}{k}}{R} \rfloor$ tests under non-adaptive Bernoulli group testing in which each item participates in a test with probability $\nu/k$. If 
\begin{align*}
R<\frac{   \nu \mathrm{e}^{-\nu}   }{\ln2} + \max_{\lambda \geq 0} \min_{\zeta \in [0,1- \delta )} Q(\lambda, \zeta, \nu,  \epsilon), 
\end{align*}
where $Q = Q(\lambda, \zeta, \nu, \epsilon) $ is given in~\eqref{formula} (due to its elaborate form), then there exists no bad local minima with respect to $(1-\epsilon-\delta)$-approximate recovery in the space of $k'$-tuples of items.
\end{theorem}

\begin{remark}\label{Pink}
Note that for any $\nu>0,$ $n_{\mathrm{easy} }  =   \lfloor  ( \frac{   \nu \mathrm{e}^{-\nu}   }{\ln2} )^{-1}    \log_2 {p \choose k }   \rfloor $  is the number of tests required for $(1-o(1) )$-approximate recovery via the straightforward algorithms known as Combinatorial Orthogonal Matching Pursuit (COMP) and Definite Defectives (DD).  COMP simply amounts to outputting every item that does not participate in a negative test.  DD amounts to first discarding every item that participates in a negative test, and then  outputting the set of items which have the property that they are the only item in a certain positive test (see also ~\cite{survey}).
\end{remark}

According to   Remark \ref{Pink},  Theorem \ref{local_search_theorem_informal} quantifies the decrease in the number of tests that is possible  (with respect to the requirements of straightforward algorithms like DD and COMP),  while at the same time guaranteeing that the landscape remains smooth enough to not contain any bad local minima. Nonetheless, note that  the final result includes a complicated optimization problem of some large deviation function $ \max_{\lambda \geq 0} \min_{\zeta \in [0,1- \delta )}Q(\lambda, \zeta, \nu,  \epsilon)$, which the bigger it is, the more tests can be saved (and of course it is always non-negative as it can be directly checked that it obtains the value $0$ for $\lambda=0$).

After proper inspection of the properties of the function $Q$, and  exploiting the facts that in Theorem~\ref{local_search_theorem_informal}  we allow $\epsilon$ to be positive and  $\nu$ to take any value of our choice, we manage, as a corollary, to (slightly) improve the state-of-the art  result regarding the number of tests required for efficient approximate recovery under the restriction that \textit{no false-negatives} are allowed. As we have already mentioned, such a requirement is potentially desirable in medical testing.  We state the corollary informally below and formally in Appendix~\ref{undesirable_local}, Corollary~\ref{no_false_negatives}. 

\begin{corollary}[Informal Statement]\label{no_false_negatives_inf}
Let $k, p \in \mathbb{N}$ with $1 \le k \le p$. We assume that $k,p \rightarrow +\infty$ with $k = o(p)$.   Under non-adaptive Bernoulli group testing in which each item participates in a test with probability $\ln(5/2)/k$, there exists a greedy local search algorithm such that, asymptotically almost surely, given as input at least   $\lfloor  1.829  \log_2 \binom{p}{k}\rfloor$  tests, it outputs a set of $ \lfloor 1.01 k \rfloor$ items that is guaranteed to contain every defective item.
\end{corollary}

\begin{remark}
The constant 1.01 is chosen for concreteness and exposition purposes and in order to simplify the proof of Corollary~\ref{no_false_negatives_inf}. We note that, with some additional technical effort, the constant can be chosen to be  $1+\epsilon$, for any $\epsilon>0$.
\end{remark}

To put  Corollary~\ref{no_false_negatives_inf}  in to context, let us point out that if $k = \lfloor p^{\alpha} \rfloor, \alpha \in (0,1)$ and we make no assumptions on the value of $\alpha \in (0,1)$, then the best known algorithm (in terms of the  number of tests it requires to succeed) for  
approximate recovery with the guarantee of never introducing false-negative errors is  COMP~\cite{survey}. Indeed, COMP succeeds given as input at least  $ n_{\mathrm{COMP} } \approx \lfloor  1.883 \log_2 {p \choose k }    \rfloor$  tests.

If $\alpha$ is known to be appropriately small then, to the best of our knowledge, the currently best known algorithm  among the ones that with high probability never  introduce false-negative errors is the so-called Separate Decoding algorithm of~\cite{scarlett2018near}. In particular, Separate Decoding requires  as input (asymptotically) at least
\begin{align*}
n_{\mathrm{SP} } =   \min_{ \delta > 0  }  \max \left\{ \frac{1}{(1-\delta)\ln 2}  ,    \frac{ \alpha }{ \left(  (1-\delta) \ln  (1-\delta) + \delta   \right)  (1- \alpha)    \ln 2  }	\right\}  \cdot       \log_2 {p \choose k } 
\end{align*}
tests,  assuming we choose the value of $\nu$ so that it satisfies $(1-\nu/k)^k = 1/2$.  (Note that $\ln 2 \approx 0.693$.)
Thus, if $\alpha$ is sufficiently large, say $\alpha\ge 0.56 $,  then it can be easily checked that the algorithm of  Corollary~\ref{no_false_negatives_inf}  outperforms both COMP and Separate Decoding on the task of approximate recovery with the guarantee of never introducing false-negative errors, with high probability.

As a final remark,  the proof of Theorem~\ref{local_search_theorem_informal} is conceptually straightforward but technically elaborate.  We thus view it as an important first step towards a  rigorous understanding of the optimization landscape. (We stress though that, based on our calculations, we do not expect the complete absence of bad local minima at rates close to the information theoretic threshold. In other words, we conjecture that the use of \textit{stochastic} local search (\textit{positive temperature} local MCMC methods) would be necessary for success at rates close to the information theoretic threshold. We leave the rigorous investigation of the above fact as future work.) The main proof strategy is presented in Appendix~\ref{local_search_proof}, while  the proofs of some of the more technical intermediate lemmas are deferred to Appendix~\ref{omitted_local}.

 %\subsection{Finding the Smallest Satisfying Set (SSS) via local search} \label{GlauberDynamics}
\subsection{Experimental results: The Smallest Satisfying Set estimator via local search.} 
In our previous results we provided evidence for the absence of OGP when $n \geq (1+\epsilon)k \log_2 \frac{p}{k}$ for every $\epsilon>0$ and showed that for some reasonable values of $\epsilon>0$ the landscape is smooth enough that greedy local search methods work. Naturally, one would like to verify that the OGP prediction is correct and that local search methods can successfully $(1-o(1))$-approximately recover for any $\epsilon>0$. In our last set of results, we fix the Bernoulli group testing design with $\nu$ satisfying $(1-\frac{\nu}{k})^k=\frac{1}{2}$ and we investigate \textit{experimentally} to what extent approximate recovery is indeed amenable to local search.  We discuss our findings here and, more extensively, in Appendix~\ref{simulations}. The key takeaways are the following:
\begin{enumerate}
\item A simple local-search (MCMC) algorithm we propose  is \emph{almost always} successful in solving the optimization task of interest (``minimize unexplained positive tests") to exact optimality when given as input at least $ (1+\epsilon) \lfloor \log_2 {p \choose k} \rfloor$ tests for \textit{any} $\epsilon>0$. (Recall Theorem~\ref{approx_phase}.) 

\item  Based on this observation, we propose an approach for solving the so-called \emph{Smallest Satisfying Set} (SSS) problem via local search. Solving the SSS  problem is a theoretically optimal method  both for exact and  approximate inference, which does \emph{not} require prior knowledge of the parameter $k$. It is typically approached in practice by being modeled as an Integer Program and solved by Branch and Bound methods  and Linear Programming Relaxation techniques, as it is  equivalent to solving  a certain random family of instances of the NP-complete Set Cover problem. (The reader is referred to Chapter 2, Section 2 in~\cite{survey} for more details on the SSS problem.)
Our experiments suggest that the random family of Set Cover instances induced by the non-adaptive Bernoulli group testing problem might in fact be tractable, and this may be of practical interest.

\end{enumerate}

To describe the algorithms we implemented we need the following  definition.

\begin{definition}
A set  $  S \subseteq \{1,2, \ldots, p \} $ is called \emph{satisfying} if:
\begin{enumerate}[(a)]
\item every positive test contains at least one item from $S$;
\item no negative test contains any item from $S$.
\end{enumerate}
\end{definition}
In other words, a set is satisfying if it explains every positive test and none of its elements participates in a negative test.  Clearly, the set of defective items is a satisfying set.

Recalling now Lemma~\ref{key_approx_lemma}, we  see that solving the $k$-Satisfying Set ($k$-SS$)$ problem , i.e.,   the problem of finding a satisfying set of size $k$, guarantees $(1-o(1))$-approximate recovery. (In fact, Lemma~\ref{key_approx_lemma} implies that even a near-optimal solution suffices, but we will not need this extra property for the purposes of this section.)  Based on this observation, we now describe a simple algorithm for solving the $k$-SS problem which is essentially the well-known \emph{Glauber Dynamics} Markov Chain Monte Carlo algorithm.

Glauber Dynamics is a simple local Markov chain that was originally used from statistical physicists to simulate the so-called Ising model (see e.g.~\cite{dobrushin1987completely}). More generally though, Glauber Dynamics is an algorithm designed for sampling from distributions with exponentially large support which has received a lot of attention due to its simplicity and wide applicability, see e.g.~\cite{levin2017markov}. Here we use it for optimization purposes, with the intention to exploit the fact that its stationary distribution assigns the bulk of its probability mass to nearly-optimal states.

Specifically, given a Bernoulli group testing instance, let $\mathrm{PD}$ denote the set of potentially defective items, which recall that is the set of items that do not participate in any negative test. Recall also $\Omega_{k}$ is the \emph{state space} of our algorithm, i.e., the set of all possible subsets of exactly $k$ potentially defective items which  do not belong in any negative test. Finally, for a set ($k$-tuple) $\sigma \in \Omega_{k}$ let  $P(\sigma)$ denote the number of  positive tests explained by $\sigma$. To solve the $k$-SS problem we will use the following simple local algorithm.

\begin{algorithm}
\begin{algorithmic}[1] 
\Procedure{Glauber Dynamics}{$\beta,k,\mathrm{PD}, \mathrm{PosTests}$}

\State  $\sigma \leftarrow$ a random state from $\Omega_{k}$

\State $N \leftarrow  \mathrm{ceil}( 20 |\mathrm{PD} | \ln  |\mathrm{PD} |) $

\For{$i = 1$ to $N$}
\State  Pick an item  $i \in \sigma$ uniformly at random.
\State Pick an item in $j \in \mathrm{PD} \setminus \sigma$ uniformly at random.
\State Let $\tau = (\sigma \cup \{j\} ) \setminus \{i \} $
\State Move to $\tau$  (i.e., $\sigma := \tau$) with probability $ \frac{ \mathrm{e}^{\beta P(\tau)  }  } { \mathrm{e}^{\beta P(\tau)  }  + \mathrm{e}^{\beta P(\sigma)  }  } $ 
\If {$P(\sigma) = |\mathrm{PosTests} |$    }          \Comment We found a $k$-SS
\State \Return $\sigma$
\EndIf
\EndFor
\Return  $\sigma$

\EndProcedure
\end{algorithmic}
\end{algorithm}

\begin{remark}
Note that $\beta$, sometimes called the inverse temperature, is a parameter to be chosen by the user.  In our experiments we choose  $\beta =  5$, but we have noticed that the proposed algorithm is actually quite robust with respect to  the choice of $\beta$.
Note also that we always run our algorithms for at most $ N  =  20 |\mathrm{PD} |  \ln  |\mathrm{PD} |  = O( p \ln p) $ steps, i.e, for a nearly linear number of steps.
\end{remark}

The main outcome of the experiments  we conducted with Glauber Dynamics is that we had an almost perfect (nearly probability one)  success rate with respect to  solving the $k$-SS problem to exact optimality given at least $(1+\epsilon)\log_2 {p \choose k } $ tests as input (see Figure~\ref{Optimization}) for any $\epsilon \geq 0$ we tried. (We emphasize that solving the $k$-SS problem might not imply success with respect to estimating the defective set. Indeed, the solution to $k$-SS implies successful recovery only asymptotically.) The reader is referred to Appendix~\ref{simulations} for more details.

This outcome certainly supports the OGP prediction which, in fact, only suggest that one can solve the problem to near-optimality, since it shows that it is solvable to exact optimality. We also check the performance of the algorithm in terms of exact/approximate recovery. While, as we discussed from a theory standpoint, approximate recovery should be guaranteed for any optimal solution (any satisfying set), in our experiment we do observe a significant but not perfect success in terms of approximate recovery. We believe this is related to the naturally bounded values of $k,p$ we consider (in the order of thousands), since we also observe that the success in the recovery task get increasingly better as we increase $k,p$.  

Motivated by our success in solving to optimality the $k$-SS problem, and as an attempt to check our success without the knowledge of the value of $k$, we investigate next the performance of a local search approach for solving the SSS problem which we describe below.

As the name suggests, the SSS problem amounts to finding the smallest satisfying set \textit{without assuming prior knowledge of $k$}. It is based on the idea that the set of defective items is a satisfying set, but since defectivity is rare, the latter is likely to be small in size compared to the rest satisfying sets. More formally, we have the following corollary of Lemma~\ref{key_approx_lemma} which implies that solving the SSS is a theoretically optimal approach for approximate recovery. 

\begin{corollary}\label{SSS_cool}
Let $k, p \in \mathbb{N}$ with $1 \le k \le p$. We assume that $k,p \rightarrow +\infty$ with $k = o(p)$.  Fix parameters $ \eta,\gamma \in (0,1)$. Assume we observe at least $n = (1+\eta) k \log \frac{p}{k}  $ tests  under  non-adaptive Bernoulli group testing in which each item participates in a test with probability $\nu/k$, where $\nu$ satisfies $(1-\frac{\nu}{k })^{\nu} = \frac{1}{2}$. The smallest  satisfying set contains at least $\lfloor(1-\gamma) k \rfloor$ defective items a.a.s. as $p \rightarrow +\infty.$
\end{corollary}
\begin{proof}
We know that the smallest satisfying set is of size at most $k$, since the set of defective items is satisfying and has size $k$. Applying Lemma~\ref{key_approx_lemma} with parameters $\delta = \gamma$ and  $\epsilon = \eta$, we also know that every other satisfying set of size $k$ contains at least $\lfloor(1- \gamma) k\rfloor$ defective items. So assume that the smallest satisfying set $\sigma$ has $k' < k$ items, and that it contains less than $\lfloor (1-\gamma)k \rfloor$  defective items. Observe now that we can form a new  satisfying set $\sigma'$ by adding $k-k'$ non-defective items to $\sigma$. By construction, $\sigma'$ is a satisfying set of size $k$ that contains less than $\lfloor (1-\gamma)k \rfloor$ items, contradicting Lemma~\ref{key_approx_lemma}, and thus, concluding the proof of the corollary.
\end{proof}

\begin{remark}\label{set_cover}
The SSS problem for a given group testing instance  is equivalent to the following Set Cover instance: Consider an element for each positive test, and a set for each potentially defective item containing every positive test (element) in which it participates.  Finding the smallest satisfying set amounts to finding the smallest in size set that covers every element, i.e., solving the corresponding Set Cover problem.
\end{remark}

Given the Glauber Dynamics algorithm, our approach to solving the SSS is straightforward.  We always start by eliminating every item that is contained in a negative set, so that we are left with the set of potentially defective items. We then use the Glauber Dynamics algorithm as an oracle for solving the $k'$-SS problem for every  $k' \in \{1, \ldots, p\}$ and proceed by the standard technique for reducing optimization problems to feasibility problems via binary search.

In Figure~\ref{Comparison} we present results from simulations suggesting that our approach for solving the SSS problems, besides theoretically optimally in terms of recovery, is efficient and furthermore is able to outperform the other known popular algorithms.

%for approximate inference of the set of defective items (including DD, COMP among others) when $n \geq (1+\epsilon)k \log_2(\frac{p}{k})$ for any $\epsilon>0.$ The reader is referred to Appendix~\ref{simulations} for more details.

\begin{figure}
 \centering
 \begin{minipage}[b]{0.4\textwidth}
 \includegraphics[width=80mm]{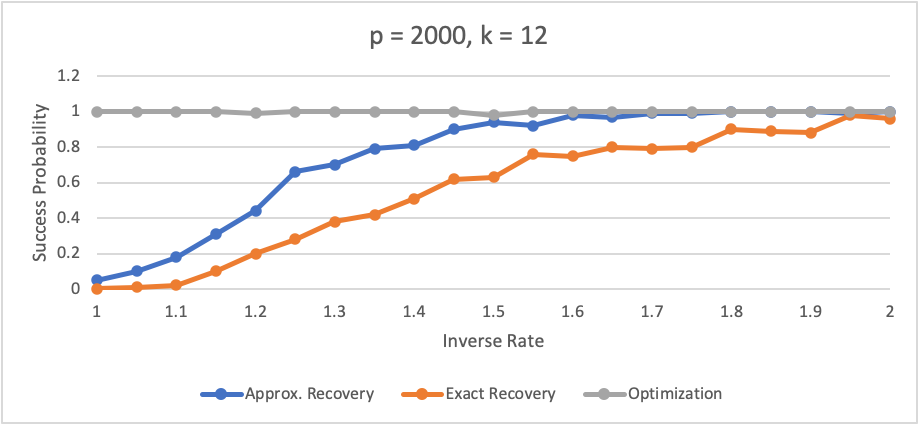}
 \end{minipage}
 \hfill
 \begin{minipage}[b]{0.4\textwidth}
 \includegraphics[width=80mm]{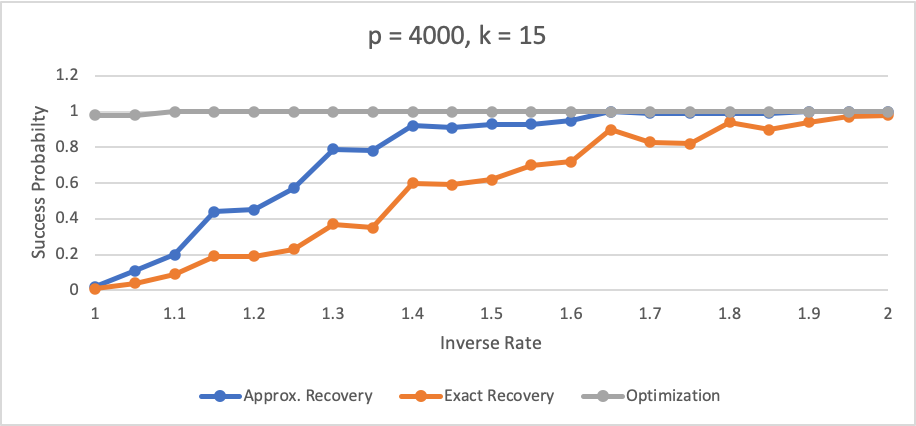}
 \end{minipage}
\hfill
\begin{minipage}[b]{0.4\textwidth}
 \includegraphics[width=80mm]{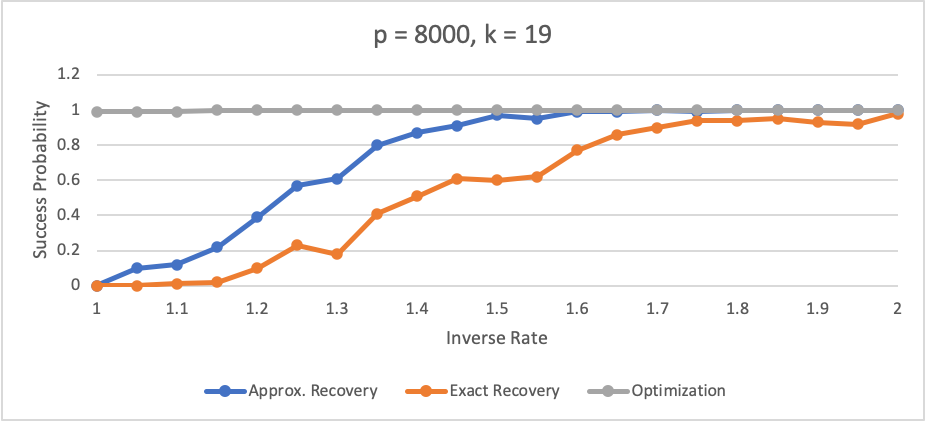}
 \end{minipage}
\caption{ Performance of Glauber Dynamics in solving the $k$-SS problem. The gray  line shows how often the proposed algorithm is able to solve the $k$-SS problem, the orange line represents the probability of exact recovery, while  the blue line  the probability of 90\%-approximate recovery. In  our experiments $k = \lfloor p^{1/3} \rfloor$ and $\nu$ satisfies $(1 - \nu/k)^{ k} = 1/2$ .}
\label{Optimization}
\end{figure}

%
%
%\begin{figure}
%\centering
%%\subfloat{
%%  \includegraphics[width=100mm]{Figures/1000.png}
%%}
%%\subfloat{
% \includegraphics[width=80mm]{Figures/2000Excell.png}
%%}
%\hspace{0mm}
%%\subfloat{
% \includegraphics[width=80mm]{Figures/4000Excell.png}
%%}
%%\subfloat{
% \includegraphics[width=80mm]{Figures/8000Excell.png}
%%}
%%\hspace{0mm}
%%\subfloat[]{   % ???
%%  \includegraphics[width=65mm]{Figures/1000.png}
%%}
%%\subfloat[fifth]{
%%  \includegraphics[width=65mm]{Figures/1000.png}
%%}
%%\caption{Performance of the SSS algorithm implemented via local search. The gray  line shows how often the proposed algorithm is able to solve the SSS problem, the orange line represents the probability of exact recovery, while  the blue line  the probability of 90\%-approximate recovery. In  our experiments $k = \lfloor p^{1/3} \rfloor$ and $\nu$ satisfies $(1 - \nu/k)^{ k} = 1/2$ .}
%%\label{Optimization}
%\caption{ Performance of Glauber Dynamics in solving the $k$-SS problem. The gray  line shows how often the proposed algorithm is able to solve the $k$-SS problem, the orange line represents the probability of exact recovery, while  the blue line  the probability of 90\%-approximate recovery. In  our experiments $k = \lfloor p^{1/3} \rfloor$ and $\nu$ satisfies $(1 - \nu/k)^{ k} = 1/2$ .}
%\label{Optimization}
%\end{figure}

\begin{figure}
 \centering
 \begin{minipage}[b]{0.4\textwidth}
 \includegraphics[width=80mm]{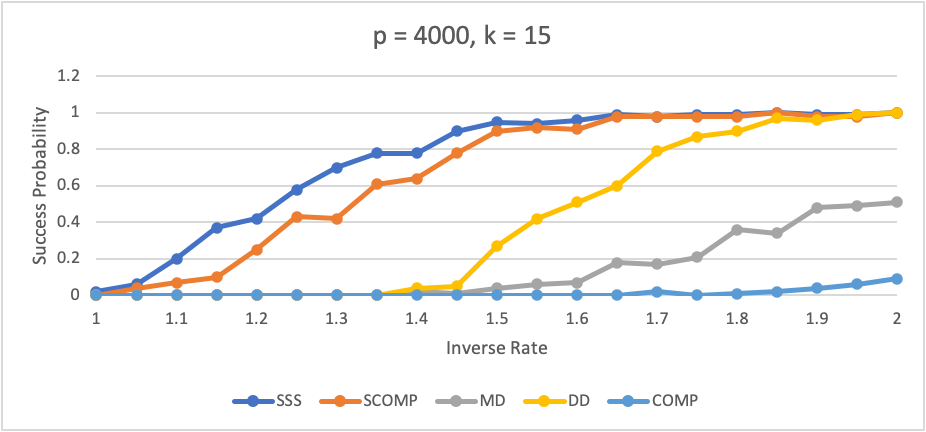}
 \end{minipage}
 \hfill
 \begin{minipage}[b]{0.4\textwidth}
 \includegraphics[width=80mm]{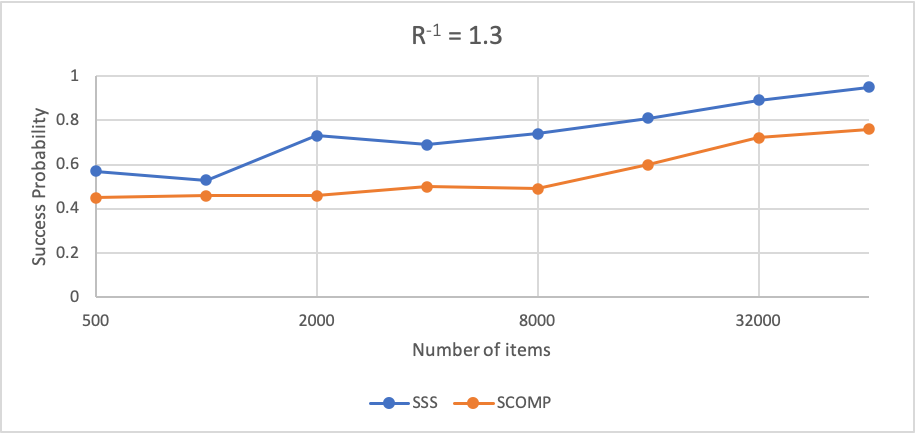}
 \end{minipage}
  \caption{Left: Comparison of the  performances of the SSS (implemented via local search),  SCOMP, COMP, DD and MD algorithms in the   $90\%$-approximate recovery task. Here $k = \lfloor p^{1/3} \rfloor$ and $\nu$ satisfies $(1 - \nu/k)^{k} = 1/2$. 
Right: Comparison of the performances of the SSS (implemented via local search) and SCOMP algorithm in the $90\%$-approximate recovery task as the number of items $p$ increases,  $k = \lfloor p^{1/3} \rfloor$, $\nu$ satisfies $(1- \nu/k)^{k} = 1/2$,  and the rate is fixed at $R = (1.3)^{-1}$.  Note that the horizontal axis is in logarithmic scale.}
    \label{Comparison}
\end{figure}

%\begin{figure}
%\centering
%%\subfloat{
%%  \includegraphics[width=100mm]{Figures/1000.png}
%%}
%%\subfloat{
% \includegraphics[width=80mm]{Figures/4000comparison.png}
%%}
%\hspace{0mm}
%%\subfloat{
% \includegraphics[width=80mm]{Figures/vsSCOMPexcell.png}
%%}
%%\hspace{0mm}
%%\subfloat[]{   % ???
%%  \includegraphics[width=65mm]{Figures/1000.png}
%%}
%%\subfloat[fifth]{
%%  \includegraphics[width=65mm]{Figures/1000.png}
%%}
%\caption{Left: Comparison of the  performances of the SSS (implemented via local search),  SCOMP, COMP, DD and MD algorithms in the   $90\%$-approximate recovery task. Here $k = \lfloor p^{1/3} \rfloor$ and $\nu$ satisfies $(1 - \nu/k)^{k} = 1/2$. 
%Right: Comparison of the performances of the SSS (implemented via local search) and SCOMP algorithm in the $90\%$-approximate recovery task as the number of items $p$ increases,  $k = \lfloor p^{1/3} \rfloor$, $\nu$ satisfies $(1- \nu/k)^{k} = 1/2$,  and the rate is fixed at $R = (1.3)^{-1}$.  Note that the horizontal axis is in logarithmic scale.} 
%\label{Comparison}
%\end{figure}
%
%

\newpage
\bibliographystyle{plain}
\bibliography{group}

\appendix

\section{Organization of the appendices} 

The appendices are organized as follows. In Appendix~\ref{formal_statement} we formally statement our theoretical results, Theorems~\ref{first_moment_calc},~\ref{monotonicity_theorem} and~\ref{local_search_theorem}.   In Appendix~\ref{simulations} we  present our experimental results.  In Appendix~\ref{OGP_proof} we prove Theorem~\ref{first_moment_calc}, Theorem~\ref{monotonicity_theorem} and Theorem~\ref{thm:OGP}.  In   Appendix~\ref{local_search_proof} we prove Theorem~\ref{local_search_theorem}.  Several technical proofs are deferred to Appendices~\ref{key_approx_lemma_proof},~\ref{omitted_formal},~\ref{omitted_local},~\ref{CorProof}.

\section{Formal statement of theoretical results}\label{formal_statement}

In this section we formally present our theoretical results.

\subsection{The Model}
We start with formally defining the inference model of interest. There are $p$ items, $k$ out of them are defective and chosen uniformly at random among the total $p$ items. We denote by $\theta^* \in \{0,1\}^p$ the indicator vector of the defective items. We perform $n$ tests, where at each test each  item is independently included in the test with probability $\nu/k$, where $\nu>0$ is some positive constant. In most cases we consider $\nu$ satisfying $(1- \frac{\nu}{k})^{ k} =  \frac{1}{2}$, as informed by the first part of Theorem~\ref{approx_phase}, but in some cases we choose some other more suitable value of $\nu$ in which case we clearly state it. For $i=1,\ldots,n$ we call $X_i \in \{0,1\}^p$ the indicator vector corresponding to the  subset of the items chosen  for the $i$-th test and $Y_i=\mathbbm{1}(\langle X_i, \theta^* \rangle \ge 1)$ the binary outcome of the $i$-th test. We also define the test matrix $X \in\{0,1\}^{n \times p}$ the matrix with rows $X_i, i=1,\ldots,n$ and the test outcomes vector $Y \in \{0,1\}^n$ the vector with elements $Y_i, i=1,\ldots,n$. Finally, we define by \begin{equation} \label{eq:rate} R = \frac{ \log_2 { p \choose k }  }{ n } \end{equation} a quantity we call the \emph{rate}. With respect to scaling purposes we assume that as $p \rightarrow +\infty$, $k,n \rightarrow +\infty $ under the restrictions that $R$ is kept equal to  a fixed constant and that $k=o(p)$.

The Bayesian task of interest is to asymptotically approximately recover the vector $\theta^*$ given access to $(Y,X),$ i.e., to construct an estimator $\hat{\theta}=\hat{\theta}(Y,X) \in \{0,1\}^p$ with \begin{equation}\label{approx_rec}
d_H(\hat{\theta},\theta^*)=o(k),
\end{equation}asymptotically almost surely (a.a.s.), with respect to the randomness of the prior on $\theta^*$ and the test matrix $X$, as $p \rightarrow +\infty.$  Here $d_H( \cdot, \cdot) $ denotes the Hamming distance between two binary vectors.  Note also that achieving~\eqref{approx_rec} with an estimator of sparsity $k$ implies $(1-o(1))$-approximate recovery in the sense of Definition~\ref{approx_rec_def}. Importantly, for our theoretical results, we work under the assumption that $k$ is known to the statistician.

As explained in the Introduction (Theorem \ref{approx_phase}), using the introduced notation of the rate, the recovery task of interest is known to be \textit{information-theoretic impossible} if $R>1$ and \textit{information-theoretic possible} if $R<1$. Specifically, if $R<1$ one can use the Bernoulli design with $\nu$ satisfying $(1-\frac{\nu}{k})^k=\frac{1}{2}$ and then use exhaustive search to find a satisfying set. Furthermore if $R<\frac{\nu e^{-\nu}}{\ln 2}$ for some $\nu>0$ one can recover by using the Bernoulli design with parameter $\nu>0$ and then run the polynomial-time algorithm COMP which simply outputs items that never participated in any negative test. (Recall Remark~\ref{Pink}.) Naturally, we would like to know  for which rates $\frac{\nu e^{-\nu}}{\ln 2}<R<1$ we can recover approximately (in the sense of \eqref{approx_rec}) using a computationally efficient estimator, the Bernoulli design with parameter $\nu$ — and in particular how close we can get to $R=1$.
  Of particular interest is, of course, the case where $\nu$ satisfies $(1-\frac{\nu}{k})^k=\frac{1}{2}$ where approximate recovery in the sense of \eqref{approx_rec} is known to be information-theoretically possible for all $R<1.$

% \textcolor{blue}{IZ: dont get why to assume generally that R is bigger than $\frac{\nu e^{-\nu}}{2}$.} \red{FI: Currently also needed for the concentration results on the number of potentially defective items. Please don't change it fow now, "slow the thing" as they say in my village. } 

% in $[\frac{ \nu \mathrm{e}^{-\nu} }{ \ln 2 } ,1)$. 

% (Note that we may assume without loss of generality  that $R \ge  \frac{ \nu \mathrm{e}^{-\nu} }{ \ln 2 } $. Indeed, if $R <  \frac{ \nu \mathrm{e}^{-\nu} }{ \ln 2 }$, then we can  utilize  only the first  $ \lfloor\log_2 { p \choose k}/ (\frac{ \nu \mathrm{e}^{-\nu} }{ \ln 2 }) \rfloor$ tests and ignore the rest.)

\subsection{Absence of the Overlap Gap Property}\label{absence_ogp}

As we have already mentioned, our first and main theorem provides first moment evidence for the absence of the Overlap Gap Property (OGP) asymptotically up to the information theoretic threshold $R=1$ for the task of $(1-o(1) )$-approximate recovery. We propose this evidence as a potentially fundamental reason for computational tractability in this regime. For this section, and towards this goal, we fix the Bernoulli design for parameter $\nu>0$ with $(1-\frac{\nu}{k})^k=\frac{1}{2}$.  Recall that the optimal estimator in this Bernoulli design, which works for all $R<1$, corresponds to finding an (approximately) satisfying set.

Towards understanding the computational difficulty of finding an (approximately)  satisfying set in this context we first apply, as a pre-processing step, the COMP algorithm which is known to achieve  $(1-o(1))$-approximate recovery for all $R< \frac{1}{2}$ and for the value of $\nu$ we chose. 
Indeed, we know this is true for $R<\frac{ \nu \mathrm{e}^{-\nu} }{ \ln 2 }$ and then notice that as $k \rightarrow +\infty, \nu=(1+o(1))\ln 2$ and therefore it holds $\frac{ \nu \mathrm{e}^{-\nu} }{ \ln 2 }=\frac{1+o(1)}{2}.$
For this reason when $R<\frac{1}{2}$ the recovery goal is trivially achieved by this polynomial-time preprocessing step. Hence, in what follows we assume $R$ to satisfy $\frac{1}{2}<R<1$ (we ignore the case where $R=\frac{1}{2}$ to avoid any unnecessary criticality issues, and focus solely on higher rates than $\frac{1}{2}$).

Now COMP simply removes from consideration any item that takes part in a negative test. We call any item that is  \emph{not} removed from consideration after this step as \emph{potentially defective}. The following two lemmas estimate the number of positive tests
and the number  of potentially defective items.  The proofs  of Lemmas~\ref{number_positive},~\ref{pi_prime_estimation} can be found in Appendix~\ref{omitted_formal}. (Note that these are well-known statements that have been implicitly proven in previous works, but we chose to include their proof in this paper as well for completeness.)

\begin{lemma}\label{number_positive}
Let $\mathcal{P} \subseteq \{1, \ldots, n \}$ denote the set of indices of the positive tests. For every constant $\eta \in (0,1)$:
\begin{align}\label{positive_bounds}
(1-\eta) \frac{n}{2}    \le  |\mathcal{P} | \le (1+\eta) \frac{n}{2},
\end{align}
asymptotically almost surely as $p \rightarrow +\infty$.
\end{lemma}

\begin{lemma}\label{pi_prime_estimation}
Suppose $\frac{1}{2}<R<1$. Let $\mathrm{PD}$ denote the set of potentially defective items. For every constant $\eta \in (0,1)$:
\begin{align}
(1-\eta)  p  \left(  \frac{k}{p}	\right)^{ \frac{1+ \eta}{ 2R}  }       \le | \mathrm{PD} |  \le  (1+\eta) p  \left( \frac{k}{p} \right)^{  \frac{ 1-\eta }  {2R }   } \label{piprime_bounds} 
\end{align}
asymptotically almost surely as $p \rightarrow +\infty$.
\end{lemma}
\noindent

\begin{remark}\label{conditioning_remark}
In the light of Lemmas~\ref{number_positive},~\ref{pi_prime_estimation}, in what follows we condition on the values of the random variables $\mathcal{P}$ and   $\mathrm{PD} $ and we  assume  that their cardinalities, which we denote by $n_{\mathrm{pos} }$ and $p'$ respectively, satisfy~\eqref{positive_bounds} and~\eqref{piprime_bounds}, respectively for a sufficiently small constant $\eta$ of interest. We also slightly abuse the notation and denote by $\theta^* \in \{0,1\}^{p'}$ the indicator vector of the defective items, by $X_i \in \{0,1 \}^{p'}$  the row vector corresponding  to the positive test with index $i \in \mathcal{P}$,  whose non-zero entries indicate the potentially defective items that participate in this test, and we switch to indexing with $p'$.
\end{remark}

Note that after the execution of COMP we have in consideration $n_{\mathrm{pos}}$ positive tests and $p'$ potentially defective items. At this point, and recalling Lemma~\ref{key_approx_lemma}, the task of finding an (approximately)  satisfying set of cardinality $k$ corresponds to finding a set of $k$ items out of the $p'$ which explain all but a vanishing fraction of the positive tests (note that here we used that the number of  positive tests is of the same order as the number of total number of tests a.a.s. — Lemma \ref{positive_bounds}).   Given a $k$-sparse $\theta \in \{0,1\}^{p'}$ we say that a positive test $i \in \mathcal{P}$ is \emph{unexplained} by $\theta$ if  $ \langle X_i, \theta \rangle   = 0$, and by $H(\theta)$ we denote the number of unexplained tests with respect to $\theta$, i.e.
\begin{align*}
H(\theta):=|\{ i \in \mathcal{P}: \langle X_i, \theta \rangle   = 0\}|.
\end{align*}
We sometimes call $H(\theta)$ the \emph{energy} at $\theta$, motivated by the statistical physics context where OGP originates from.
Notice that $\theta$ corresponds to an (approximately) satisfying set if and only if it holds ($H(\theta)=o(n)$) $H(\theta)=0$. Using this notation, to find an (approximately) satisfying set it suffices to solve to near-optimality the following optimization problem

\begin{align*}
\begin{array}{clc} \left(\Phi\right) & \min  & H(\theta) \\ &\text{s.t.}&\theta \in \{0,1\}^{p'}  \\
&& \|\theta\|_0=k.
\end{array}
\end{align*}
We make two observations. First the optimal value of $(\Phi)$ is clearly zero and achieved by $\theta^*$. Second, again by Lemma~\ref{key_approx_lemma}, the level of near-optimality which suffices for approximate recovery corresponds to the $\theta$ with $H(\theta)=o(n)$.

As explained in the Introduction, inspired by statistical physics, and the successful prediction of the computational thresholds at least in the context of the sparse regression~\cite{gamarnik2017high} and of the planted clique problems~\cite{gamarnik2019landscape}, we study the presense/absence of OGP in the landscape of $(\Phi)$ to offer a heuristic understanding of the rates in $R \in (\frac{1}{2},1)$ for which the optimization problem can be solved to near-optimality, or equivalently $\theta^*$ can be approximately recovered via local methods. As explained also in the Introduction, we provide evidence that the OGP never appears for any $R \in (\frac{1}{2},1)$.

The OGP informally says that the space of near-optimal solution of $(\Phi)$ separate into two disjoint clusters, one corresponding to $\theta$ which are ``close" to $\theta^*$ (the high overlap cluster) and one corresponding to $\theta$ which are ``far" to $\theta^*$ (the low overlap cluster). We now formally define it for any fixed $p',k.$ 
\begin{definition}[$(\zeta_{p'},W_{p'},H_{p'})$-OGP]\label{def:OGP} Fix some $W_{p'} \in \{0,1,\ldots,k-1\},$ $\zeta_{p'} \in \{ \lfloor \frac{k^2}{p'} \rfloor ,\lfloor \frac{k^2}{p'} \rfloor+1,\ldots,k-W_{p'}-1\}$ and some $H_p'>0$. 
 We say that    $(\Phi)$ exhibits the $(\zeta_{p'},W_{p'},D_{p'})$-Overlap Gap Property 
($(\zeta_{p'},W_{p'},H_{p'})$-OGP) if there exists a threshold value $r_{p'} \in \mathbb{R}$ satisfying the following properties.
\begin{itemize}
\item[(1)] For any $\theta \in \{0,1\}^{p'}$ with $\|\theta\|_0=k$ and $H(\theta)  \leq r_{p'}$, it holds that either $\langle \theta,\theta^* \rangle \leq \zeta_{p'}$ or $\langle \theta,\theta^* \rangle \geq \zeta_{p'}+W_{p'}+1.$
\item[(2)] There exist $\theta_1,\theta_2 \in \{0,1\}^{p'}$ such that  $\|\theta_1\|_0=\|\theta_2\|_0=k$, $\langle \theta_1,\theta \rangle \leq \zeta_{p'},$ $\langle \theta_2,\theta \rangle \geq \zeta_{p'}+W_{p'}+1$ and  $\min \{ H(\theta_1),  H(\theta_2) \} \leq r_{p'}$;
\item[(3)] It holds $\max_{\theta \in \{0,1\}^{p'}: \|\theta\|_0=k, \zeta_{p'}+1 \leq \langle \theta,\theta^* \rangle \leq \zeta_{p'}+W_{p'}-1} H(\theta) \geq r_{p'}+H_{p'}.$
\end{itemize} 
\end{definition}

Let us provide some intuition on the definition of the OGP, as it slightly generalizes the previous definitions used for the Overlap Gap Property in the literature~\cite{gamarnik2017high,gamarnik2019landscape,gamarnik2019sparse}, since we define it adjusted to the task of \textit{approximate recovery}. The parameter $W_{p'}$ corresponds to \textit{the width} of the OGP and the parameter $H_{p'}$ to \textit{the height} of the OGP. The first condition says that all solutions $\theta$ achieving energy $H(\theta)$ less than some threshold value $r_{p'}$ must either have dot product (overlap) with $\theta^*$ less than $\zeta_{p'}$ or bigger than $\zeta_{p'}+W_{p'}.$ This gives rise to two clusters of near-optimal solutions (in the sense of achieving energy at most $r_n$) whose overlap with $\theta^*$ differs by at least $W_{p'}.$ The second condition makes sure that the two clusters are non-empty. Now, the third condition implies that there is some energy level achieved by some $\theta$ with overlap with $\theta^*$ strictly between $\zeta_{p'}$ and $\zeta_{p'}+W_{p'}$  and energy higher by at least the height $H_{p'}>0,$ as compared to all energies achieved by solutions in the two clusters. Finally, notice that the requirement $\zeta_{p'} \geq \lfloor \frac{k^2}{p'} \rfloor$ holds, as one can always achieve overlap with $\theta^*$ of the order $(1+o(1))\lfloor \frac{k^2}{p'} \rfloor$, by simply choosing a binary $k$-sparse vector at random, and therefore an OGP for smaller overlap sizes than it is not relevant for recovery, but it is definitely relevant for all overlaps bigger than $\lfloor \frac{k^2}{p'} \rfloor.$

A \textit{highly informal}, yet surprising accurate in certain contexts, computational prediction is that local search algorithms attempting to solve $(\Phi)$ are able to find in polynomial in $p'$ time a $\theta$ with $\langle \theta,\theta^* \rangle > \zeta_{p'}$ if and only if for all width levels $W_{p'} \in \{1,\ldots,k-\zeta_{p'}-1\}$ with $W_{p'}=\omega(1)$ and height levels $H_{p'}>0$ with  $H_{p'}=\omega(\log p')$ the $(\Phi)$ does not exhibit the $(\zeta_{p'},W_{p'},H_{p'})$-OGP a.a.s. as $p \rightarrow +\infty.$ The heuristic intuition behind this prediction is as follows. If two clusters of near-optimal solutions of $(\Phi)$ are separated by a growing width $W_{p'}=\omega(1)$ and height $H_{p'}=\omega(\log p')$  then no local search algorithm is able to either ``jump across" the clusters by tuning the local search radius at Hamming distance $W_{p'}$ from its current state (as this would take $\binom{k}{W_{p'}}\binom{p'-k}{W_{p'}}=e^{\Omega(W_{p'} \log p')}$-time), or use natural MCMC methods such as Glauber dynamics with finite temperature to ``jump over" the height $H_{p'}$ (as this would normally require $e^{H_{p'}}=e^{\omega(\log p')}$-time). The prediction now is that unless such an OGP appears, an appropriate local search algorithm works.

We now provide evidence that for all $\frac{1}{2}<R<1$ such an OGP indeed does not take place in $(\Phi)$ when $\zeta_{p'} \leq (1-\epsilon)k,$ for arbitrarily small fixed $\epsilon>0.$ Such a result suggests that local search methods can obtain overlap $(1-\epsilon)k$,   i.e. achieve $(1-\epsilon)$-approximate recovery, in polynomial-time. We first define the following restricted optimization problems.

\begin{definition}\label{phidef}
For $\ell=0,1,2,\ldots,k$, let $\phi(\ell)$ denote the optimal value of the optimization problem
\begin{align*}\begin{array}{clc} 
\left(\Phi\left(\ell\right)\right) & \min  &H(\theta) \\ &\text{s.t.}&\theta \in \{0,1\}^{p'}  \\
&& ||\theta||_0=k, \langle \theta,\theta^* \rangle =\ell.
\end{array}
\end{align*} 
\end{definition}
Note that $\Phi(\ell)$ is simply $\Phi$ constrained on only $\theta$ satisfying $\langle \theta,\theta^* \rangle =\ell.$  Of course $\theta^*$ is not known to the statistician, and $\Phi(\ell)$ are considered solely for analysis purposes. Trivially, as $\ell$ spans all possible values of $\langle \theta,\theta^* \rangle$, $\min_{\ell=0,1,\ldots,k} \phi(\ell)=H(\theta^*)=0.$ Furthermore,  for $R<1$ by Lemma~\ref{key_approx_lemma} it holds $\arg \min_{\ell =0,1,\ldots,k} \phi(\ell)=(1-o(1))k$ a.a.s.

Now we offer a necessary implication of the existence of OGP in terms of the monotonicity of $\phi(\ell)$, which allows us to argue for its abscence in what follows. As mentioned in the introduction, such links between OGP and the monotonicity of $\phi(\ell)$ have appeared in the literature but the exact lemma below is not known, to the best of our knowledge. The proof of Lemma~\ref{lem:monot} can be found in Appendix~\ref{omitted_formal}.
\begin{lemma}\label{lem:monot}
Let  $\zeta, M,W \in \{0,1,\ldots,k\}, H>0$ with $M+W \leq k$.  If $\Phi$ satisfies the $(\zeta,W,H)$-OGP for some $\zeta \leq M$,  then there exists  $u \in \{0,1,\ldots,W-1\}$, such that $\phi(\ell W+u)$ is not non-increasing for $\ell=0,1,2,\ldots,\lfloor M/W \rfloor.$
\end{lemma}

Using  Lemma \ref{lem:monot} we now provide evidence that for any arbitrarily small but fixed $\epsilon \in (0,1)$ $\Phi$ a.a.s. as $p \rightarrow +\infty$ \textbf{does not satisfy} the $(\zeta_{p'},W_{p'},H_{p'})$-OGP for any choice of parameters $\zeta_{p'} \leq M= \lfloor (1-\epsilon)k \rfloor,$ $W_{p'}=\omega(1)$ and $H_{p'}>0$. Notice that the additional stronger condition of $H_{p'}=\omega(\log p')$ is not required for our argument and only the weaker $H_{p'}>0$ suffices. To prove the abscence of such an OGP in light of Lemma \ref{lem:monot} according to which it suffices to study the monotonicity of $\phi(\ell)$ across the different arithmetic progressions with some difference $W_{p'}=\omega(1)$.

Towards understanding the concentration properties of $\phi(\ell)$ we use the moments method, which we have already explained in the Introduction, but which we repeat here for completeness. We define for $t, \ell \ge 0$, the counting random variable
\begin{align*}
Z_{\ell,t}=  \left|  \{ \theta \in \{0,1\}^{p'}: \| \theta \|_0 = k, \langle \theta, \theta^* \rangle = \ell, H(\theta)  \le t \}   \right|
\end{align*}and observe that
$$\phi(\ell) \le t \Leftrightarrow Z_{\ell,t}\geq 1.$$
In particular,  by Markov's inequality and Paley's-Zigmund's inequality we have for all $t>0$ and $\ell \in \{0,1,2,\ldots, \lfloor(1-\epsilon) k \rfloor \}$:
\begin{align*}
\frac{\mathbb{E}\left[Z_{t,\ell}\right]^2}{\mathbb{E}\left[Z^2_{t,\ell}\right]} \leq \Pr\left[\phi(\ell) \le t\right] = \Pr[  Z_{\ell,t } \ge 1  ]  \leq \mathbb{E}\left[Z_{t,\ell}\right].
\end{align*}Hence if for some $t_1,t_2>0$ it holds $\mathbb{E}\left[Z_{t_1,\ell}\right]=o(1)$ we have $\phi(\ell)>t_1$ a.a.s as $p \rightarrow +\infty$ and if $\frac{\mathbb{E}\left[Z_{t_2,\ell}\right]^2}{\mathbb{E}\left[Z^2_{t_2,\ell}\right]} =1-o(1)$ or equivalently $\frac{\mathrm{Var}\left[Z_{t_2,\ell}\right]}{\mathbb{E}\left[Z_{t_2,\ell}\right]^2} =o(1)$ we have $\phi(\ell) \leq t_2$ a.a.s. The employment of the first moment to get a a.a.s. lower bound is called a first moment method, and the employment of the second moment to get an a.a.s. upper bound is called the second moment method.

% In many cases of sparse combinatorial optimization problems, such as $(\Phi(\ell)),$ it has been established in the literature that the first and second moment methods can be proven for $t_1,t_2$ sufficiently close to each other, sometimes satisfying even $t_2 \leq t_1+1,$ a phenomenon known as 2-point concentration. \textcolor{blue}{2-point concentration.} Here we say we provide first moment evidence for the OGP, because we do not check the second moment method to prove the sufficient concentration, but we only use the first moment to derive a prediction for the value on which $\phi(\ell)$ concentrates on. Specifically,

As we have already explained in the Introduction, we define the function which maps $\ell$ to the value $t=t_{\ell}>0$ for which the first moment satisfies $\mathbb{E}\left[Z_{t_1,\ell}\right]=1 $ (\textit{first moment prediction}), which we use as our heuristic approximation for $\phi(\ell)$, which has been a successful approximation in both regression \cite{gamarnik2017high}, the planted clique \cite{gamarnik2019overlap} and sparse PCA models \cite{ArousSparse}. Now using standard large deviation theory of the Binomial random variable we arrive at the following definition of the first moment prediction.

\begin{definition}\label{fpdef} 
Suppose $\frac{1}{2}<R<1$. Then for every fixed value of $p', n_{\mathrm{pos}}$ satisfying Lemma \ref{pi_prime_estimation} and \ref{number_positive} for sufficiently small $\eta>0$ we denote by $F_{p'}(\lambda),\lambda \in [0,1)$ the \textbf{first moment function} defined implicitly by being the unique function satisfying the following two constraints for each $\lambda \in [0,1),$
\begin{itemize}
    \item[(1)] $F_{p'}(\lambda) \leq 1-2^{\lambda-1}$;
    \item[(2)] \begin{align} \label{func_prop}
    \alpha \left(F_{p'}(\lambda),1-2^{\lambda-1}\right)= \frac{\ln \left[\binom{k}{ \lfloor \lambda k \rfloor }\binom{p' -k }{ \lfloor (1-\lambda)k \rfloor}\right]}{n_{\mathrm{pos} }}.
    \end{align}
\end{itemize}
where $\alpha(x,y) := x \ln \frac{x}{y} + (1-x)  \ln \frac{1- x }{1-y } $.
\end{definition}

\begin{remark}
Here $\alpha(x,y)$, $x, y \in [0,1]$,  is the Kullback–Leibler divergence (relative entropy function) between two Bernoulli random variables with probability of success $x$ and $y$, respectively. It naturally appears here using the large deviation properties of the Binomial distributions.
\end{remark}

\begin{remark}
Technically $F_{p'}$ also depends on the value of $n_{\mathrm{pos}}$, and not just $p'$. However, we choose to emphasize only the dependency on the population of potentially defective items —  which is the main (growing) parameter we use to index all quantities in our work—  in order to simplify the notation. 
\end{remark}

The non-voidness of Definition \ref{fpdef} is given in the following Proposition, along with some basic analytic properties of the first moment function. The proof of Proposition~\ref{prop_fm} can be found in Appendix~\ref{omitted_formal}. 

\begin{proposition}\label{prop_fm}
Suppose $\frac{1}{2}<R<1.$ Then the first moment function $F_{p'}(\lambda)$ exists, is unique and for any $\epsilon>0$ it is continuously differentiable for $\lambda \in [0,1-\epsilon],$ a.a.s. as $p \rightarrow +\infty$ (with respect to the randomness of $p', n_{\mathrm{pos}}$ satisfying Lemma \ref{pi_prime_estimation} and \ref{number_positive} for sufficiently small $\eta.$ ).
\end{proposition}

Now we prove rigorously using the first moment method that $\phi(\ell)$ can be lower bounded in terms of the first moment function $F_{p'}(\ell)$.
\begin{theorem}\label{first_moment_calc}
Suppose $\frac{1}{2}<R<1$ and $k = \Theta(p^{\alpha})$ for some constant $\alpha \in (0,1)$.  For every $\epsilon >0$, there exists a constant $C = C(\epsilon)$ such that a.a.s. as $p \rightarrow +\infty$  for every integer $\ell \in \{ 0,1,\ldots,\lfloor (1-\epsilon)k \rfloor \}$ we have:
\begin{align*}
    \phi(\ell) \geq n_{\mathrm{pos} } F_{p'}\left(\frac{\ell}{k}\right)-C \ln k.
\end{align*}
\end{theorem}

Using the second moment method to establish concentration, we conjecture that the Theorem \ref{first_moment_calc} can be strengthened to prove the following result (as explained in the literature, similar results have been proven to be tight in multiple cases where OGP has been studied \cite{gamarnik2017high, balister2019dense, gamarnik2019landscape,ArousSparse}.

\begin{conjecture}\label{conj}
Suppose $\frac{1}{2}<R<1.$  For every $\epsilon >0$, there exists a constant $C = C(\epsilon)$ such that a.a.s. as $p \rightarrow +\infty$  for every integer $\ell \in \{ 0,1,\ldots,\lfloor (1-\epsilon)k \rfloor \}$ we have:
\begin{align*}
   \big| \phi(\ell) - n_{\mathrm{pos} } F_{p'}\left(\frac{\ell}{k}\right)\big| \leq C \ln k.
\end{align*}
\end{conjecture}

\begin{figure}[!tbp]
 \centering
 \begin{minipage}[b]{0.35\textwidth}
    \includegraphics[width=\textwidth]{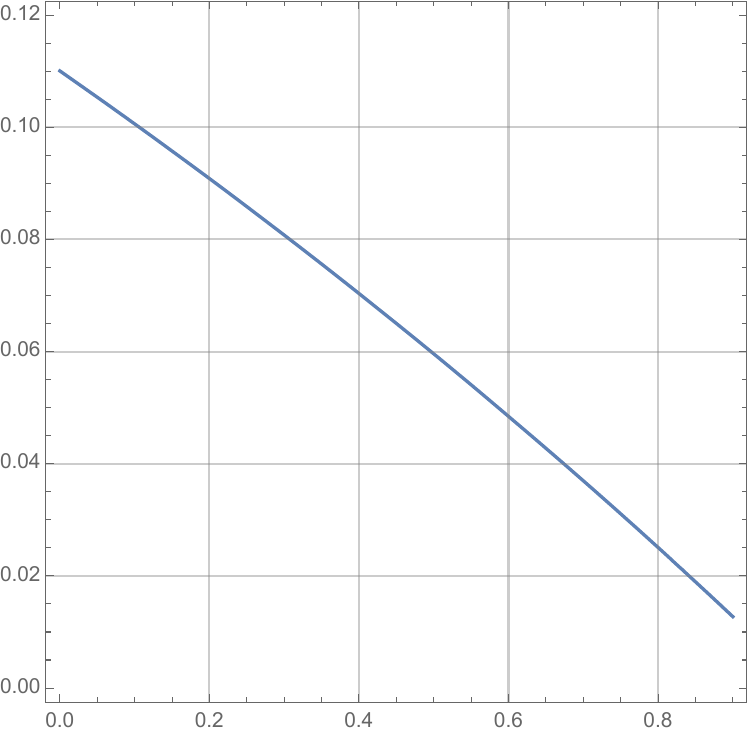}
    \caption{$R=0.75$}
 \end{minipage}
 \hfill
 \begin{minipage}[b]{0.35\textwidth}
    \includegraphics[width=\textwidth]{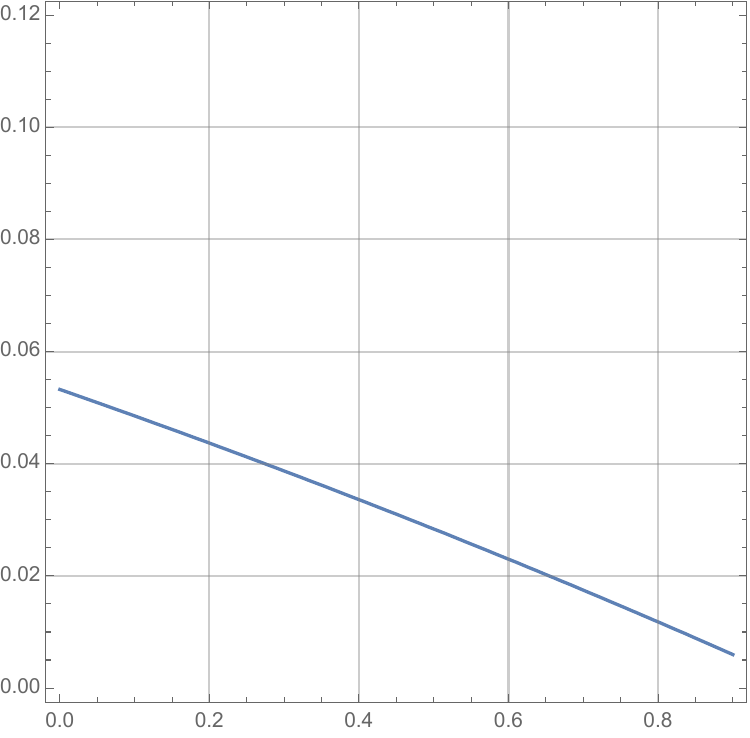}
    \caption{$R=0.85$}
 \end{minipage}
 \hfill
 \begin{minipage}[b]{0.35\textwidth}
    \includegraphics[width=\textwidth]{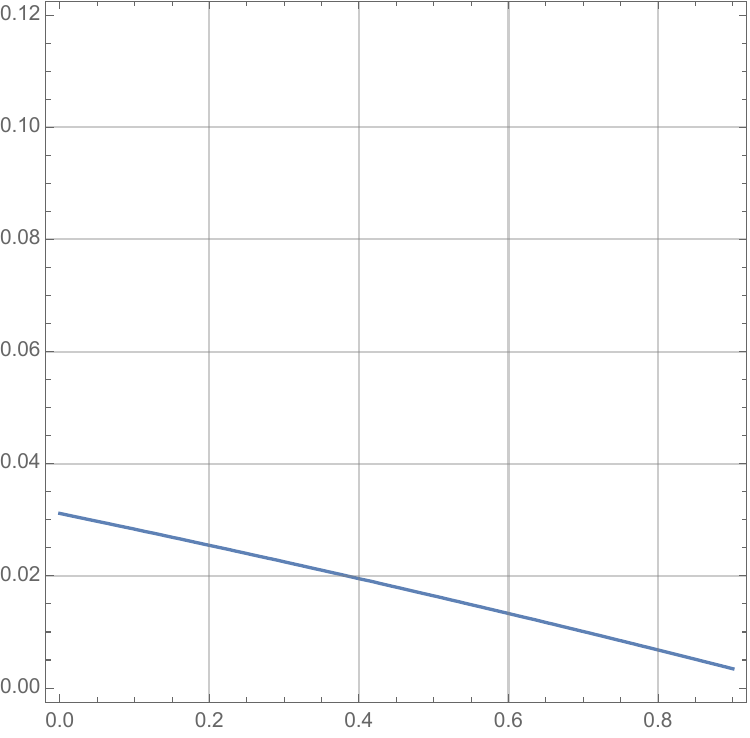}
    \caption{$R=0.9$}
 \end{minipage}
 \hfill
 \begin{minipage}[b]{0.35\textwidth}
    \includegraphics[width=\textwidth]{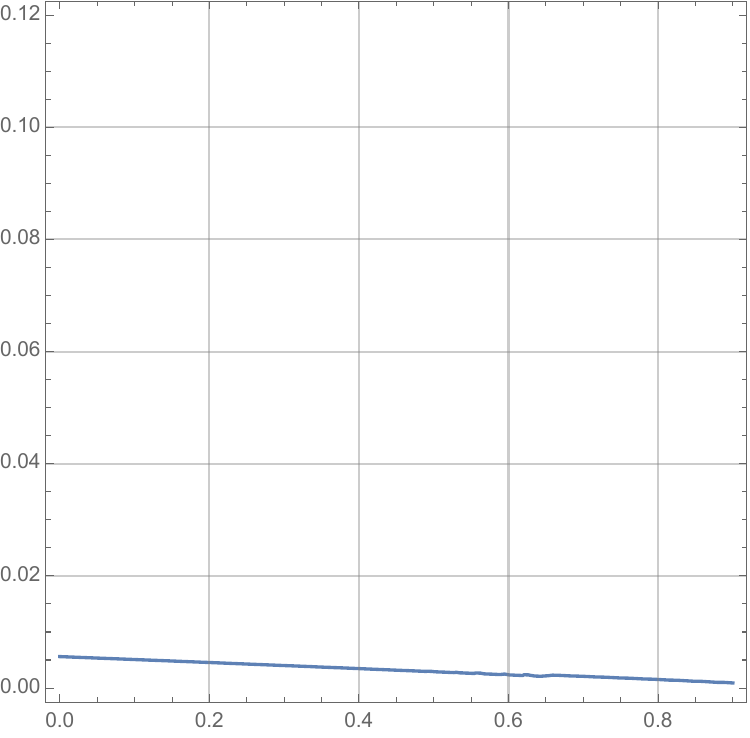}
    \caption{$R=0.975$}
 \end{minipage}
  \caption{ For plotting purposes, we plot a ``first order approximation", $\tilde{F}: [0,0.9] \rightarrow \mathbb{R}$ of the implicitly defined first moment curve
        $F_{p'}(\lambda)$ which is the curve satisfying for all  $\lambda \in [0,0.9]$, $\tilde{F}(\lambda) \leq 1-2^{\lambda-1}$ and $ \alpha \left(\tilde{F}(\lambda),1-2^{\lambda-1}\right)= (2R-1)\ln 2 (1-\lambda) .$   The first order approximation we use is $\ln \left[\binom{k}{ \lfloor \lambda k \rfloor }\binom{p' -k }{ \lfloor (1-\lambda)k \rfloor}\right]/n_{\mathrm{pos} } \approx (2R-1)\ln 2 (1-\lambda) $ 
      and follows from a relatively straightforward derivation using Lemmas \ref{number_positive}, \ref{pi_prime_estimation} and basic asymptotics since $k=o(p)$. One can see that the $\tilde{F}$ remains strictly decreasing throughout the difference choices of $R<1$ (in accordance with Theorem \ref{monotonicity_theorem}) and it converges to the zero function as $R$ tends to $1$.}
    \label{fig:monoton}
\end{figure}

%
%Given Conjecture \ref{conj} and our desire to study the monotonicity of $\phi(\ell)$ it seems natural to study the monotonicity of the implicitly defined, but potentially easier to handle first moment function $F_{p'}(\ell).$  We are able to analyze the monotonicity of the first moment function $F_{p'}(\frac{\ell}{k})$ and prove that it is indeed strictly decreasing as a function of $\ell$ \textit{for any $R$ with $\frac{1}{2}<R<1$.} (See also Figure \ref{fig:monoton}.)
%The result can be cast as elementary but in fact requires the careful use of various multivariate calculus and asymptotics tools to get around the rather complicated definition of the function $F_{p'}(\lambda).$

\begin{theorem}\label{monotonicity_theorem}
Let arbitrary $\frac{1}{2}<R<1$. For every  $\epsilon>0$ there exists  a constant $D = D(\epsilon)>0$ such that a.a.s. as $p \rightarrow +\infty$ for every  $\ell \in \{ 0,\ldots,\lfloor (1-\epsilon)k \rfloor \} $ we have:
\begin{align} F_{p'} \left(\frac{\ell+1}{k}\right)-F_{p'}\left(\frac{\ell}{k}\right) \leq -D\frac{\ln \left( \frac{p'-k}{k} \right)  }{n_{\mathrm{pos} }}.\end{align}In particular, $F_{p'}(\ell)$ is strictly decreasing as a function of $\ell \in \{ 0,1,2,\ldots,\lfloor (1-\epsilon)k \rfloor \},$ a.a.s. as $p \rightarrow +\infty.$
\end{theorem}This monotonicity result and Lemma \ref{lem:monot} suggests that possibly $\phi(\ell)$ is decreasing as a function of $\ell$ implying indeed the absence of the OGP for all $R \in (\frac{1}{2},1).$ Yet, to establish this we have to tolerate the error naturally appearing in Theorem \ref{first_moment_calc} and Conjecture \ref{conj}. We are able to show that indeed such a error tolerance is possible and we obtain the following result.

\begin{theorem}\label{thm:OGP}
Let arbitrary $\frac{1}{2}<R<1$ and $k \leq p^{1-c}$ for some $c>0$. Suppose Conjecture \ref{conj} holds. Then for any $\epsilon>0$ the following is true asymptotically almost surely as $p \rightarrow +\infty.$ For any $\zeta_{p'}, W_{p'} \in \{0,1,\ldots,k\}$ with $ \zeta_{p'} \leq (1-\epsilon)k,$ $\zeta_{p'}+W_{p'} \leq k,$  $W_{p'}=\omega(1)$ and $H_{p'}>0$ the $(\zeta_{p'},W_{p'},H_{p'})$-OGP does not hold.
\end{theorem}
Note that in Theorem \ref{thm:OGP} we require the very weak assumption that $k \leq p^{1-c}$ for some $c>0$ which captures almost all of the sublinear regime $k=o(p).$ We consider this assumption to be of technical nature.

\subsection{Absence of bad local minima}\label{undesirable_local}

Motivated by the evidence for the absence of the OGP for all rates $R<1$ in the landscape of Bernoulli group testing with $\nu$ with $(1-\frac{\nu}{k})^k=\frac{1}{2}$, which is described in the previous section, we now turn our attention to a much more strict notion of optimization landscape smoothness, namely the absence of ``bad"  local minima. The fact that OGP may not appear for any $R<1$, suggests that the absence of bad local minima may hold for reasonable values of $R$ near $1$. If true, this has clear rigorous algorithmic implications such as the certain success of greedy local search methods. Furthermore, partial motivation for studying this notion of optimization complexity is a rigorous understanding of the performance of Glauber Dynamics, the algorithm we described in the Introduction, and which we implemented for our experimental results (and which achieves an almost perfect success rate in solving the corresponding optimization problem for all $R<1$, see Section \ref{simulations}). Indeed,  the absence of bad local minima guarantees the success of the Glauber dynamics algorithm for large values of $\beta$ and, in particular, for $\beta = + \infty$.

%In our second theorem we study the optimization landscape for provides theoretical evidence that a local improvement algorithm can boost the performance of COMP in the approximate recovery task, while not introducing any false-negatives. In fact, our algorithm provides the best theoretical guarantees in that regard. 

%
%\begin{algorithm}\label{greedia}
%\begin{algorithmic}[1] 
%\Procedure{Greedy Local Search}{$k',\mathrm{PD}, \mathrm{PosTests}$}
%
%
%\State  $\sigma_0 \leftarrow$ an arbitrary state from $\Omega_{k'}$
%
%\State  $N \leftarrow  |\mathrm{PosTests}|$, $t \leftarrow 0$
%
%\While{$P(\sigma_t) < | \mathrm{PosTests} | $ }
%
% \State For all possible pairs $(i,j) \in \sigma_t \times ( \mathrm{PD} \setminus \sigma_t)$, let $\tau_{ij} :=  (\sigma_t \cup \{j\} ) \setminus \{i \} $.
% \State Define $\mathcal{T}_{\sigma_t} = \{    \tau_{ij}:   (i,j) \in \sigma_t \times ( \mathrm{PD} \setminus \sigma_t)    \}$
% \State Move to the state $\tau   \in \mathcal{T}_{\sigma_t} \cup \{\sigma_t \}$ with the highest $P$-value (solving ties uniformly at random).
% \If { $\sigma_t = \sigma_{t+1} $ }
% \State \Return $\sigma_t$
% \EndIf 
%   \State $t \leftarrow t+1$
%\EndWhile
%\Return  $\sigma_t$	
%
%
%
%\EndProcedure
%\end{algorithmic}
%\end{algorithm}

\begin{algorithm}\label{greedia}
\begin{algorithmic}[1] 
\Procedure{Greedy Local Search}{$k',\mathrm{PD}, \mathrm{PosTests}$}

\State  $\sigma_0 \leftarrow$ an arbitrary state from $\Omega_{k'}$

\State  $N \leftarrow  |\mathrm{PosTests}|$, $t \leftarrow 0$

\While{$P(\sigma_t) < N$ }

 \State For all possible pairs $(i,j) \in \sigma_t \times ( \mathrm{PD} \setminus \sigma_t)$, let $\tau_{ij} :=  (\sigma_t \cup \{j\} ) \setminus \{i \} $.
 \State Define $\mathcal{T}_{\sigma_t} = \{    \tau_{ij}:   (i,j) \in \sigma_t \times ( \mathrm{PD} \setminus \sigma_t)    \}$
 \State Let $\tau   \in \mathcal{T}_{\sigma_t} \cup \{\sigma_t \}$ be the state with the highest $P$-value (solving ties uniformly at random).
 \State $\sigma_{t+1} := \tau $
 \If { $\sigma_t = \sigma_{t+1} $ }
 \State \Return $\sigma_t$
 \EndIf 
   \State $t \leftarrow t+1$
\EndWhile
\Return  $\sigma_t$

\EndProcedure
\end{algorithmic}
\end{algorithm}

Recall that $\mathrm{PD}$ denotes the set of potentially defective items, that for every integer $k' \in \{1,2 \ldots, p \}$,  $\Omega_{k'}$ denotes the set of all possible subsets of exactly $k'$ items which do not participate in any negative test and that,  for a state ($k'$-tuple) $\sigma \in \Omega_{k'}$,  $P(\sigma)$ denotes the number of positive tests explained by state $\sigma$. We consider the above greedy local search algorithm (which can be seen as an ``aggressive" version of Glauber dynamics).

In the second theorem of the present work we study the optimization landscape of the [approximate] $k'$-SS problem for $k' \ge k$. One potential motivation for studying values of $k'$ that are larger (but still close to) $k$ is to obtain  algorithms for approximate recovery which never introduce false-negatives errors, with high probability. Indeed, we will show that utilizing our theorem we will be able to slightly  improve upon the  state-of-the-art results for this task. 

To formally state our theorem, we need the following definition.

\begin{definition}\label{bad_local_minima}
Let $k,p \in \mathbb{N}$ with $ 1 \leq k \leq p$. Fix parameters $\delta, \epsilon \in [0,1)$, and set $k' = \lfloor (1+\epsilon) k\rfloor$.  We say that a state $\sigma \in \Omega_{k'} $ is a \emph{ $(\delta,\epsilon) $-bad local minimum} if it contains less
than $\lfloor(1-\delta)k \rfloor $ defective items and there exists no state $\tau \in \mathcal{T}_{\sigma}$ such that $P(\tau) \ge P(\sigma) +1$. 
\end{definition}

Our second theoretical contribution is a sufficient condition for the absence of $(\delta,\epsilon)$-bad local minima.

\begin{theorem}\label{local_search_theorem}
Let $k,p \in \mathbb{N}$ with $ 1 \leq k \leq p$. We assume that $k,p \rightarrow +\infty$ with $k=o(p)$. Fix parameters $\nu>0, R \in (0,1),\delta, \epsilon \in [0,1)$ such that $\delta \epsilon > 0$, set $ k' = \lfloor(1+\epsilon)k \rfloor$, and assume that we observe the outcome of  $n=\lfloor \frac{ \log_2 \binom{p}{k}}{R} \rfloor$ tests under non-adaptive Bernoulli group testing in which each item participates in a test with probability $\nu/k$. If 
\begin{align*}
R<\frac{   \nu \mathrm{e}^{-\nu}   }{\ln2} + \max_{\lambda \geq 0} \min_{\zeta \in [0,1- \delta )} Q(\lambda, \zeta, \nu,  \epsilon), 
\end{align*}
where $Q = Q(\lambda, \zeta, \nu, \epsilon)$ is 
 \begin{align}\label{formula}
Q =-\frac{\ln\left(\mathrm{e}^{-\nu(1+\epsilon) }(\mathrm{e}^{\nu(1-\zeta)  (\mathrm{e}^{-\frac{\lambda (1 + \epsilon -\zeta)}{1 -\zeta }}-1) }-1)  +  (1+ \epsilon-\zeta)\nu \mathrm{e}^{-\nu(1+\epsilon)  }(1-\mathrm{e}^{-\nu(1-\zeta)})(\mathrm{e}^{\lambda}-1)  +1 \right)  }{ (1+\epsilon-\zeta) \ln2},
\end{align}
%then  {\sc Greedy Local Search}  terminates in at most $n$ steps almost surely, and furthermore outputs a $k$-tuple that contains at least $(1-\epsilon)k$ of the defective items asymptotically almost surely as $p \rightarrow +\infty.$
then there exists no $(\delta,\epsilon)$-bad local minimum in $\Omega_{k'}$.
\end{theorem}

A  concrete consequence of Theorem~\ref{local_search_theorem} is the following corollary.  Its proof can be found in Appendix~\ref{CorProof}.
\begin{corollary}\label{no_false_negatives}
Let $k, p \in \mathbb{N}$ with $1 \leq k \leq p$.  We assume that $k,p \rightarrow +\infty$ with $k=o(p)$. Assume that we observe  the outcome of  $n=\lfloor \frac{ \log_2 \binom{p}{k}}{R} \rfloor$ tests under non-adaptive Bernoulli group testing in which each item participates in a test with probability $\ln(5/2)/k$. If  $R  <0.5468 $, then  {\sc Greedy Local Search} with input $k' =  \lfloor(1.01)k \rfloor $  terminates in at most $n$ steps almost surely, and furthermore outputs a $k'$-tuple that contains the $k$ defective items asymptotically almost surely as $p \rightarrow +\infty.$ 
\end{corollary}

%
%If $\alpha$ is know to be appropriately small, then there exists  another algorithm known as ``Separate Decoding"~\cite{scarlett2018near} which achieves a better rate than COMP while also not introducing false-negative errors. In particular, the (asymptotically) achievable rate of Separate Decoding is
%\begin{align*}
%R_{\mathrm{SP} } = \max_{ \delta > 0  }  \min \left\{ (1-\delta) \ln 2 ,   \left(\frac{1}{\alpha} -1 \right) \left(  (1-\delta) \ln  (1-\delta) + \delta   \right)     \ln 2  	\right\}
%\end{align*}
%assuming we choose the value of $\nu$ so that it satisfies $(1-\nu/k)^k = 1/2$.  (Note that $\ln 2 \approx 0.693$.)
%
%Overall, it is not hard to see that if $\alpha$ is sufficiently large, say $\alpha\ge 0.559 $,  then the algorithm of  Corollary~\ref{no_false_negatives}  outperforms both COMP and Separate Decoding on the task of approximate recovery with the guarantee of never introducing false-negative errors.

\section{Experimental results}\label{simulations}

In this section we present experimental results that provide further evidence suggesting the absence of a computational-statistical gap in the Bernoulli  group testing problem with $\nu$ such that $(1-\frac{\nu}{k})^k=\frac{1}{2}$ and the tractability of finding a satisfying set for all $R \leq 1$. In particular, our experiments are in agreement with the prediction from the absence of the OGP for all $R<1.$
%In particular, we approach the problem of finding the Smallest Satisfying Set (SSS)   via a simple local search algorithms. As we have already explained, it is conjectured that the absence of OGP implies that local search algorithms will succeed.
The key insights from our experimental results can be summarized as follows.

\begin{enumerate}

\item The $k$-SS problem is efficiently solvable to exact optimality via a simple local search algorithm even at rate $R =1$. 

\item Solving the SSS  problem via a local search method is efficient and outperforms other popular algorithms in terms of approximate recovery.

\end{enumerate}

We give more details forthwith. We note that all our experiments were performed in a  MacBookPro with a 2.3 GHz 8-Core Intel Core i9 Processor and 16 GB 2667 MHz DDR4 RAM.

In Figure~\ref{Optimization} we demonstrate the performance of Glauber Dynamics for solving the $k$-SS problem  on instances where $k = \lfloor p^{1/3} \rfloor$ and $\nu$ satisfies $(1- \nu/k)^k =1/2$.  While the choice of $\nu$ should already be well-motivated, the choice of $k$ is chosen because $k = \lfloor p^{1/3} \rfloor$ corresponds the largest power of $p$ for which, not only approximate, but even exact recovery is theoretically possible for any $R<1$ by the Bernoulli design of choice  \cite{survey}. Now each data point that corresponds to a particular value of the inverse rate $R^{-1}$ is computed by creating $100$ Bernoulli group testing instances and counting how many times the algorithm succeeds (where  success refers to solving the optimization, exact- and approximate-recovery task, respectively). The gray lines, which support our first point above,  correspond to the probability that Glauber Dynamics successfully solves the $k$-SS problem to exact optimality (i.e. finds a satisfying set). Remarkably, we observe an almost perfect success rate in this task ($>99.7\%$).   The orange lines correspond to the probability of successful exact recovery, while the blue lines correspond to the probability of successful 90\%-approximate recovery.  Note that in terms of exact/approximate recovery the algorithm is not succeeding with high probability when $R$ is close to 1, albeit always outputing a satisfying set, seemingly opposing the theoretical results (Lemma~\ref{key_approx_lemma}). We naturally consider this an artifact of considering finite $p$, which is supported by the second diagram of Figure~\ref{Comparison}, where as $p$ increases, the success probability of the approximate recovery task (of solving the harder SSS problem, which does not know the value of $k$) at a given fixed rate $R<1$ also increases, as predicted by Lemma~\ref{key_approx_lemma}.

% \begin{remark}
% Notice that the first observation above does not necessarily directly translate to successful approximate recovery at rates very close to $1$ in our experiments since we consider finite values of $p$. However, as it can be seen in the second diagram of Figure~\ref{Comparison}, as $p$ increases, the success probability of the approximate recovery task at a given fixed rate $R<1$ also increases, as predicted by Lemma~\ref{key_approx_lemma}.
% \end{remark}

In Figure~\ref{Comparison} we demonstrate how our proposed algorithm for solving now the SSS problem (solving the $k'$-SS problem for $k'=0,1,\ldots,p$ and performing binary search) compares to a series of known popular algorithms for $90\%$-approximate recovery.  Again, we choose $k = \lfloor p^{1/3} \rfloor $, $\nu$ that satisfies $(1- \nu/k)^k = 1/2$, and we compute each data point by creating $100$ Bernoulli group testing instances and counting how many times the algorithm succeeds in terms of the $90\%$-approximate-recovery task. 
The algorithms we compare our approach with are the following. (Some of them we have already discussed, but we include their description here again for completeness. The reader is also referred to~\cite{survey} for more details.)

\begin{itemize}
\item \emph{Combinatorial orthogonal matching pursuit  (COMP)}: In the COMP algorithm we remove every item that takes part in a negative test, and output the rest. Each item in the output of COMP is called \emph{potentially defective} (PD). 

\item  \emph{Definite defective (DD)}: In the DD algorithm we output the set of PD items, which have the property that they are the only PD item in a certain positive test. Each item in the output of DD is called \emph{definite defective} (DD).

\item \emph{Sequential COMP (SCOMP)}: The SCOMP algorithm is defined as follows. (Recalling Remark~\ref{set_cover}, SCOMP is essentially  Chvatal's greedy approximation algorithm for Set Cover.)

\begin{enumerate}
\item Initialize $S$ to be the set of DD items.

\item Say that a positive test is \emph{unexplained} if it does not contain any items from $S$. Add to $S$ the PD item not in $S$ that is in the most unexplained tests, and mark the corresponding tests as no longer unexplained. (Ties may be broken arbitrarily).

\item Repeat Step 2 until no tests remain unexplained. The estimate of SCOMP is $S$.
\end{enumerate}

\item \emph{Max Degree (MD) }:  The MD algorithm sorts the PD items in decreasing order with respect to how many positive tests each such item takes part in. It then outputs the $k$ first PD items in this order.

\end{itemize}

As a final remark,  our results are  in agreement with earlier experimental findings on Markov Chain Monte Carlo algorithms for noisy group testing settings in applied contexts, such as computational biology~\cite{knill1996interpretation,schliep2003group} and security~\cite{furon2012decoding}, but we note that our work is the first one to provide a concrete theoretical explanation for  their strong performance in simulations, and a principled way to exploit local algorithms for implementing the SSS estimator.

\section{Proofs related to the Overlap Gap Property}\label{OGP_proof}

In this section we prove Theorems~\ref{first_moment_calc},~\ref{monotonicity_theorem} and~\ref{thm:OGP} .  

We will find helpful the following  technical results regarding the relative entropy function $\alpha(x,y) := x\ln \frac{x}{y} + (1-x)  \ln \frac{1- x }{1-y } $, $x, y \in [0,1]$.

\begin{lemma}\label{deviation} 
Let $x,y \in (0,\frac{1}{2}]$ with $x<y$ be fixed and $N \rightarrow \infty$. We have
\begin{align*}
\Pr[ \mathrm{Bin}(N,y)  \le xN ]  = \mathrm{e}^{- N \alpha(x,y)  + O(\ln N  )}.
\end{align*}
\end{lemma}

\begin{lemma}\label{partial_derivatives}
The partial derivatives of $\alpha(x,y)$ with respect to $x$ and $y$ are given by the following expressions:
\begin{eqnarray*}
\frac{ \partial }{ \partial x} \alpha(x,y)& = & \ln\left( \frac{ x}{1-x}  \right) - \ln\left( \frac{y}{1-y}  \right); \\
\frac{ \partial }{ \partial y} \alpha(x,y)& = & - \frac{x}{y}  + \frac{1 -x }{ 1- y }.
\end{eqnarray*}
\end{lemma}

\begin{lemma}\label{lem:tayl}
For any $\delta \in (0,1),$ and $c_0=c_0(\delta)>0$ it holds that for all   $ x_0,x,y >0 $ such that $x <  (1-\delta) y$, $x_0 < x$, $y \le \frac{1}{2}$,    we have:  
\begin{align*}
\frac{ \partial }{ \partial x} \alpha(x,y) \le - c_0
\end{align*}
As a consequence:
\begin{align*}
    \alpha(x-x_0,y) \geq \alpha(x,y)+c_0 x_0.
\end{align*}
\end{lemma}

For the proof of Lemma~\ref{deviation} see e.g. Lemma 2 in~\cite{balister2019dense}. Lemma~\ref{partial_derivatives} follows trivially by direct calculations, and Lemma~\ref{lem:tayl} is shown in Appendix~\ref{omitted_formal}.

\subsection{Proof of Theorem~\ref{first_moment_calc}}

Recall that for  any $\theta \in \{0,1\}^{p'}$ with $\|\theta\|_0=k$ we denote by $H(\theta)$ the number of unexplained tests with respect to $\theta$.  Denote also by $\theta^*$ the unknown binary $k$-sparse vector supported on the true defective items and notice that $H(\theta^*)=0.$
Finally, recall Remark~\ref{conditioning_remark}.

Now for $t, \ell \ge 0$, let
\begin{align*}
Z_{\ell,t}=  \left|  \{ \theta \in \{0,1\}^{p'}: \| \theta \|_0 = k, \langle \theta, \theta^* \rangle = \ell, H(\theta)  \le t \}   \right|
\end{align*}and observe that
$$\phi(\ell) \le t \Leftrightarrow Z_{\ell,t}\geq 1.$$
In particular,  by Markov's inequality, for all $t>0$ and $\ell \in \{0,1,2,\ldots, \lfloor(1-\epsilon) k \rfloor \}$:
\begin{align}\label{firstOGPtutu}
\Pr\left[\phi(\ell) \le t\right] = \Pr[  Z_{\ell,t } \ge 1  ]  \leq \mathbb{E}\left[Z_{t,\ell}\right].
\end{align}
Therefore, in order to prove Theorem~\ref{first_moment_calc} it  suffices to show that there exists an appropriate large constant $C = C(\epsilon) >0$ such that if
 \begin{align*}
 t_{\ell}:=n_{\mathrm{pos}} F_{p'}\left(\frac{\ell}{k}\right)-C\ln k, 
 \end{align*} 
 then
 \begin{align}\label{firstOGP_goal}
   \lim_{p \rightarrow +\infty} \sum_{\ell=0}^{\lfloor (1-\epsilon)k\rfloor } \mathbb{E}\left[Z_{t_{\ell},\ell}\right] =0.
\end{align}
To see this notice that combining~\eqref{firstOGPtutu} and~\eqref{firstOGP_goal} implies  that  for all $\ell \leq (1-\epsilon)k$ it holds $\phi(\ell) \ge  n_{\mathrm{pos}} F_{p'}\left(\frac{\ell}{k}\right)-C\ln k$ asymptotically almost surely, which is our claim.

Towards that end, fix  $\ell \in \{ 0,1,2,\ldots, \lfloor(1-\epsilon) k \rfloor \}$ and $\theta \in \{0,1\}^{p'}$ such that  $\| \theta \|_0 = k$ and $ \langle \theta, \theta^* \rangle = \ell$, and observe that by the linearity of expectation:
\begin{align}\label{first_first}
\ex[ Z_{t_{\ell},\ell} ]  = {k \choose \ell} { p' - k\choose k-\ell  } \Pr[  H(\theta) \le t_{\ell} ].
\end{align}

\begin{lemma}\label{OGP_dist}
 $H(\theta)$ is distributed as a binomial random variable $\mathrm{Bin}(n_{\mathrm{pos} }, 1 - 2^{- (1 - \frac{\ell}{k} ) })$.
\end{lemma}
\begin{proof}
Recall that each item takes part in a certain positive test independently of the other items and tests and that we have conditioned on the value of the random variable $\mathcal{P}$, namely the set of indices of the positive tests.
For any fixed $i \in \mathcal{P}$:
\begin{eqnarray*}
\Pr[   \langle X_i, \theta \rangle = 0   \mid i \in \mathcal{P}   ]  &=& \frac{ \Pr[   (\langle X_i, \theta \rangle = 0   )\wedge  ( i \in \mathcal{P}  ) ]    }{ \Pr[ i \in  \mathcal{P}  ]  }  \\
											    & = &  \frac{ (1- \frac{\nu}{k})^k \left( 1 - (1- \frac{\nu }{k }   )^{k-\ell }   \right)      }{ 1 - \left( 1- \frac{\nu}{k} \right)^{k} } \\
											    & = &  \frac{ \frac{1}{2}  \left( 1 - (1- \frac{\nu }{k }   )^{k-\ell }   \right)      }{ \frac{1}{2} } \\
											    & =& 1 - 2^{- (1 - \frac{\ell}{k} ) },
\end{eqnarray*}
concluding the proof. Note that in the above calculation $\Pr[  \cdot ]$ denotes the probability with respect to the original, unconditional  probability space. Recall also that we have chosen $\nu$ so that $(1- \nu/k)^k = 1/2$, a fact we use in the third line of the above calculation.

%\begin{align*}
%\Pr[  \langle X_i, \theta \rangle = 0  ]  =    1 -       \left(1 - \frac{\nu}{k} \right)^{k-\ell  }         =  1 - 2^{- (1 - \frac{\ell}{k} ) } .
%\end{align*}

\end{proof}
Combining~\eqref{first_first}, Lemma~\ref{OGP_dist}, and Lemma~\ref{deviation} (with  $N= n_{\mathrm{pos} }$,  $x = \frac{ t_{\ell} }{  n_{\mathrm{pos} }}$ and $y =  1 - 2^{- (1 - \frac{\ell}{k} ) }$  ) we obtain:
\begin{align*}
\ex[ Z_{t_\ell,\ell} ]  \le {k \choose \ell} { p'-k \choose k-\ell  } \mathrm{e}^{- n_{ \mathrm{pos}} \alpha\left( F_{p'}\left(\frac{\ell}{k}\right)- \frac{ C\ln k}{ n_{\mathrm{pos} } }     , 1-2^{-(1- \frac{\ell}{k})    } \right)  +O( \ln n_{\mathrm{pos} }) }.
\end{align*}
For large enough $p$, and therefore large enough $n_{\mathrm{pos} }$ according to   Lemma~\ref{number_positive},  we can apply Lemma~\ref{lem:tayl} with $x =  \frac{ t_{\ell} }{  n_{\mathrm{pos} }} $, $x_0 =  \frac{C \ln k }{n_{\mathrm{pos} } } $,  $y = 1 -  2^{- (1 - \frac{\ell}{k} ) } $,    $\delta =  2^{- (1 - \frac{\ell}{k} ) } $ to get:
\begin{align*}
\ex[ Z_{t_{\ell},\ell} ]  \le {k \choose \ell} { p'-k \choose k-\ell  } \mathrm{e}^{-n_{\mathrm{pos} } \alpha \left( F_{p'}\left(\frac{\ell}{k}\right)  , 1-2^{-(1- \frac{\ell}{k})    } \right)  -c_0C\ln k  +O( \ln n_{\mathrm{pos} })  } \le  \mathrm{e}^{-c_0C\ln k  +O( \ln n_{\mathrm{pos} }) }.
\end{align*}
where for the second inequality we used the definition of the first moment function $F_{p'}$ and, in particular,~\eqref{func_prop}.

Overall,
\begin{align*}
\sum_{\ell=0}^{\lfloor (1-\epsilon)k \rfloor} \ex[ Z_{t_{\ell}, \ell} ]  \le  k\mathrm{e}^{-c_0C\ln k  + O(  \ln n_{\mathrm{pos} })    }
\end{align*}
which indeed tends to zero for sufficiently large $C$ since, from our conditioning, $n_{\mathrm{pos} } =O(n) = O( k \ln (p/k)  )$ and $k = \Theta(p^{\alpha})$ for some constant $\alpha>0$,  concluding the proof.

\subsection{Proof of Theorem~\ref{monotonicity_theorem}}

% \subsection{Auxilary properties of the first moment function}

We start by showing the following technical lemmas in Appendix~\ref{omitted_formal}.

\begin{lemma}\label{eta_thing}
Suppose $\frac{1}{2}<R<1$. There exists $\delta=\delta(\epsilon)>0$  such that a.a.s. as $p \rightarrow +\infty$ it holds for all $\lambda \in [0, 1-\epsilon ] $:
\begin{itemize}
\item[(a)]  $$\delta \leq \frac{\binom{k}{ \lfloor \lambda k \rfloor }\binom{p'-k}{ \lfloor k(1-\lambda) \rfloor }  }{n_{\mathrm{pos} }(1-\lambda)\ln 2  }\leq 1-\delta,$$ and
 \item[(b)]  $$\delta (1-2^{\lambda-1}) \leq F_{p'}(\lambda) \leq (1-\delta)(1-2^{\lambda-1}).$$
 \end{itemize}
\end{lemma}

\begin{lemma}\label{lem:upper_bound} 
There exists a sufficiently small constant $\eta' > 0$ such that for all $\lambda \in [0,1-\epsilon]$  it holds
\begin{align*}
F_{p'}(\lambda) \leq  (1-2^{\lambda-1})\left(1-(1+\eta')\frac{\binom{k}{ \lfloor \lambda k \rfloor }\binom{p'-k}{ \lfloor k(1-\lambda) \rfloor } }{n_{\mathrm{pos} } (1-\lambda)\ln 2  }\right),
\end{align*}a.a.s. as $p \rightarrow +\infty.$
\end{lemma}

 All the constants (included for example in the asymptotic notations) in this proof may depend on the value of $\epsilon.$

Let us now fix some $\ell \in \{ 0,1,2,\ldots,\lfloor (1-\epsilon)k \rfloor \} $. Using \eqref{func_prop} and elementary algebra with binomial coefficients we obtain that almost surely:
\begin{eqnarray}
    \alpha \left(F_{p'}\left(\frac{\ell+1}{k}\right),1-2^{\frac{\ell+1}{k}-1}\right)-\alpha \left(F_{p'} \left(\frac{\ell}{k}\right) ,1-2^{\frac{\ell}{k}-1}\right) & =&  \frac{ \ln \left( {k \choose  \ell +1 } { {p'-k \choose  k - \ell-1   }    }  \right)  -    \ln \left( {k \choose  \ell  } { {p'-k \choose  k - \ell  }    }  \right)   }{  n_{\mathrm{pos} }  } \nonumber  \\
    							 &= & \frac{\ln \left[\frac{(k-\ell)^2}{(\ell+1)(p'-2k+\ell)}\right]}{n_{\mathrm{pos} }}. \label{first_manip}
\end{eqnarray}

Let $\delta$ be the constant promised by  Lemma~\ref{eta_thing} and let us define the compact convex set  \begin{align*}
 \mathbb{T} = \{(x,y) \in [0,1]^2: \delta y \leq x \leq (1-\delta)y, y \leq 1-2^{-\epsilon}\},
 \end{align*}for which we have for all $\lambda \in [0,1-\epsilon],$ $(F_{p
 '}(\lambda),1-2^{\lambda-1}) \in \mathbb{T},$ a.a.s. as $p \rightarrow +\infty.$ Combining~\eqref{first_manip}  with an  application of the two dimensional mean value theorem on  $\mathbb{T}$ (e.g. by restricting $\alpha$ on the line segment connecting $(\frac{\ell}{k},1-2^{\frac{\ell}{k}-1})$ and $(\frac{\ell+1}{k},1-2^{\frac{\ell+1}{k}-1})$)
we conclude that there exists $(\xi_1,\xi_2) \in \mathbb{T}$ with $\xi_1 \in \left(F_{p'}(\frac{\ell}{k}), F_{p'}(\frac{\ell+1}{k}) \right)$ and $\xi_2 \in [1-2^{\frac{\ell+1}{k}-1},1-2^{\frac{\ell}{k}-1}]$  such that

\begin{align*} 
  \bigg{\langle} \nabla \alpha \left(  \xi_1,\xi_2 \right)  ,\left(F_{p'} \left(\frac{\ell+1}{k}\right),1-2^{\frac{\ell+1}{k}-1}\right)-\left(F_{p'}\left(\frac{\ell}{k} \right) ,1-2^{\frac{\ell}{k}-1} \right)\bigg{\rangle}= \frac{\ln \left[\frac{(k-\ell)^2}{(\ell+1)(p'-2k+\ell )}\right]}{n_{\mathrm{pos}}}.
\end{align*}
or equivalently
\begin{align}\label{sheshe} 
   \frac{ \partial }{ \partial \xi_1} \alpha(\xi_1,\xi_2) \left(F_{p'} \left(\frac{\ell+1}{k}\right)-F_{p'}\left(\frac{\ell}{k} \right)\right) - \frac{ \partial }{ \partial \xi_2} \alpha(\xi_1,\xi_2) \left(2^{\frac{\ell+1}{k}-1}-2^{\frac{\ell}{k}-1} \right)= \frac{\ln \left[\frac{(k-\ell)^2}{(\ell+1)(p'-2k+\ell )}\right]}{n_{\mathrm{pos}}}.
\end{align}

Now from Proposition \ref{prop_fm}, $F_{p'}$ is continuously differentiable in $[0,1-\epsilon].$ Therefore, for constants possibly dependent on $\epsilon,$ we have that it necessarily holds $$\xi_1=F_{p'}\left(\frac{\ell}{k}\right)+O\left(\frac{1}{k}\right).$$ Furthermore since $\ell \leq (1-\epsilon)k$ we also have uniformly over all such $\ell,$ $$\Omega(1) \leq \xi_2=1-2^{\frac{\ell}{k}-1}+O(\frac{1}{k}).$$ Hence, applying Lemma~\ref{partial_derivatives} we get:
\begin{align}\label{blip}
    \frac{ \partial }{ \partial \xi_2} \alpha(\xi_1,\xi_2)=-\frac{\xi_1}{\xi_2}+\frac{1-\xi_1}{1-\xi_2}=\frac{1}{2^{\frac{\ell}{k}-1}}-\frac{F_{p'}(\frac{\ell}{k})}{2^{\frac{\ell}{k}-1}(1-2^{\frac{\ell}{k}-1}) }+O\left(\frac{1}{k}\right).
\end{align}
For similar reasons for constants possibly dependent only on $\epsilon>0,$
\begin{align}\label{bloup}
    2^{\frac{\ell+1}{k}-1}-2^{\frac{\ell}{k}-1}=2^{\frac{\ell}{k}-1}\frac{\ln 2}{k}+O\left(\frac{1}{k^2}\right).
\end{align} 
Combining~\eqref{blip},~\eqref{bloup} allows us to conclude that
\begin{align}\label{part_y}
    \frac{ \partial }{ \xi_2} \alpha(\xi_1,\xi_2) \left(2^{\frac{\ell+1}{k}-1}-2^{\frac{\ell}{k}-1} \right)=\frac{\ln 2}{k}\left(1-\frac{F_{p'}(\frac{\ell}{k})}{1-2^{\frac{\ell}{k}-1}}\right)+O\left(\frac{1}{k^2} \right)
\end{align}

Combing now \eqref{sheshe} and \eqref{part_y}  we obtain:
\begin{align}\label{partial}
   \frac{ \partial }{ \partial \xi_1} \alpha(\xi_1,\xi_2) \left(F_{p'}\left(\frac{\ell+1}{k}\right)-F_{p'}\left(\frac{\ell}{k}\right)\right)= \frac{\ln \left[\frac{(k-\ell)^2}{(\ell+1)(p'-2k+\ell+1)}\right]}{n_{\mathrm{pos} }}+\frac{\ln 2}{k}\left(1-\frac{F_{p'}(\frac{\ell}{k})}{1-2^{\frac{\ell}{k}-1}}\right)+O \left(\frac{1}{k^2} \right).
\end{align}
Now, using Lemma \ref{lem:tayl} and the definition of $\mathbb{T}$ we get that for the constant $c_0=\ln 2>0$,
\begin{align}
    \frac{ \partial }{ \partial \xi_1} \alpha(\xi_1,\xi_2) \leq -c_0,
\end{align}which therefore using \eqref{partial} shows that it suffices to prove that for some  constant $D_0 = D_0(\epsilon)$ it holds
\begin{align*}
  \frac{\ln \left[\frac{(k-\ell)^2}{(\ell+1)(p'-2k+\ell)}\right]}{n_{\mathrm{pos} }}+\frac{\ln 2}{k}\left(1-\frac{F_p(\frac{\ell}{k})}{1-2^{\frac{\ell}{k}-1}}\right)\ge D_0 \frac{\ln  \frac{p'-k}{k}   }{n_{\mathrm{pos}}}+\omega \left(\frac{1}{k^2} \right).
\end{align*}Now since we have $\frac{1}{2}<R<1$ it holds  $\frac{\ln  \frac{p'-k}{k}   }{n_{\mathrm{pos}}}=\Omega(\frac{1}{k})=\omega(\frac{1}{k^2})$ a.a.s. as $p \rightarrow +\infty$ and therefore it suffices to show that for some  constant $D_0 = D_0(\epsilon)$ it holds
\begin{align}\label{goal_RHS}
  \frac{\ln \left[\frac{(k-\ell)^2}{(\ell+1)(p'-2k+\ell)}\right]}{n_{\mathrm{pos} }}+\frac{\ln 2}{k}\left(1-\frac{F_p(\frac{\ell}{k})}{1-2^{\frac{\ell}{k}-1}}\right)\ge D_0 \frac{\ln  \frac{p'-k}{k}   }{n_{\mathrm{pos}}}.
\end{align}

Now using Lemma \ref{lem:upper_bound} and simple rearrangement of the terms, we obtain for some constant $\eta'>0,$
\begin{align}\label{calma}
\frac{\ln 2}{k}\left(1-\frac{F_{p'}(\frac{\ell}{k})}{1-2^{\frac{\ell}{k}-1}}\right) \geq (1+\eta')\frac{\ln \binom{k}{ \ell}\binom{p'-k}{k-\ell}}{(k-\ell)n_{\mathrm{pos} }} \geq (1+\eta')\frac{\ln  \left( \frac{k(p-k)}{(k-\ell)^2}  \right)}{n_{\mathrm{pos} }},
\end{align}
where for the second inequality we used that   $ \binom{k}{ \ell}\binom{p-k}{k-\ell} \geq \left( \frac{k(p-k)}{(k-\ell)^2} \right)^{k-\ell}$, which in turn follows by the well-known inequality  ${a \choose b  } \ge \left( \frac{a}{ b} \right)^{b} $, which is true  
for any pair of integers $a \ge b$.

Hence, by elementary algebra it holds,
\begin{eqnarray}
    \frac{\ln \left[\frac{(k-\ell)^2}{(\ell+1)(p'-2k+\ell+1)}\right]}{n_{\mathrm{pos} }}+\frac{\ln 2}{k}\left(1-\frac{F_p(\frac{\ell}{k})}{1-2^{\frac{\ell}{k}-1}}\right) & \ge &  \frac{\ln \left[\frac{(k-\ell)^2}{(\ell+1)(p'-2k+\ell )}\right]}{n_{\mathrm{pos}} }+(1+\eta')\frac{\ln \frac{k(p'-k)}{(k-\ell)^2}}{n_{\mathrm{pos}}}+O\left(\frac{1}{k^2} \right)    \nonumber \\
   & = &   \frac{\ln \left[\frac{(k-\ell)^2}{(\ell+1)(p'-2k+\ell)}\right]}{n_{\mathrm{pos}}}+\frac{\ln \frac{k(p'-k)}{(k-\ell)^2}}{n_{\mathrm{pos}}} \nonumber \\
    & &+\eta'\frac{\ln \frac{k(p'-k)}{(k-\ell)^2}}{n_{\mathrm{pos}}} +O\left(\frac{1}{k^2} \right)   \nonumber \\
   & = &\frac{\ln  \left[\frac{k(p'-k)}{(\ell+1)(p'-2k+\ell)} \right]}{n_{\mathrm{pos}}}+\eta'\frac{\ln \frac{k(p'-k)}{(k-\ell)^2}}{n_{\mathrm{pos}}}+O \left(\frac{1}{k^2} \right)    \nonumber   \\
  & \ge& \frac{\ln \frac{k}{\ell+1}}{n_{\mathrm{pos}}}+\eta'\frac{\ln  \frac{p'-k}{k} }{n_{\mathrm{pos}}}+O \left(\frac{1}{k^2} \right)  \nonumber \\
   & \ge& D_0 \frac{\ln  \frac{p'-k}{k}   }{n_{\mathrm{pos}}}  \nonumber,
\end{eqnarray}
where $D_0$ is an appropriately small positive constant,  since $k=o(p)$, a.a.s. as $p \rightarrow +\infty$ $n_{\mathrm{pos}}=\Theta(k \ln \frac{p}{k})$ and $\ell \leq (1-\epsilon)k$. This completes the proof of the theorem.

\subsection{Proof of Theorem~\ref{thm:OGP}}

Notice that if for any $\zeta_{p'}, W_{p'} \in \{0,1,\ldots,k\}, H_{p'}>0$ with $ \zeta_{p'} \leq (1-\epsilon)k,$ $\zeta_{p'}+W_{p'} \leq k,$  $W_{p'}=\omega(1)$ and $H_{p'}>0$ the $(\zeta_{p'},W_{p'},H_{p'})$-OGP holds then we can apply Lemma \ref{lem:monot} for $\zeta=\zeta_{p'}, W=W_{p'}, H=H_{p'}$ and finally $M=\lfloor (1-\epsilon)k \rfloor$ (by redefining $W_{p'}$ by $\min\{W_{p'},(1-\epsilon)k\}$ if necessary) to deduce that for some $u=0,1,\ldots,W_{p'}-1$, the function $\phi(\ell W_{p'}+u)$ is not  non-increasing as a function of $\ell =0,1,\ldots, \lfloor (1-\epsilon)k/W_{p'} \rfloor.$

Hence it suffices to show that for any $\epsilon>0$ a.a.s. as $p\rightarrow +\infty$  any $W_{p'}=\omega(1)$ and any $u<W_{p'}$ the function  $\phi(\ell W_{p'}+u)$ is decreasing as a function of $\ell =0,1,\ldots, \lfloor (1-\epsilon)k/W_{p'} \rfloor.$ Notice that now to prove this, it actually suffices to prove that for any $\epsilon>0$ a.a.s. as $p\rightarrow +\infty$  there exists a constant $Q(\epsilon)>0$ such that for any $\ell_2<\ell_1<(1-\epsilon)k$ with $\ell_1>\ell_2+Q(\epsilon)$ it holds $\phi(\ell_2)>\phi(\ell_1)$.

Towards that goal, let us fix some $Q=Q(\epsilon)$ that we will take sufficiently large for our needs. Notice that from Conjecture \ref{conj} for some constant $C=C(\epsilon)$, a.a.s. it holds for any such $\ell_2,\ell_1$  
\begin{align}\label{step_1}
    \phi(\ell_2) -\phi(\ell_1) \geq n_{\mathrm{pos}}\left(F\left(\frac{\ell_2}{k}\right)-F\left(\frac{\ell_1}{k}\right) \right)-2C \ln k.
\end{align}Now using Theorem \ref{monotonicity_theorem} we also have for some constant $D=D(\epsilon)>0$ a.a.s. for any such $\ell_2,\ell_1$ by telescopic summation,
\begin{align}\label{step_2}
    F\left(\frac{\ell_2}{k}\right)-F\left(\frac{\ell_1}{k}\right) \geq (\ell_1-\ell_2)D \frac{\ln \left( \frac{p'-k}{k} \right)  }{n_{\mathrm{pos} }} \geq Q(\epsilon) D \frac{\ln \left( \frac{p'-k}{k} \right)  }{n_{\mathrm{pos} }}
\end{align} Combining \eqref{step_1} and \eqref{step_2} we conclude that a.a.s. for any such $\ell_2,\ell_1$
\begin{align}\label{step_3}
    \phi(\ell_2) -\phi(\ell_1) \geq Q(\epsilon) D \ln \left( \frac{p'-k}{k} \right) -2C \ln k.
\end{align}
Using that we condition on the a.a.s. event that for some sufficiently small $\eta<1-\frac{1}{2R}$ Lemma \ref{pi_prime_estimation}  hold we have a.a.s. $\ln \left( \frac{p'-k}{k} \right) \geq (1-\frac{1}{2R}-\eta) \ln \frac{p}{k}.$ Hence we conclude that a.a.s. for any such $\ell_2,\ell_1$, 
\begin{align}\label{step_4}
    \phi(\ell_2) -\phi(\ell_1) \geq Q(\epsilon) D \left(1-\frac{1}{2R}-\eta\right) \ln \left( \frac{p}{k} \right) -2C \ln k.
\end{align} 
Using now that $k\leq p^{1-c}$ for some $c>0$ we have that $\ln \left( \frac{p}{k} \right)=\Omega(\ln k).$ Hence, indeed we can choose a constant $Q(\epsilon)>0$ so that as $p \rightarrow +\infty$ it holds
\begin{align}\label{step_5}
Q(\epsilon) D \left(1-\frac{1}{2R}-\eta\right) \ln \left( \frac{p}{k} \right) >2C \ln k.
\end{align} Combining \eqref{step_4} and \eqref{step_5} completes the proof.

\section{Proof of Theorem~\ref{local_search_theorem}}\label{local_search_proof}

In this section we prove Theorem~\ref{local_search_theorem}.  We will  assume without loss of generality  that $R \ge  \frac{ \nu \mathrm{e}^{-\nu} }{ \ln 2 } $. (Indeed, if $R <  \frac{ \nu \mathrm{e}^{-\nu} }{ \ln 2 }$, then we can  utilize  only the first  $ \lfloor\log_2 { p \choose k}/ (\frac{ \nu \mathrm{e}^{-\nu} }{ \ln 2 }) \rfloor$ tests and ignore the rest.) 

We start the analysis by estimating the number of possibly defective items which are actually non-defective by proving Lemma~\ref{whatsleft}. Its proof can be found in Appendix~\ref{omitted_local}.

\begin{lemma}\label{whatsleft}
Let $q $ denote the number of possibly defective items which are actually non-defective. For every constant $\eta \in (0,1)$:
\begin{align}\label{piprime}
q  \le (1+\eta) p  \left( \frac{k}{p} \right)^{  \frac{ (1-\eta)  \nu \cdot \mathrm{exp}\left(-\nu (1 +  \frac{ \nu}{2k } ) \right) }  {R \ln2}   },
\end{align}
asymptotically almost surely.
\end{lemma}

%The second lemma provides an estimation to the number of positive tests.
%\begin{lemma}
%Let $n_p$ denote the number of positive tests. Then, for every constant $\eta > 0$:
%\begin{align*}
% n_p \le (1+\eta)(1- \nu/k)^k n,
% \end{align*}
% almost surely.
%\end{lemma}
%\begin{proof}
%Each test is positive with probability $(1- \nu /k)^k$ independently of all other tests and, therefore, $\ex[ n_p] = n(1- \nu/k)^k$. Standard Chernoff estimates conclude the proof of the lemma.
%\end{proof}

We call $q$ the number of possible defective items that are actually non-defective and we condition on its value satisfying the condition of Lemma~\ref{whatsleft} in what follows.

We proceed now towards proving Theorem~\ref{local_search_theorem}. For every $\ell \in \{0, \ldots, k-1\}$ let $\Omega_{k'}^{\ell} \subseteq  \Omega_{k'}$ denote  the set of potentially defective $k'$-tuples that  contain exactly $\ell$ defective items. Fix now an integer $\ell \in \{0, \ldots, k-1\}$ and, for each  $\sigma \in \Omega_{k'}^{\ell}$, let $\mathrm{B}_{\sigma} $ denote the bad event  that there exists no state $\tau_{ij} \in \mathcal{T}_{\sigma} $ such that $P(\tau_{ij}) > P(\sigma)$ and  $i,j$ are a non-defective  and a defective item, respectively. 
(Recall that for all possible pairs $(i,j) \in \sigma \times ( \mathrm{PD} \setminus \sigma  ) $, we define $\tau_{i,j} = ( \sigma \cup \{j \} ) \setminus \{i \} $.)
Recalling Definition~\ref{bad_local_minima}, to show our result it suffices to prove that 
\begin{align}\label{goal}
\lim_{ p \rightarrow +\infty } \sum_{\ell=0}^{ \lfloor (1-\delta) k\rfloor-1} \mathrm{Pr} \left[   \bigcup_{ \sigma  \in \Omega_{k'}^{\ell} } B_{\sigma}  \right] = 0.
\end{align} 
%----noise----
%
%
%\begin{align*}
%q= p \left( 1 - \frac{\nu}{k}((1-\epsilon)\left(1 - \frac{\nu}{k} \right)^k +\epsilon (1-\mathrm{e}^{-\nu}))\right)^n   \le p  \left( \frac{k}{p} \right)^{  \frac{ R\nu( (1-\epsilon) \mathrm{e}^{-\nu}+\epsilon (1-\mathrm{e}^{-\nu}))}{\ln2}   },
%\end{align*}
%for sufficiently large $p$.
%
%
%
%
%
%------

We first upper bound $ \max_{\sigma \in \Omega_{k'}^{\ell}}  \Pr[ B_{\sigma}] $ since, by a union bound and the elementary observation that $|\Omega_{k'}^{\ell}|=\binom{k}{\ell}  \binom{q }{k'-\ell}$, in order to show~\eqref{goal} it suffices to show that
\begin{align}\label{simpler_goal} 
\lim_{p \rightarrow +\infty}  \sum_{\ell=0}^{\lfloor (1-\delta) k'\rfloor-1}\binom{k'}{\ell}  \binom{q }{k'-\ell} \max_{\sigma \in \Omega_{k'}^{\ell}}  \Pr[ B_{\sigma} ] = 0.
 \end{align}
   To that end, fix an $\ell \in \{0,1,\ldots, \lfloor (1-\delta) k\rfloor-1\}$ and an arbitrary state $\sigma \in \Omega_{k'}^{\ell}$.  Let also $\mathcal{T}_{\sigma}^* \subseteq \mathcal{T}_{\sigma}$ denote the set of states $\tau_{ij}$ such that $i$ is a non-defective item in $\sigma$, a set of items we denote by $\mathrm{ND}_{\sigma}$, and $j$ is a defective item which does not belong in $\sigma$, a set of items we denote by $\mathrm{D}_{\setminus \sigma}$.  For each $\tau_{ij} \in \mathcal{T}_{\sigma}^*$ let $\Delta (\tau_{ij} ) = P(\tau_{ij} ) - P(\sigma)$ denote the random variable that equals the difference between the number of positive tests explained by $\tau_{ij}$ and the  number of positive tests explained by $\sigma$. Let $\Theta_{j}$ denote the number of positive tests that contain $j$ and do not contain any item from $\sigma$. Let also $\Theta_{i \overline{j}}$ denote the number of positive tests in which (i) $i$ is the only item in $\sigma$ that participates in the test and; (ii) $j$ does not participate in the test. Using our notation, it is an elementary observation that
\begin{align}\label{identity_Theta}
\Delta( \tau_{ij}  )= \Theta_{j} - \Theta_{i \overline{j} }.
\end{align}
To see this, at first observe that any positive test that contains at least two items from $\sigma$ is explained both by the set of items in $\sigma$ and by the set of items of every  state in $\mathcal{T}_{\sigma}$. This is because our algorithm removes (and, therefore, also  inserts to) at most one item  from $\sigma$. Thus, such tests do not contribute to $\Delta(\tau_{ij})$. Therefore it suffices to focus on the contribution of tests with either one or zero items from $\sigma$. Now observe that the only tests that contribute positively to $\Delta(\tau_{ij})$ (and increase its value by one) among the ones that contain at most one item from $\sigma$, are exactly the ones that contain $j$ and no item of $\sigma$. Analogously, the only tests that contribute negatively to $\Delta(\tau_{ij})$ (by decreasing its value by one) among the ones that contain at most one item from $\sigma$ are the ones in which $i$ is the only item that is contained in $\sigma$, and $j$ is not contained in $\sigma$. The identity~\eqref{identity_Theta} follows.

Using~\eqref{identity_Theta} we now bound $ \Pr[B_{\sigma} ]$ as follows.
\begin{align}
 \Pr[B_{\sigma} ]  & \le \Pr\left[  \bigcap_{ \tau_{ij} \in  \mathcal{T}_{\sigma}^*  } \{ \Delta ( \tau_{ij}  )  \le 0 \} \right]  \nonumber \\ & \le  \Pr\left[  \sum_{ \tau_{ij} \in  \mathcal{T}_{\sigma}^*  } \Delta ( \tau_{ij}  )  \le 0 \right]  \nonumber\\
 &=  \Pr \left[  \sum_{ i \in \mathrm{ND}_{\sigma}, j \in \mathrm{D}_{\setminus \sigma}  }  	( \Theta_j -  \Theta_{i \overline{j} } )   	\le 0	\right]. \label{Bsigmabound} 
\end{align}

We now prove the following equality in distribution. The proof of Lemma~\ref{dist_lemma} can be found in Appendix~\ref{omitted_local}. 
\begin{lemma}\label{dist_lemma} 
For every $\ell=0,1,2,\ldots,k-1$ and $\sigma \in \Omega_{k'}^{\ell}$: 
\begin{align*}
 \sum_{ i \in \mathrm{ND}_{\sigma}, j \in \mathrm{D}_{\setminus \sigma}  }  	( \Theta_j -  \Theta_{i \overline{j} } )   	\overset{d}{=}     \sum_{q = 1}^{ n} C_q^{(\ell) } ,
 \end{align*}
where $C_q^{( \ell) } $ are independently and identically distributed random variables whose distribution  samples with probability  $(1-\frac{\nu}{k})^{ k'} $  from a (scaled\footnote{Here we mean that the result of the binomial is multiplied by a fixed number, namely, $(k'-\ell)$ in this case.}) binomial random variable $(k'-\ell) \mathrm{Bin}(k-\ell, \frac{\nu}{k}  )    $, with probability $\nu \left(\frac{k'}{k}- \frac{\ell}{k} \right) \left( 1- \frac{\nu}{k} \right)^{k'-1}   \left(1 - \left(1 - \frac{\nu}{k} \right)^{k-\ell} \right)$ samples from a (scaled) binomial $-1 \cdot \mathrm{Bin}(k-\ell, 1 - \frac{\nu}{k} )$ and it is $0$, otherwise.
 \end{lemma}

A direct corollary of Lemma~\ref{dist_lemma} is the following. Its proof can also be found in in Appendix~\ref{omitted_local}.
\begin{corollary}\label{cor:prob_bound}
For every $\ell=0,1,2,\ldots,k-1$ and $\sigma \in \Omega_{k}^{\ell}$:
\begin{align*}
  \Pr \left[  \sum_{ i \in \mathrm{ND}_{\sigma}, j \in \mathrm{D}_{\setminus \sigma}  }   	( \Theta_j -  \Theta_{i \overline{j} } )   	\le 0	\right]  \le    \Pr \left[  \sum_{q = 1}^{ n} T^{(\ell)}_q   	\le 0	\right],
 \end{align*}
where $T^{(\ell)}_q$ are independently and identically distributed random variables whose distribution  samples with probability  $(1-\frac{\nu}{k})^{ k'} $  from a (scaled) binomial random variable   $ \frac{k'-\ell}{k-\ell }   \mathrm{Bin}(k-\ell, \frac{\nu}{k}  )    $, equals to $-1$ with probability $\nu \left(\frac{k'}{k}- \frac{\ell}{k} \right) \left( 1- \frac{\nu}{k} \right)^{k'-1}   \left(1 - \left(1 - \frac{\nu}{k} \right)^{k-\ell} \right)$, and it is $0$, otherwise.
\end{corollary}

Overall, combining~\eqref{Bsigmabound} with Corollary~\ref{cor:prob_bound}, and  using the standard strategy for proving large deviation bounds, we obtain:
\begin{align}\label{B_bound}
\Pr[ B_{\sigma} ] \le \Pr \left[ \sum_{q =1}^{n} T^{(\ell)}_q \le 0 \right]  \le \min_{ \lambda \geq 0 } \prod_{q=1}^n \mathbb{E} \left[ \mathrm{e}^{ -  \lambda  T^{(\ell)}_i}  \right]=\min_{ \lambda \geq 0 }  \mathbb{E} \left[ \mathrm{e}^{ -  \lambda  T^{(\ell)}_1}  \right]^n.
\end{align}

Now~\eqref{B_bound} implies that  in order to prove~\eqref{simpler_goal} it suffices to find $\lambda \geq 0$ such that  
\begin{align*}
\lim_{p \rightarrow +\infty } \sum_{\ell=0}^{ \lfloor  (1-\delta )k  \rfloor-1} \exp  \left(  \ln\binom{k}{\ell} \binom{q}{k'-\ell}+n \ln\mathbb{E}\left[ \mathrm{e}^{-\lambda T^{(\ell)}_1} \right] \right) = 0,
\end{align*}or it suffices to have  
\begin{align*}
\lim_{p \rightarrow +\infty } k \exp  \left( \max_{\ell=0,1,\ldots,   \lfloor (1-\delta) k \rfloor -1} \left\{ \ln\binom{k}{\ell} \binom{q}{k'-\ell}+n \ln\mathbb{E}\left[ \mathrm{e}^{-\lambda T^{(\ell)}_1} \right] \right\} \right) = 0,
\end{align*}
or, equivalently, by taking logarithms 
\begin{align}\label{goal2}
\lim_{p \rightarrow +\infty } \left[\ln k+ \max_{\ell=0,1,\ldots,  \lfloor (1-\delta)k \rfloor -1} \left\{ \ln\binom{k}{\ell} \binom{q}{k'-\ell}+n \ln\mathbb{E}\left[ \mathrm{e}^{-\lambda T^{(\ell)}_1} \right] \right\} \right]= -\infty.
\end{align}
It is a straightforward observation using  the elementary $\binom{m_1}{m_2} \leq (m_1 e/m_2)^{m_2}$ and  Lemma~\ref{whatsleft} with a parameter $\eta > 0$ to be defined later, that it holds for some constants $C_0, C_1>0,
$ 
\begin{eqnarray*}
\ln\binom{k}{\ell} \binom{q}{k'-\ell}  &\leq & (k'-\ell) \ln \frac{q}{(1+\epsilon)k- (1-\delta)k}+C_0 k \\
						& \leq & \left(1- (1-\eta) \frac{\nu \mathrm{e}^{-\nu(1 +  \frac{\nu}{2k} )  }}{ R\ln2}\right) \left(1+\epsilon-\frac{\ell}{k} \right) k \ln\frac{p}{k}+C_1 k,
\end{eqnarray*}
where in deriving the second inequality we used the fact that $\epsilon \delta > 0$, i.e., at least one of $\epsilon, \delta$ is positive.
Using the last inequality and the fact that $ n \leq  \frac{ \log_2 { p \choose k }  }{R } $, to show~\eqref{goal2} it suffices to show that
\begin{align*}
 \lim_{p \rightarrow +\infty } \left[C_1k+ F_{1}(\nu,k,\eta, R, \delta, \epsilon)  + \ln k\right] = -\infty,
\end{align*}
where,
\begin{align*}
F_{1}(\nu,k,\eta, R, \delta, \epsilon)  = \max_{\ell=0,\ldots,\lfloor (1-\delta)k\rfloor-1} \left\{ \left(1- (1-\eta) \frac{\nu \mathrm{e}^{-\nu(1 +  \frac{\nu}{2k} )  }}{ R\ln2}\right) \left(1+\epsilon-\frac{\ell}{k} \right) k \ln\frac{p}{k}+  \frac{ \log_2 {p \choose k }  }{R }     \ln\mathbb{E}\left[ \mathrm{e}^{-\lambda T^{(\ell)}_1} \right]  \right\}.
\end{align*}

Now using first $\ln \binom{p}{k} \leq  k \ln ( p/k)  +k   $ and second that since $T^{(\ell)}_1 \geq -1$ almost surely, it holds $\ln\mathbb{E}\left[ \mathrm{e}^{-\lambda T^{(\ell)}_1} \right]  \leq \ln\mathbb{E}\left[ \mathrm{e}^{\lambda } \right] =\lambda$, it suffices to show that
\begin{align*}
 \lim_{p \rightarrow +\infty } \left[ \left(C_1+\frac{ \lambda}{R \ln 2}\right)k+ \ln k + F_2(\nu,k,\eta,R, \epsilon)  \right] = -\infty,
\end{align*}
where,
\begin{align*}
F_2 (\nu,k,\eta,R,\delta, \epsilon) = \max_{\ell=0,\ldots,\lfloor(1-\delta)k \rfloor-1 } \left\{  \left(1- (1-\eta) \frac{\nu \mathrm{e}^{-\nu(1 +  \frac{\nu}{2k} )  }}{ R\ln2}\right) \left(1+\epsilon-\frac{\ell}{k} \right) + \frac{1}{R \ln 2 }  \ln\mathbb{E}\left[ \mathrm{e}^{-\lambda T^{(\ell)}_1} \right]  \right\} k \ln\frac{p}{k}.
\end{align*}

Now, since $k=o(p)$ and also $\lambda \geq 0, R>0$ are fixed constants,     it suffices to show that
$$\limsup_{p \rightarrow +\infty} \max_{\ell=0,\ldots, \lfloor(1-\delta)k \rfloor -1} \left\{  \left(1- (1-\eta) \frac{\nu \mathrm{e}^{-\nu(1 +  \frac{\nu}{2k} )  }}{ R\ln2}\right) \left(1+\epsilon-\frac{\ell}{k} \right) + \frac{1}{R \ln 2 }  \ln\mathbb{E}\left[ \mathrm{e}^{-\lambda T^{(\ell)}_1} \right]  \right\} k \ln\frac{p}{k}<0,$$
 or 
\begin{align*}
R&< \liminf_{p \rightarrow +\infty}  \min_{\ell=0,1,\ldots,  \lfloor(1-\delta)k \rfloor  -1}  \frac{ (1-\eta)  \nu \mathrm{e}^{-\nu (1 +  \frac{\nu}{2k} )} (1+\epsilon-\frac{\ell}{k}) -     \ln\mathbb{E}\left[ \mathrm{e}^{-\lambda T^{(\ell)}_1   } \right]         }{ (1+ \epsilon -\frac{\ell}{k}) \ln2 },
%\\&=\frac{   \nu \mathrm{e}^{-\nu}   }{\ln2}-\limsup_{p \rightarrow +\infty} \max_{\ell=0,1,\ldots,\lfloor (1-\epsilon)k \rfloor} \frac{     \ln\mathbb{E}\left[ \mathrm{e}^{-\lambda T^{(\ell)}_1 \right]  }}{\ln2 (1-\frac{\ell}{k})},
\end{align*} 
which is equivalent to 
\begin{align}\label{final_cond0}
R&< (1-\eta)\frac{   \nu \mathrm{e}^{-\nu (1 +  \frac{\nu}{2k} )} }{\ln2}-\limsup_{p \rightarrow +\infty} \max_{\zeta \in [0,1-\delta]} \frac{\ln\mathbb{E}\left[ \mathrm{e}^{-\lambda T^{(\lfloor \zeta k \rfloor)}_1} \right]  }{ (1+\epsilon-\zeta) \ln2 }.
\end{align}
Now since it suffices that \eqref{final_cond0} holds for some $\eta>0$ our condition by continuity is equivalent with 
\begin{align}\label{final_cond}
R&<  \frac{   \nu \mathrm{e}^{-\nu (1 +  \frac{\nu}{2k} )} }{\ln2}-\limsup_{p \rightarrow +\infty} \max_{\zeta \in [0,1-\delta)} \frac{\ln\mathbb{E}\left[ \mathrm{e}^{-\lambda T^{(\lfloor \zeta k \rfloor)}_1} \right]  }{ (1+\epsilon-\zeta) \ln2 }.
\end{align}

We now prove the following technical lemma in Appendix~\ref{omitted_local}.

\begin{lemma}\label{SwapLimMax} 
Under the assumptions of Theorem~\ref{local_search_theorem}:
$$ \limsup_{p \rightarrow +\infty} \max_{\zeta \in [0,1-\delta)} \frac{\ln\mathbb{E}\left[ \mathrm{e}^{-\lambda T^{(\lfloor \zeta k \rfloor)}_1} \right]  }{ (1+\epsilon-\zeta)\ln2 } =   \max_{\zeta \in [0,1- \delta) }      \limsup_{p \rightarrow +\infty}  \frac{\ln\mathbb{E}\left[ \mathrm{e}^{-\lambda T^{(\lfloor \zeta k \rfloor)}_1} \right]  }{ (1+\epsilon-\zeta) \ln2  }.$$
\end{lemma}
Combining Lemma~\ref{SwapLimMax} with~\eqref{final_cond}, it suffices to have 
\begin{align}\label{final_cond2}
R&<  \frac{    \nu \mathrm{e}^{-\nu } }{\ln2}-\max_{\zeta \in [0,1-\delta)}\limsup_{p \rightarrow +\infty}  \frac{\ln\mathbb{E}\left[ \mathrm{e}^{-\lambda T^{(\lfloor \zeta k \rfloor)}_1} \right]  }{ (1+\epsilon-\zeta)\ln2 }.
\end{align}

By direct computation for every $\zeta \in [0,1-\delta)$ it holds
\begin{align}
\mathbb{E}\left[ \mathrm{e}^{-\lambda T^{(\lfloor \zeta k \rfloor)}_1} \right]&=\left(1-\frac{\nu}{k} \right)^{k'}\left( \left(\mathrm{e}^{-\mu}\frac{\nu}{k}+1-\frac{\nu}{k}\right)^{k-\lfloor \zeta k \rfloor}-1\right) \nonumber \\
												&+\left(\frac{k'}{k}-\frac{\lfloor \zeta k \rfloor}{k}\right)\nu\left(1-\frac{\nu}{k}\right)^{k'-1}\left( 1-\left(1-\frac{\nu}{k}\right)^{k-\lfloor \zeta k \rfloor}\right)(\mathrm{e}^{\lambda}-1)+1,
\end{align}
where $\mu = \mu( \lambda, \epsilon, \zeta,k) = \lambda  \frac{  \lfloor (1+\epsilon)k  \rfloor -   \lfloor \zeta k \rfloor   }{k - \lfloor \zeta k \rfloor }$. Since $k \rightarrow +\infty$ as $p \rightarrow +\infty$ 
we conclude by using elementary calculus that
\begin{align}
  \lim_{p \rightarrow +\infty}   \mathbb{E}\left[ \mathrm{e}^{-\lambda T^{(\lfloor \zeta k \rfloor)}_1} \right]& = \mathrm{e}^{-\nu (1+\epsilon) }(\mathrm{e}^{\nu(1-\zeta)  (\mathrm{e}^{- \frac{\lambda (1 + \epsilon -\zeta)}{1 -\zeta }  }-1) }-1)  \nonumber \\ 
  &+  (1+\epsilon-\zeta)\nu \mathrm{e}^{-\nu (1+\epsilon) }(1-\mathrm{e}^{-\nu(1-\zeta)})(\mathrm{e}^{\lambda}-1)  +1.\label{gen_function}
\end{align}
Combining \eqref{final_cond2} and \eqref{gen_function} it suffices to have 
\begin{align}\label{final_cond3}
R&< \frac{   \nu \mathrm{e}^{-\nu } }{\ln2}+  \min_{\zeta \in [0,1-\delta)} Q( \lambda, \zeta,  \nu,  \epsilon ).
\end{align}
where recall that
\begin{align*}
Q( \lambda, \zeta, \nu, \epsilon ) = - \frac{\ln\left(\mathrm{e}^{-\nu(1+\epsilon) }(\mathrm{e}^{\nu(1-\zeta)  (\mathrm{e}^{-\frac{\lambda (1 + \epsilon -\zeta)}{1 -\zeta }}-1) }-1)  +  (1+ \epsilon-\zeta)\nu \mathrm{e}^{-\nu(1+\epsilon)  }(1-\mathrm{e}^{-\nu(1-\zeta)})(\mathrm{e}^{\lambda}-1)  +1 \right)  }{ (1+\epsilon-\zeta) \ln2}.
\end{align*}

As it suffices that this condition is satisfied by some $\lambda \geq 0$, by optimizing \eqref{final_cond3} over $\lambda \geq 0$ it suffices to have
 \begin{align*}
R&< \frac{   \nu \mathrm{e}^{-\nu } }{\ln2}+\max_{\lambda \geq 0} \min_{\zeta \in [0,1-\delta)} Q(\lambda, \zeta,  \nu,  \epsilon),
\end{align*}
concluding the proof.

\section{Proof of Lemma~\ref{key_approx_lemma}}\label{key_approx_lemma_proof}

% \begin{lemma}\label{local_local_local_New}
% Let $k, p,d \in \mathbb{N}$ with $1 \le k \le p$. We assume that $k,p \rightarrow +\infty$ with $k = o(p)$.  Fix parameters $\delta, \epsilon \in (0,1)$. Assume we observe  $n \geq  (1+\epsilon) k \log_2 \frac{p}{k}  $ tests  under  non-adaptive Bernoulli group testing in which each item participates in a test with probability $\nu/k$, where $\nu$ satisfies $(1-\frac{\nu}{k })^{\nu} = \frac{1}{2}$. There exists an $\epsilon'>0$ such that every set of size $k$ which satisfies at least $(1-\epsilon')n$ tests, it must contain at least $\lfloor(1-\delta) k \rfloor$ defective items, a.a.s. as $p \rightarrow +\infty.$
% \end{lemma}

Let us consider some fixed $\epsilon'>0$ which will be chosen sufficiently small given the needs of the proof. 

For $\ell=0,1,\ldots,\lfloor (1-\delta) k \rfloor$ let $Z_{\epsilon,\epsilon',\ell}$ be the counting random variable of the number of sets of size $k$ which satisfy at least $(1-\epsilon')n $ tests and contain exactly $\ell$ defective items (with overlap $\ell$). By two union bounds, first over $\ell$ and then over the different sets of overlap $\ell$ and by fixing for each $\ell,$ an arbitrary set $\theta_{\ell}$ with overlap $\ell$ we have
\begin{align*}
\Pr\left[\bigcup_{\ell=0}^{\lfloor (1-\delta) k \rfloor} \{ Z_{\epsilon,\epsilon',\ell} \geq 1\}\right] & \leq \sum_{\ell=0}^{\lfloor (1-\delta) k \rfloor }\Pr\left[ Z_{\epsilon,\epsilon',\ell}\geq 1\right]\\
&\leq \sum_{\ell=0}^{\lfloor (1-\delta) k \rfloor }\binom{k}{\ell} \binom{p-k}{k-\ell} \Pr\left[\{ \theta_{\ell} \text{ satisfies }   (1-\epsilon') n  \text{ tests} \}\right],
\end{align*}where we used that the exchangeability of the events  $\{ \theta_{\ell} \text{ satisfies }   (1-\epsilon') n  \text{ tests} \}$ for the different possible choice of $\theta_{\ell}.$ Now using also a union bound over the possible subset of tests of cardinality  $\lceil (1-\epsilon') n \rceil$ that $\theta_{\ell}$ satisfies, we have for some fixed subset of tests $S$ of cardinality $\lceil (1-\epsilon') n \rceil$ that
\begin{align*}
\Pr\left[\bigcup_{\ell=0}^{\lfloor (1-\delta) k \rfloor} \{ Z_{\epsilon,\epsilon',\ell} \geq 1\}\right] 
&\leq \binom{n}{\lceil (1-\epsilon') n \rceil} \sum_{\ell=0}^{\lfloor (1-\delta) k \rfloor }\binom{k}{\ell} \binom{p-k}{k-\ell} \Pr\left[\{ \theta_{\ell} \text{ satisfies the tests in }S \}\right]\\
&\leq \binom{n}{\lceil (1-\epsilon') n \rceil} \sum_{\ell=0}^{\lfloor (1-\delta) k \rfloor }\binom{k}{\ell} \binom{p-k}{k-\ell} \Pr\left[\{ \theta_{\ell} \text{ satisfies a fixed test} \}\right]^{\lceil (1-\epsilon') n \rceil},
\end{align*}where we have used the exchangeability of the events $\{ \theta_{\ell} \text{ satisfies the tests in }S \}$ for the different $S$ of a fixed size, and the independence of the tests. Now it is well-known computation in the group testing literature (see e.g. \cite[Section II.A]{Truongallornothing}) that for each $\ell$, $$\Pr\left[\{ \theta_{\ell} \text{ satisfies a fixed test} \}\right]=2^{\frac{\ell}{k}-1}.$$Hence it holds 
\begin{align*}
\Pr\left[\bigcup_{\ell=0}^{\lfloor (1-\delta) k \rfloor} \{ Z_{\epsilon,\epsilon',\ell} \geq 1\}\right] 
&\leq \binom{n}{\lceil (1-\epsilon') n \rceil} \sum_{\ell=0}^{\lfloor (1-\delta) k \rfloor }\binom{k}{\ell} \binom{p-k}{k-\ell} 2^{-\lceil (1-\epsilon') n  \rceil \frac{k-\ell}{k} },
\end{align*}which now using standard asymptotics $\binom{n}{\alpha n} \leq e^{ C\alpha \log \frac{1}{\alpha} n}$ for $\alpha>0$ and some universal constant $C>0$ and $\binom{k}{\ell} \binom{p-k}{k-\ell} \leq e^{(k-\ell) \ln \frac{p}{k}+O(k)}$ allows us to conclude for some constant $C'>C>0$
\begin{align*}
\Pr\left[\bigcup_{\ell=0}^{\lfloor (1-\delta) k \rfloor} \{ Z_{\epsilon,\epsilon',\ell} \geq 1\}\right] 
&\leq e^{C\epsilon' \log (\frac{1}{\epsilon'}) n} \sum_{\ell=0}^{\lfloor (1-\delta) k \rfloor } e^{(k-\ell) \ln \frac{p}{k}+O(k)} 2^{- (1-\epsilon') n  \frac{k-\ell}{k} }\\
&\leq \sum_{\ell=0}^{\lfloor (1-\delta) k \rfloor } \exp \left(C \epsilon' \log (\frac{1}{\epsilon'}) k \ln \frac{p}{k}+(k-\ell) \ln \frac{p}{k}+O(k) -(1-\epsilon') (1+\epsilon)(k-\ell)\ln \frac{p}{k} \right) \\
&\leq \sum_{\ell=0}^{\lfloor (1-\delta) k \rfloor }  \exp \left( (1-(1-\epsilon') (1+\epsilon))(k-\ell)\ln \frac{p}{k} +C \epsilon' \log (\frac{1}{\epsilon'}) k \ln \frac{p}{k}+O(k))\right)\\
&\leq k \exp \left( (1-(1-\epsilon') (1+\epsilon))\delta k\ln \frac{p}{k} +C\epsilon' \log (\frac{1}{\epsilon'}) k \ln \frac{p}{k}+O(k) \right)\\
&=  \exp \left( (1-(1-\epsilon') (1+\epsilon))\delta k\ln \frac{p}{k} +C\epsilon' \log (\frac{1}{\epsilon'}) k \ln \frac{p}{k}+O(k+\log k) \right)\\
&=  \exp \left( ( (1-(1-\epsilon') (1+\epsilon))\delta+C'\epsilon'\log (\frac{1}{\epsilon'})) k\ln \frac{p}{k} \right),
\end{align*}where for the asymptotics we used that $k=o(p).$ We now choose $\epsilon'>0$ sufficently small such that  $(1-(1-\epsilon') (1+\epsilon))\delta+C'\epsilon' \log (\frac{1}{\epsilon'})<-\frac{\epsilon \delta}{2}<0$ to get  
\begin{align*}
\Pr\left[\bigcup_{\ell=0}^{\lfloor (1-\delta) k \rfloor} \{ Z_{\epsilon,\epsilon',\ell} \geq 1\}\right] 
&\leq  \exp \left( -\frac{\epsilon \delta}{2} k\ln \frac{p}{k} \right),
\end{align*}which clearly tends to zero as $p$ grows to infinity and $k=o(p).$ Hence $\bigcap_{\ell=0}^{\lfloor (1-\delta) k \rfloor} \{ Z_{\delta,\epsilon,\epsilon',\ell} =0\} $ happens a.a.s. as $p \rightarrow +\infty,$ which is exactly what we want.

\section{OGP related proofs Omitted from Sections~\ref{formal_statement}  and~\ref{OGP_proof}  }\label{omitted_formal}

In this section we present the proofs of Lemmas~\ref{number_positive},~\ref{pi_prime_estimation},~\ref{lem:monot},~\ref{lem:tayl},~\ref{eta_thing},~\ref{lem:upper_bound} and Proposition~\ref{prop_fm}.

\subsection{Proof of Lemma~\ref{number_positive}} 

Observe that each test is positive with probability $1-(1- \frac{\nu}{k})^{k}$ independently of all other tests. Hence, by linearity of expectation we have  $\ex[ | \mathcal{P} | ] = n (1- (1- \frac{\nu}{k})^k ) \xrightarrow{p \rightarrow +\infty} n/2 $. Now standard Chernoff estimates imply that for any constant $\eta' \in (0,1)$
$$\Pr[   | | \mathcal{P} | - \ex[  |\mathcal{P} |  ]    |  > \eta' \ex[ | \mathcal{P} | ]   ] \le \mathrm{exp}(- (\eta')^2  \ex[ | \mathcal{P} |] /3 )   \xrightarrow{ p \rightarrow +\infty} 0.$$
Combining these two facts concludes the proof.

\subsection{Proof of Lemma~\ref{pi_prime_estimation}}
Let $p'= | \mathrm{PD} | $. We  only show the lower bound on $p'$. The  proof of the upper bound is similar, and in fact almost identical to the one of the more refined Lemma~\ref{whatsleft}.

Recall that $k = \lfloor p^{\alpha} \rfloor, \alpha \in (0,1)$, and let $ \delta = \min\{\alpha, \eta \}/10 $. Let   $n_{\mathrm{neg}} = n - n_{\mathrm{pos} } $  denote the number of negative tests and notice that  Lemma~\ref{number_positive} implies that $n_{\mathrm{neg} } \le (1 +\delta) n/2 $.  We condition on that  $n_{\mathrm{neg} }  = N_{\mathrm{neg} }$, where $ N_{\mathrm{neg} } \le (1+\delta) n/2   $ is a fixed value. Since the probability that a fixed non-defective item appears in a particular test is $\nu/k$, we see that:
\begin{eqnarray}
\mathbb{E}[p'-k \mid n_{\mathrm{neg} }  = N_{\mathrm{neg} } ] &= & (p-k) \left( 1 - \frac{\nu}{k} \right)^{ N_{ \mathrm{neg}} }      \nonumber  \\
											    & \ge & (p-k) \mathrm{exp} \left(  - \frac{  N_{\mathrm{neg } } }{  \frac{k}{\nu } -1  }    \right)  \label{first_pi_bound} \\
											    & =&  (p-k)  \mathrm{exp}\left( - \frac{ (1+\delta)  \ln   {p \choose k }   }{    \left(  \frac{k}{\nu } -1 \right)  R \ln 2   }  		\right)  \nonumber \\
											     &  \ge & 	(p-k)  \mathrm{exp}\left( - \frac{ (1+\delta)  ( k \ln \frac{p}{k} + k )     ) }{   2 \left(  \frac{k}{\nu } -1 \right)  R \ln 2   }  		\right) \label{second_pi_bound}  \\
											        & \ge& 	(p-k)  \mathrm{exp}\left( - \frac{ (1+2\delta)  (  \ln \frac{p}{k} + 1 )     ) }{   2R    }  		\right)   \nonumber \\
											        & =& (p-k) \mathrm{e}^{ - \frac{ 1 + 2 \delta}{2R }  }  \left(  \frac{k}{p}	\right)^{ \frac{1+ 2 \delta}{ 2R}  }   \nonumber.
\end{eqnarray}
Note that in~\eqref{first_pi_bound} we used the fact that $1 - 1/x > \mathrm{e}^{ -1/( x-1 ) } $, for $x \ge 2$, while in~\eqref{second_pi_bound} the fact that $ { p \choose k } \le  \left( \frac{ \mathrm{e} p  }{k  } \right)^k $.

Note also that our assumption  that $R \ge \frac{1}{2}  $ implies that as $p \rightarrow +\infty$, 
$$ (p-k) \mathrm{e}^{ - \frac{ 1 + 2 \delta}{2R }  }  \left(  \frac{k}{p}	\right)^{ \frac{1+ 2 \delta}{ 2R}  }    = \Omega( k^{ 2 \delta  }  )=\omega(1).$$ 
Thus, since each non-defective item participates in a negative test independently of the other non-defective items,  standard Chernoff estimates  imply that
\begin{align*}
\Pr\left[   p' -k \le (1-\delta) (p-k) \mathrm{e}^{ - \frac{ 1 + 2 \delta}{2R }  }  \left(  \frac{k}{p}	\right)^{ \frac{1+ 2 \delta}{ 2R}  }   \big| n_{ \mathrm{neg}} = N_{ \mathrm{neg} }   \right]  \le \mathrm{exp}( - \delta^2 k^{2 \delta }  /3 ) \xrightarrow{p \rightarrow +\infty }  0.
\end{align*}
The above concludes the proof since $ N_{\mathrm{neg} } \le (1+2\delta) n/2 $ is chosen arbitrarily,    we have already established that $n_{\mathrm{neg} } \le (1+2\delta) n/2 $ asymptotically almost surely and, recalling that $\delta \le \eta/10$, we have  
\begin{align*}
 (1-\delta) (p-k) \mathrm{e}^{ - \frac{ 1 + 2 \delta}{2R }  }  \left(  \frac{k}{p}	\right)^{ \frac{1+ 2 \delta}{ 2R}  } \ge (1-\eta)  p  \left(  \frac{k}{p}	\right)^{ \frac{1+ \eta}{ 2R}  }   
\end{align*}
for sufficiently large $p$.

% \subsection{Proof of Lemma~\ref{lem:monot}} 

% Assume that $\Phi$ satisfies the $(\zeta_{p'},L)$-OGP for some $\zeta_{p'}<M$. Then by the definition of the OGP and of  function $\phi$ it holds $$\phi(\zeta_{p'})\leq r_{p'} < \phi(\zeta_{p'}+L).$$If by Euclidean division $\zeta_{p'}=L \lfloor \zeta_{p'}/L \rfloor +u$ for $u \in \{0,1,\ldots,L-1\}$, then we have $\zeta_{p'}+L=L( \lfloor \zeta_{p'}/L \rfloor+1) +u$. In particular,  $$\phi(L \lfloor \zeta_{p'}/L \rfloor +u) < \phi(L( \lfloor \zeta_{p'}/L \rfloor+1) +u),$$implying that indeed $\phi(\ell L+u)$ is not non-increasing for $\ell=0,1,2,\ldots,\lfloor M/L \rfloor.$

\subsection{Proof of Lemma~\ref{lem:monot}}

Assume that $\Phi$ satisfies the $(\zeta,W,H)$-OGP for some for some $\zeta<M<M+W<\arg \min_{\ell=0,1,\ldots,k} \phi(\ell)$. Then by the definition of the OGP and of  function $\phi$ it holds $$\phi(\zeta)\leq r_{p'} < \phi(\zeta+W).$$If by Euclidean division $\zeta=W \lfloor \zeta/W \rfloor +u$ for $u \in \{0,1,\ldots,W-1\}$, then we have $\zeta+W=W( \lfloor \zeta/W \rfloor+1) +u$. In particular,  $$\phi(W \lfloor \zeta/W \rfloor +u) < \phi(W( \lfloor \zeta/W \rfloor+1) +u),$$implying that indeed $\phi( W \ell+u)$ is not non-increasing for $\ell=0,1,2,\ldots,\lfloor M/W \rfloor.$

\subsection{Proof of Proposition \ref{prop_fm}}

Let us first fix a $\lambda \in [0,1).$ Using Lemma \ref{eta_thing} for some $0<\epsilon<1-\lambda$ we have that \begin{align}\label{first}
0<\frac{\ln \left[\binom{k}{ \lfloor \lambda k \rfloor }\binom{p' -k }{ \lfloor (1-\lambda)k \rfloor}\right]}{n_{\mathrm{pos} }}<(1-\lambda)\ln 2<\ln 2,
\end{align}a.a.s. as $p \rightarrow +\infty$. By elementary inspection the function that sends $x \in (0,1-2^{\lambda-1}]$ to $\alpha(x,1-2^{\lambda-1})$ is continuous, strictly decreasing with image $[0,\ln 2).$ Hence given (\ref{first}) we conclude that there exists a unique $x=F_{p'}(\lambda)$ satisfying \eqref{func_prop} for this fixed value of $\lambda.$

For the last part, let us fix some $\epsilon>0.$ Using again Lemma \ref{eta_thing} we conclude that for some $\delta=\delta(\epsilon)>0$ it holds for all $\lambda \in [0,1-\frac{\epsilon}{2}]$ that
 \begin{align}\label{first2}
\delta(1-2^{\lambda-1}) \leq F_{p'}(\lambda)  \leq (1-\delta)(1-2^{\lambda-1}),
\end{align}a.a.s. as $p \rightarrow +\infty.$ In particular, $F_{p'}(\lambda)$ is the unique solution of $$G(x,\lambda):=\alpha(x,1-2^{\lambda-1})-\frac{\ln \left[\binom{k}{ \lfloor \lambda k \rfloor }\binom{p' -k }{ \lfloor (1-\lambda)k \rfloor}\right]}{n_{\mathrm{pos} }}=0$$ where $(x,\lambda)$ for now on are assumed to always satisfy $x \leq (1-\delta)(1-2^{\lambda-1})$ and $1>x,1-2^{\lambda-1} >0$.

Towards proving the continuous differentiability of $F_{p'}(\lambda)$ on some $\lambda^* \in [0,1-\epsilon]$, notice that we can restrict $\lambda$ in a sufficiently small interval around $\lambda^* \in [0,1-\epsilon]$ which is always included in $[0,1-\frac{\epsilon}{2}]$, so that the term $\frac{\ln \left[\binom{k}{ \lfloor \lambda k \rfloor }\binom{p' -k }{ \lfloor (1-\lambda)k \rfloor}\right]}{n_{\mathrm{pos} }}$ is constant as a function of $\lambda$ on that interval. Now in this small interval $G(x,\lambda)$ is continuous differentiable and furthermore since $x \leq (1-\delta)(1-2^{\lambda-1})$ and $1>x,1-2^{\lambda-1} >0$, it also holds $\frac{\partial G}{\partial x} (x,\lambda^*) =\frac{\partial \alpha}{\partial x} (x,\lambda^*) \not =0$ based on Lemma \ref{lem:tayl}. Hence, by the two dimensional implicit function theorem we can conclude indeed the continuous differentiability of $F_{p'}(\lambda)$ at $\lambda=\lambda^*.$

\subsection{Proof of Lemma~\ref{lem:tayl}} 

If we are able to prove that  $$-\frac{ \partial }{ \partial x} \alpha(x,y) \geq c_0$$ then the second  claim of the lemma follows immediately by the mean value theorem. Recalling Lemma~\ref{partial_derivatives} we see that, equivalently,  it suffices to show 
\begin{align}\label{targetaki} 
\ln\left( \frac{ y(1-x)}{(1-y)x }  \right)\geq c_0.
\end{align}
Using the hypothesis we obtain
\begin{align*}
    \frac{ y(1-x)}{(1-y)x} \geq \frac{ 1-(1-\delta)y}{(1-y)(1-\delta)} \geq \frac{1}{1-\delta}>1,
\end{align*}
which implies~\eqref{targetaki}  for $c_0:=\ln \frac{1}{1-\delta}>0$, concluding the proof.

\subsection{Proof of Lemma~\ref{eta_thing}} 

We first prove part (a). We fix a sufficiently small but fixed $\eta>0$ and condition on \eqref{number_positive} and \eqref{pi_prime_estimation} for this value of $\eta$. The exact bounds on $\eta$ become apparent by going through the proof.
Note that as an easy consequence of \eqref{pi_prime_estimation} it holds \begin{align}\label{ln}
    \ln \frac{p'}{k}=(1-\frac{1}{2R}+O(\eta))\ln \frac{p}{k}
\end{align} and in particular since $k=o(p)$ and $\frac{1}{2}<R<1$ we can choose $\eta$ sufficiently small such that it also holds $k=o(p').$ Using these observations and the elementary $\binom{a}{b}=(1+o(1))a \log \frac{b}{a}$ if $a=o(b),$ it follows that a.a.s. as $p \rightarrow +\infty,$ it holds that for all $\lambda \in [0, 1-\epsilon ] $,

\begin{align*}
  \frac{\ln \left[\binom{k}{ \lfloor \lambda k \rfloor }\binom{p' -k }{ \lfloor (1-\lambda)k \rfloor}\right]}{n_{\mathrm{pos} }} &= \frac{\ln \left[\binom{p' -k }{ \lfloor (1-\lambda)k \rfloor}\right]}{n_{\mathrm{pos} }}+O\left(\frac{k}{n_{\mathrm{pos} }}\right) \\
   &= 2\frac{\ln \left[\binom{p' -k }{ \lfloor (1-\lambda)k \rfloor}\right]}{n}+O\left(\frac{k}{n}\right)+O(\eta), \text{ using } \eqref{number_positive}\\
    &= 2\frac{\lfloor (1-\lambda)k \rfloor \ln \frac{p' -k }{ \lfloor (1-\lambda)k \rfloor}}{n}+O\left(\frac{k}{n}\right)+O(\eta) \\
  &= 2\frac{((1-\lambda)k+O(1)) \ln \frac{p' -k }{ \lfloor (1-\lambda)k \rfloor}}{n}+O\left(\frac{k}{n}\right)+O(\eta)\\
  &= 2R\frac{(1-\lambda)k \ln \frac{p' -k }{ \lfloor (1-\lambda)k \rfloor}}{ k \log_2 \frac{p}{k}}+O\left(\frac{1}{k} \right)+O\left(\frac{k}{ k \log_2 \frac{p}{k}}\right)+O\left(\eta\right) \\.
\end{align*}Now using that $k=o(p'),k=o(p)$ and \eqref{ln} as well as that $k$ grows with $p$ we can further conclude that

\begin{align*}
  \frac{\ln \left[\binom{k}{ \lfloor \lambda k \rfloor }\binom{p' -k }{ \lfloor (1-\lambda)k \rfloor}\right]}{n_{\mathrm{pos} }}
&= 2R\ln 2\frac{(1-\lambda) (\ln (\frac{p' }{ k })+O(\frac{k}{p'
}))}{  \ln \frac{p}{k}}+O\left(\eta\right) \\
&= 2R(1-\frac{1}{2R}+O(\eta)) \ln 2\frac{(1-\lambda) \ln \frac{p}{k}}{  \ln \frac{p}{k}}+O\left(\eta\right) \\
&=(2R-1+O(\eta)) \ln 2(1-\lambda)+O\left(\eta\right) \\
&= (2R-1) \ln 2(1-\lambda)+O\left(\eta\right), \\
\end{align*} which, since  $\epsilon<1-\lambda <1$, it implies
\begin{align*}
    \frac{\ln \left[\binom{k}{ \lfloor \lambda k \rfloor }\binom{p' -k }{ \lfloor (1-\lambda)k \rfloor}\right]}{n_{\mathrm{pos}} \ln2 (1-\lambda) }=2R-1+O(\eta).
\end{align*}Now recall that we assume $\frac{1}{2}<R<1$ and therefore $0<2R-1<1$. Hence choosing $\eta$ small enough allows to conclude that the ratio is bounded away from $0$ and $1$ a.a.s. as $p \rightarrow +\infty.$

For the second part, we define the function $Q(x,y)=\alpha(x,y)/\ln( \frac{1}{1-y})$ which is continuous on $T:=\{(x,y) \in [\frac{1}{2},1-2^{-\epsilon}]^2: x \leq y\}$ with image $Q(T)=[0,1)$. Furthermore, by direct inspection we have that $Q(x,y)$ approaches $1$ if and only if $x/y$ approaches $0$ and $Q(x,y) $ approaches $ 0 $ if and only if $x/y$ approaches $1.$ Hence, by the definition of $F_{p'}(\lambda)$, in order to show the lemma it suffices to show that there exists $\delta=\delta(\epsilon) \in (0,1)$ such that, for sufficiently large $p$, it holds for all $\lambda \in [0,1-\epsilon]$: \begin{align} \label{unif0}
  \delta \leq  Q(F_{p'}(\lambda),1-2^{\lambda-1}) \leq 1-\delta,
\end{align} a.a.s. as $p \rightarrow +\infty $. Indeed, under \eqref{unif0} we have for all $\lambda \in [0,1-\epsilon]$, $(F_{p'}(\lambda),1-2^{\lambda-1}) \in Q^{-1}[\delta,1-\delta]$ a.a.s. as $p \rightarrow +\infty $ and, therefore, the result follows as using the continuity of $Q$, and that $Q^{-1}[\delta,1-\delta]$ is a compact set included in the square $[\frac{1}{2},1-2^{-\epsilon}]^2$ and has positive distance from the lines $x/y=1$ and $x/y=0.$ 

Notice that \eqref{unif0} from the definition of $Q$ is equivalent with the property that a.a.s as $p \rightarrow +\infty$ for all $\lambda \in [0,1]$
\begin{align} \label{unif}
  \delta \leq  \frac{\alpha(F_{p'}(\lambda),1-2^{\lambda-1})}{\log \frac{1}{1-2^{\lambda-1}}}= \frac{\ln \left[\binom{k}{ \lfloor \lambda k \rfloor }\binom{p' -k }{ \lfloor (1-\lambda)k \rfloor}\right]}{n_{\mathrm{pos}}\ln 2 (1-\lambda) } \leq 1-\delta,
\end{align}which is proven in the first part of the lemma.

\subsection{Proof of Lemma~\ref{lem:upper_bound}}

We first claim that to show the  lemma it suffices to prove that there exists sufficiently small  $\eta'>0$ and $\delta'>0$ such that for all $\lambda,u$ with $\lambda \in [0,1-\epsilon]$ and $\delta' \leq u \leq 1-\delta'$ it holds
\begin{align}\label{ineq1}
\alpha((1-2^{\lambda-1})(1-u),1-2^{\lambda-1})<\frac{u}{1+\eta'}(1-\lambda)\ln2.
\end{align} and 
\begin{align}\label{ineq2}
  \delta' \leq  (1+\eta')\frac{\binom{k}{\lambda k}\binom{p-k}{k(1-\lambda)} }{n_{\mathrm{pos} }(1-\lambda)\ln 2  } \leq 1-\delta',
\end{align} a.a.s as $p \rightarrow +\infty$. Indeed the result would then follow by setting $u^*=(1+\eta')\frac{\binom{k}{\lambda k}\binom{p-k}{k(1-\lambda)} }{n_{\mathrm{pos} }(1-\lambda)\ln 2  }$ and observing that using the definition of the function $F_{p'}$ it holds \begin{align*}
  \alpha((1-2^{\lambda-1})(1-u^*),1-2^{\lambda-1})<\alpha(F_p(\lambda),1-2^{\lambda-1}).
\end{align*}Now using the fact that $\alpha(x,y)$ is decreasing, for fixed $y$ as a function of $x \leq y$, we conclude that it must hold for all $\lambda \in [0,1-\epsilon]$ \begin{align*}
   F_p(\lambda) \leq  (1-2^{\lambda-1})(1-u^*)=(1-2^{\lambda-1})\left(1-(1+\eta')\frac{\binom{k}{ \lfloor \lambda k \rfloor }\binom{p'-k}{ \lfloor k(1-\lambda) \rfloor } }{n_{\mathrm{pos} } (1-\lambda)\ln 2  }\right),
\end{align*}which is what we want.

Note that \eqref{ineq2} is true for sufficiently small $\eta',\delta'>0$ directly by the first part of Lemma \ref{eta_thing}.

We now prove the deterministic inequality \eqref{ineq1} by fixing some $\delta' >0$ satisfying \eqref{ineq2} for some value of $\eta'$, say $\eta'_1>0$, and then choosing $\eta'$ sufficiently small with $\eta' \in (0,\eta'_1)$, which allows us to establish \eqref{ineq1}.

Notice that the $\lambda,u$ for which we want to establish \eqref{ineq1}, satisfy $\lambda \in [0,1-\epsilon]$ and $\delta' \leq u \leq 1-\delta'$, correspond to a compact set. Hence by continuity to establish the strict inequality \eqref{ineq1} for some sufficiently small $\eta' \in (0,\eta'_1)$ it in fact suffices to show it for $\eta'=0$ that is prove for all $\lambda \in [0,1-\epsilon]$ and $\delta' \leq u \leq 1-\delta'$,
\begin{align*}
\alpha((1-2^{\lambda-1})(1-u),1-2^{\lambda-1})<u(1-\lambda)\ln 2. 
\end{align*}
Equivalently,  setting $x=1-2^{\lambda-1} \in [\frac{1}{2},1-2^{-\epsilon}]$, it suffices to show that for all $x \in [\frac{1}{2},1-2^{-\epsilon}]$ and $\delta' \leq u \leq 1-\delta'$,
\begin{align}\label{bbb}
\alpha(x(1-u),x)<u\ln \frac{1}{1-x}, 
\end{align}
Now using the definition of $\alpha$ in~\eqref{bbb} we obtain:
\begin{align*}
x(1-u)\ln (1-u)+(1-x(1-u))\ln \left(\frac{1-x(1-u)}{1-x} \right)<-u\ln (1-x)
\end{align*}
which is equivalent to 
\begin{align}\label{eq:goal}
x(1-u)\ln \frac{(1-u)(1-x)}{1-x(1-u)}+\ln(1-x(1-u))<(1-u)\ln (1-x).
\end{align} 
Finally, to prove~\eqref{eq:goal} we fix $u\in [\delta', 1-\delta']$  and consider  the function
 $$G(x)=x(1-u)\ln \frac{(1-u)(1-x)}{1-x(1-u)}+\ln(1-x(1-u))-(1-u)\ln (1-x)$$
 defined in $x \in [\frac{1}{2},1-2^{-\epsilon}] $. Clearly, it suffices to show that $G(x)<0$ for all $x,u$ of interest.

Towards that end, we calculate the derivative with respect to $x$:
$$G'(x)=(1-u)\ln \frac{(1-u)(1-x)}{1-x(1-u)}-\frac{xu}{(1-x)(1-x(1-u))}-\frac{1-u}{1-x(1-u)}+(1-u)\frac{1}{1-x}$$
which, after some elementary algebra, simplifies to
\begin{align}\label{ccc}
G'(x)=(1-u)\ln \frac{(1-u)(1-x)}{1-x(1-u)}-\frac{xu^2}{(1-x)(1-x(1-u))}.
\end{align}
Now the righthand side of~\eqref{ccc} is clearly strictly less than zero since $0<\delta' \leq u \leq 1-\delta'<1$ and $0<\frac{1}{2} \leq x \leq 1-2^{-\epsilon}<1$ and, therefore, both of its summands are negative.

As a consequence, we get that $G(x)$ is decreasing and, in particular,
\begin{align}\label{ksydia}
G(x) \leq \lim_{x \rightarrow \frac{1}{2}} G(x)=\frac{(1-u)}{2}\ln \frac{1-u}{1+u}+\ln \left(\frac{1+u}{2} \right)+(1-u)\ln 2.
\end{align}
Recalling the definition of the binary entropy $h(w)=-w\log_2w-(1-w)\log_2(1-w), w \in [0,1]$ we see that the righthand side of~\eqref{ksydia} (after we divide by $\ln2 $) can be written as:
\begin{eqnarray*}
\frac{\lim_{x \rightarrow \frac{1}{2}} G(x)}{ \ln 2 }     & =& \frac{(1-u)}{2}\log_2 \frac{1-u}{1+u}+\log_2 \left(\frac{1+u}{2} \right)+(1-u) \\
						    & =& \frac{(1-u)}{2}\log_2  \left( 2\cdot \frac{(1-u)}{2} \right)  - \frac{(1-u)}{2} \log_2 \left(2 \cdot \frac{1+u}{2} \right)   + \log_2\left( \frac{ 1+u}{2 } \right) +(1-u) \\
						    & =&  \frac{(1-u)}{2}\log_2 \left(  \frac{1-u}{2} \right)  - \frac{(1-u)}{2} \log_2 \left( \frac{1+u}{ 2} \right)  +  \log_2\left( \frac{ 1+u}{2 } \right) +(1-u)  \\
						    & =& \frac{(1-u)}{2}\log_2 \left(  \frac{1-u}{2} \right)  + \frac{(1+u)}{2} \log_2 \left( \frac{1+u}{ 2} \right)   +(1-u)   \\
						    & =&-  h\left(  \frac{ 1-u}{2 } \right) + (1-u).
\end{eqnarray*}
Therefore,  $\lim_{x \rightarrow \frac{1}{2}} G(x)<0$  if and only if $h(\frac{1-u}{2})>1-u$. The latter  inequality follows from the fact that $h$ is strictly concave in $[0,\frac{1}{2}]$ and, further,  $h(0)=0$ and $h(\frac{1}{2})=1.$ Indeed, as a corollary of this fact we obtain that $h(w)>2h(\frac{1}{2})w=2w$   for any $w \in (0,\frac{1}{2})$  and, thus, the result claim follows  by setting $w=\frac{1-u}{2} \in [\frac{\delta'}{2},\frac{1-\delta'}{2}] \subset (0,\frac{1}{2}).$ Thus, we have shown that $G(x)< 0$ for every $x,u$ of interest, completing the proof of the lemma.

\section{Proofs omitted from Section~\ref{local_search_proof}}\label{omitted_local}
In this section we present the proofs of  Lemmas~\ref{whatsleft},~\ref{dist_lemma},~\ref{SwapLimMax}  and of Corollary~\ref{cor:prob_bound}.

\subsection{Proof of Lemma~\ref{whatsleft}} 
Let   $\delta' \in  (0,1)$ be a constant to be defined later.  Our first step is to prove that the number of negative tests, $n_{\mathrm{neg}}$, is at least $(1-\delta') n(1- \frac{\nu}{k})^k$ asymptotically almost surely. To see this, observe that each test is negative with probability $(1- \frac{\nu}{k})^{k}$ independently of all other tests. Hence, by linearity of expectation we have  $\ex[n_{\mathrm{neg}} ] = n(1- \frac{\nu}{k})^k $ and standard Chernoff estimates imply that 
$$\Pr[  n_{\mathrm{neg} } < (1-\delta') \ex[  n_{\mathrm{neg} }]   ] \le \mathrm{exp}(- (\delta')^2 \ex[ n_{\mathrm{neg} }] /2 )   \xrightarrow{ p \rightarrow +\infty} 0.$$

Next, we will condition on that $n_{\mathrm{neg} }  = N_{\mathrm{neg} }$, where $ N_{\mathrm{neg} } \ge (1-\delta') n(1- \frac{\nu}{k})^k    $ is a fixed value. In this case, since the probability that a fixed non-defective item appears in a particular test
is $\nu/k$, we see that
\begin{eqnarray}
\mathbb{E}[q \mid n_{\mathrm{neg} }  = N_{\mathrm{neg} } ] &= & (p-k) \left( 1 - \frac{\nu}{k} \right)^{ N_{ \mathrm{neg}} }      \nonumber \\
     & \le & p \cdot \mathrm{exp} \left( -(1-\delta') \frac{\nu}{k}  n  \mathrm{e}^{ -  \frac{ k}{ \frac{k}{\nu} -1}  } 	\right)	\label{quicky}\\
    & \le & p  \left( \frac{k}{p} \right)^{  \frac{ (1-\delta') \nu \cdot \mathrm{exp}\left(-\nu (1 +  \frac{ \nu}{2k } ) \right) }  {R \ln2}   }  \label{quicky2},
\end{eqnarray}
for sufficiently large $p$. Note that in~\eqref{quicky} we used the facts that $1 - x \le \mathrm{e}^{-x}$,  for every $x \in \mathbb{R}$, and that $1 - 1/x > \mathrm{e}^{ -1/( x-1 ) } $, for $x \ge 2$, while in~\eqref{quicky2} the fact that ${p \choose k } \ge \left( \frac{p  }{k }   \right)^k$   and the definition of $R$.

Note also that our assumption  that $R \ge \frac{ \nu \mathrm{e}^{ -\nu }  }{  \ln 2}  $ implies that as $p \rightarrow +\infty$, $$ p  \left( \frac{k}{p} \right)^{  \frac{ (1-\delta') \nu \cdot \mathrm{exp}\left(-\nu (1 +  \frac{ \nu}{2k } ) \right) }  {R \ln2}   }  = \Omega( p^{ \frac{\delta'}{2}})=\omega(1).$$ Thus, since each non-defective item participates in a negative test independently of the other non-defective items,  standard Chernoff estimates  imply that
\begin{align*}
\Pr\left[   q > (1+\eta)  p  \left( \frac{k}{p} \right)^{  \frac{(1-\delta') \nu \cdot \mathrm{exp}\left(-\nu (1 +  \frac{ \nu}{2k } ) \right) }  {R \ln2}   }  \big| n_{ \mathrm{neg}} = N_{ \mathrm{neg} }   \right]  \le \mathrm{exp}( - \eta^2 p^{ \delta'/2 } /3 ) \xrightarrow{p \rightarrow +\infty }  0.
\end{align*}
Choosing $\delta' = \eta$, concludes the proof, since $ N_{\mathrm{neg} } \ge (1-\delta') n(1- \frac{\nu}{k})^k    $ is chosen arbitrarily, and   we have already established that $n_{\mathrm{neg} } \ge(1-\delta') n(1- \frac{\nu}{k})^k   $ asymptotically almost surely.

\subsection{Proof of Lemma~\ref{dist_lemma}}

We proceed by analyzing the (clearly i.i.d.) impact of each test to the sum $ \sum_{ i \in \mathrm{ND}_{\sigma}, j \in \mathrm{D}_{\setminus \sigma}  }  	( \Theta_j -  \Theta_{i \overline{j} } ) $. Note that each test can either contribute positively to the sum $ \sum_{ i \in \mathrm{ND}_{\sigma}, j \in \mathrm{D}_{\setminus \sigma}  }  	( \Theta_j -  \Theta_{i \overline{j} } ) $, by being counted by  variables in the family $\{ \Theta_j \}_{j \in \mathrm{D}_{\setminus \sigma } }$, or negatively, by being counted by variables in the family $\{ \Theta_{i \overline{j}} \}_{i \in \mathrm{ND}_{\sigma}, j \in \mathrm{D}_{\setminus \sigma} }  $, or not at all.  Below we fix a certain test and estimate its contribution which we denote by $C_q^{(\ell) }$. Clearly $C_q^{ (\ell) }$ are i.i.d. random variables and $ \sum_{ i \in \mathrm{ND}_{\sigma}, j \in \mathrm{D}_{\setminus \sigma}  }  	( \Theta_j -  \Theta_{i \overline{j} } )   	\overset{d}{=}     \sum_{q = 1}^{ n} C_q^{(\ell) } .$

To be counted by the variables in the family $\{ \Theta_j \}_{j \in D_{\setminus \sigma} }$, the test should not contain any item of $\sigma$, which happens with probability $(1-\frac{\nu}{k})^{ k'}  $. Conditional on this event, the positive contribution of  the test to the sum $ \sum_{ i \in \mathrm{ND}_{\sigma}, j \in \mathrm{D}_{\setminus \sigma}  }  	( \Theta_j -  \Theta_{i \overline{j} } ) $ equals simply to  $| \mathrm{ND}_{\sigma}| \cdot X = (k'-\ell )\cdot X$, where $X$ is the random variable equal to the number of defective items in $D_{\setminus \sigma }$ tht participate in the test.  Clearly, $X \sim \mathrm{Bin}(k-\ell, \frac{\nu}{k})$.

For each non-defective item $i \in \mathrm{ND}_{\sigma}$, let $E_i$ denote the event that $i$ is the only element in $\sigma$ that participates in the test, and that the test is positive. Observe that events $E_i$ are mutually-exclusive by definition and, therefore, the negative contribution of the test to the sum $ \sum_{ i \in \mathrm{ND}_{\sigma}, j \in \mathrm{D}_{\setminus \sigma}  }  	( \Theta_j -  \Theta_{i \overline{j} } ) $ is given by  $\mathbbm{1}(\cup_{i \in \mathrm{ND}_{\sigma}}  E_i) \cdot Y = \sum_{ i \in \mathrm{ND}_{\sigma} }  \mathbbm{1}(E_i) \cdot Y$,
%\begin{align*}
% \sum_{ i \in \mathrm{ND}_{\sigma} }  \mathbbm{1}(E_i) \cdot Y_i,
%\end{align*}
where $Y \sim \mathrm{Bin}(k-\ell, 1- \frac{\nu}{k}  )  $ counts the number of items in $\mathrm{D}_{  \setminus \sigma} $ that do not participate in the test.  Now it is not hard to see that
\begin{align*}
\Pr[E_i]  =  \frac{\nu}{k} \left( 1- \frac{\nu}{k} \right)^{k'-1}   \left(1 - \left(1 - \frac{\nu}{k} \right)^{k-\ell} \right)
\end{align*}
and, therefore
\begin{align*}
\Pr[\cup_{i \in \mathrm{ND}_{\sigma}  } E_i] = \sum_{i \in \mathrm{ND}_{\sigma}  } \Pr[E_i] = \nu \left(\frac{k'}{k}- \frac{\ell}{k} \right) \left( 1- \frac{\nu}{k} \right)^{k'-1}   \left(1 - \left(1 - \frac{\nu}{k} \right)^{k-\ell} \right).
\end{align*}

Combining the findings in the above two paragraphs, we conclude that the random variable $C_q$ indeed follows the distribution as described in the statement of the lemma, which completes the proof.

\subsection{Proof of Corollary~\ref{cor:prob_bound}}

We use the fact that $\mathrm{Bin}(k-\ell,1-\frac{\nu}{k} ) \le k-\ell$ with probability $1$ to obtain that  each of the  random variables  $C_q^{( \ell)  } $, $q \in \{1, \ldots, n \}$ in Lemma~\ref{dist_lemma} is stochastically dominated by a random variable $D^{(\ell)}_q$, whose distribution samples with probability 
$(1- \frac{\nu}{k} )^{k'}$ from a (scaled) binomial random variable $(k'-\ell)\mathrm{Bin}(k-\ell, \frac{\nu}{k})$, it is equal to $-(k-\ell)$ with probability $\nu \left(\frac{k'}{k}- \frac{\ell}{k} \right) \left( 1- \frac{\nu}{k} \right)^{k'-1}   \left(1 - \left(1 - \frac{\nu}{k} \right)^{k-\ell} \right)$, and it is $0$, otherwise.

We then observe that we can couple each $D^{(\ell)}_q$ variable with the corresponding $T^{(\ell)}_q$ variable in the obvious way,  so that
$$\sum_{q=1}^{n } D^{(\ell)}_q  = \sum_{q=1}^{n}(k-\ell) T^{(\ell)}_q, $$
and, therefore, we obtain:
$$  \Pr \left[  \sum_{ i \in \mathrm{ND}_{\sigma}, j \in \mathrm{D}_{\setminus \sigma}  }   	( \Theta_j -  \Theta_{i \overline{j} } )   	\le 0	\right]  \le \Pr \left[  \sum_{q = 1}^{ n} D^{(\ell)}_q   	\le 0	\right]   \le  \Pr \left[  \sum_{q = 1}^{ n} T^{(\ell)}_q   	\le 0	\right], $$
concluding the proof.
\subsection{Proof of Lemma~\ref{SwapLimMax}} 

It suffices to show the uniform convergence over $\zeta \in [0,1-\delta)$ of the sequence of functions mapping each  $ \zeta \in [0,1-\delta) $ to $ \frac{\ln\mathbb{E}\left[ \mathrm{e}^{-\lambda T^{(\lfloor \zeta k \rfloor)}_1} \right]  }{ (1+\epsilon -\zeta )\ln2 },$ as $p \rightarrow +\infty$, Since the denominator is independent of $p$ and bounded away from zero by $\max\{ \epsilon, \delta  \}\ln 2 >0$ (recall that we have assumed $\epsilon \delta > 0$) it suffices to show uniform convergence over $\zeta \in [0,1-\delta)$ of the sequence of functions mapping each $ \zeta \in [0,1-\delta)$ to $  \ln\mathbb{E}\left[ \mathrm{e}^{-\lambda T^{(\lfloor \zeta k \rfloor)}_1} \right],$ as $p \rightarrow +\infty$.

Notice that if $\lambda=0$ the uniform convergence is immediate as our sequence of functions is identically equal to zero. Hence, from now on we assume without loss of generality that $\lambda>0.$

By direct computation for every $\zeta \in [0,1-\delta)$ it holds
\begin{align}\label{moment_gen}
\mathbb{E}\left[ \mathrm{e}^{-\lambda T^{(\lfloor \zeta k \rfloor)}_1} \right]&=\left(1-\frac{\nu}{k} \right)^{k'}\left( \left(\mathrm{e}^{-\mu}\frac{\nu}{k}+1-\frac{\nu}{k}\right)^{k-\lfloor \zeta k \rfloor}-1\right) \nonumber \\
												&+\left(\frac{k'}{k}-\frac{\lfloor \zeta k \rfloor}{k}\right)\nu\left(1-\frac{\nu}{k}\right)^{k'-1}\left( 1-\left(1-\frac{\nu}{k}\right)^{k-\lfloor \zeta k \rfloor}\right)(\mathrm{e}^{\lambda}-1)+1,
\end{align}
where $\mu = \mu( \lambda, \epsilon, \zeta,k) = \lambda  \frac{  \lfloor (1+\epsilon)k  \rfloor -   \lfloor \zeta k \rfloor   }{k - \lfloor \zeta k \rfloor }$. Since $k \rightarrow +\infty$ as $p \rightarrow +\infty$ 
we conclude by using elementary calculus that
\begin{align}\label{limitara}
  \lim_{p \rightarrow +\infty}   \mathbb{E}\left[ \mathrm{e}^{-\lambda T^{(\lfloor \zeta k \rfloor)}_1} \right]& = \mathrm{e}^{-\nu (1+\epsilon) }(\mathrm{e}^{\nu(1-\zeta)  (\mathrm{e}^{- \frac{\lambda (1 + \epsilon -\zeta)}{1 -\zeta }  }-1) }-1)  \nonumber \\ 
  &+  (1+\epsilon-\zeta)\nu \mathrm{e}^{-\nu (1+\epsilon) }(1-\mathrm{e}^{-\nu(1-\zeta)})(\mathrm{e}^{\lambda}-1)  +1.
\end{align}
Now notice that since $\epsilon, \nu>0$ and $\lambda \geq 0, \zeta<1$, for all $\zeta \in [0,1-\delta)$, we have:
\begin{align}\label{lb:limit}
\lim_{p \rightarrow +\infty}   \mathbb{E}\left[ \mathrm{e}^{-\lambda T^{(\lfloor \zeta k \rfloor)}_1} \right]>1-\mathrm{e}^{-\nu(1+\epsilon) }(1 - (1+\epsilon- \zeta) \nu ( 1 - \epsilon^{- \nu(1-\zeta) } )(\mathrm{e}^{ \lambda} -1  )>1-\mathrm{e}^{-\nu(1+\epsilon)}>0.
    \end{align}Hence the limiting function is always positive. Therefore, using the elementary inequality $\ln x \leq x-1,$ for all $x>0$ to conclude the desired uniform convergence it suffices to show that
\begin{align}
  \lim_{p \rightarrow +\infty}  \max_{\zeta \in [0,1-\delta)}  \left| \frac{\mathbb{E}\left[ \mathrm{e}^{-\lambda T^{(\lfloor \zeta k \rfloor)}_1} \right]}{\lim_{p \rightarrow +\infty} \mathbb{E}\left[ \mathrm{e}^{-\lambda T^{(\lfloor \zeta k \rfloor)}_1} \right]}-1\right| =0.\label{unif_goal}
\end{align}

In what follows we use the $O(\cdot)$ notation, including constants possibly depending on the (fixed for the purposes of this lemma) $\nu,\lambda, \epsilon$ but not on $\zeta$. Using that for $a=O(1),$ $(1+a/k)^k=\mathrm{e}^a (1+O(1/k))$, $|\lfloor \zeta k \rfloor /  k-\zeta|\leq 1/k $ and for $x=o(1), \mathrm{e}^x=1+O(x)$, we have that for arbitrary $\zeta \in [0,1-\delta)$ and  $k$ sufficiently large,
\begin{align}
  \left(1-\frac{\nu}{k} \right)^{k'}\left( \left(\mathrm{e}^{-\mu}\frac{\nu}{k}+1-\frac{\nu}{k}\right)^{k-\lfloor \zeta k \rfloor}-1\right) &= \left(1+O\left(\frac{1}{k}\right) \right)\mathrm{e}^{-(1+\epsilon)\nu}\left(\mathrm{e}^{(\mathrm{e}^{-\mu}-1)\nu (1-\frac{\lfloor \zeta k \rfloor}{k})}-1\right) \nonumber\\
  &=\left(1+O\left(\frac{1}{k}\right)\right)\mathrm{e}^{-(1+\epsilon)\nu}\left(\mathrm{e}^{(\mathrm{e}^{-\mu}-1)\nu (1-\zeta+O\left(\frac{1}{k}\right))}-1\right) \nonumber\\
  % &=\left(1+O\left(\frac{1}{k}\right)\right)\mathrm{e}^{-(1+\epsilon)\nu}\left(\mathrm{e}^{(\mathrm{e}^{-\mu}-1)\nu (1-\zeta)+O\left(\frac{(1-\zeta)(\mathrm{e}^{-\mu}-1)}{k }\right)}-1\right) \nonumber\\
%&=\left(1+O\left(\frac{1}{k}\right)\right)\mathrm{e}^{-(1+\epsilon)\nu}\left(\mathrm{e}^{(\mathrm{e}^{-\mu}-1)\nu (1-\zeta)+O\left(\frac{1}{k \epsilon}\right)}-1\right), \text{ using } 1-\zeta \geq \epsilon, \nonumber\\
   &=\left(1+O\left(\frac{1}{k}\right)\right)\mathrm{e}^{-(1+\epsilon)\nu}\left(\mathrm{e}^{\nu (\mathrm{e}^{-\mu}-1)\nu (1-\zeta)}-1\right) \nonumber\\
   &=\mathrm{e}^{-\nu}\left(\mathrm{e}^{\nu \left(\mathrm{e}^{-\mu}-1\right)(1-\zeta)}-1\right)+O\left(\frac{1}{k}\right) \label{first_term}.
\end{align}Using the same identities and elementary algebra we also have for arbitrary $\zeta \in [0,1-\delta)$ and  $k$ sufficiently large,
\begin{align}
 & \left( \frac{k'}{k}-\frac{\lfloor \zeta k \rfloor}{k}\right)\nu\left(1-\frac{\nu}{k}\right)^{k'-1}\left( 1-\left(1-\frac{\nu}{k}\right)^{k-\lfloor \zeta k \rfloor}\right)\left(\mathrm{e}^{\lambda}-1\right) \nonumber\\
   &=\left(1 + \epsilon - \zeta+O\left(\frac{1}{k}\right)\right)\left(1+O\left(\frac{1}{k}\right)\right) \nu \mathrm{e}^{-(1+\epsilon)\nu}\left(1-\mathrm{e}^{-\nu(1-\zeta)}\right)\left(\mathrm{e}^{\lambda}-1\right) \nonumber \\
  &=\left(1 + O\left(\frac{1}{k }\right)\right) \left(1+\epsilon-\zeta \right)\nu \mathrm{e}^{-(1+\epsilon)\nu}\left(1-\mathrm{e}^{-\nu(1-\zeta)}\right)\left(\mathrm{e}^{\lambda}-1\right) \nonumber\\
%  &=\left(1+\epsilon - \zeta + O\left(\frac{1}{k }\right)\right) \left(1-\zeta \right)\nu \mathrm{e}^{-(1+\epsilon)\nu}\left(1-\mathrm{e}^{-\nu(1-\zeta)}\right)\left(\mathrm{e}^{\lambda}-1\right) \nonumber\\
  &= \left(1 + \epsilon - \zeta \right)\nu \mathrm{e}^{-(1+\epsilon)\nu}\left(1-\mathrm{e}^{-\nu(1-\zeta)}\right)\left(\mathrm{e}^{\lambda}-1\right)+O\left(\frac{1}{k }\right). \label{second_term}
\end{align} 
Hence, combining \eqref{moment_gen}, \eqref{limitara} \eqref{first_term} and \eqref{second_term} we conclude that for sufficiently large $p$ (and therefore $k$) it holds
\begin{align}
&\max_{\zeta \in [0,1)}  \left| \frac{\mathbb{E}\left[ \mathrm{e}^{-\lambda T^{(\lfloor \zeta k \rfloor)}_1} \right]}{\lim_{p \rightarrow +\infty} \mathbb{E}\left[ \mathrm{e}^{-\lambda T^{(\lfloor \zeta k \rfloor)}_1} \right]}-1 \right|=\frac{O\left(\frac{1}{k} \right)}{\lim_{p \rightarrow +\infty}   \mathbb{E}\left[ \mathrm{e}^{-\lambda T^{(\lfloor \zeta k \rfloor)}_1} \right]}.
\end{align}But now since from \eqref{lb:limit} the limiting function is uniformly lower bounded by $1-\mathrm{e}^{-(1+\epsilon)\nu}>0$ we conclude that 
\begin{align*}
    \max_{\zeta \in [0,1)}  \left| \frac{\mathbb{E}\left[ \mathrm{e}^{-\lambda T^{(\lfloor \zeta k \rfloor)}_1} \right]}{\lim_{p \rightarrow +\infty} \mathbb{E}\left[ \mathrm{e}^{-\lambda T^{(\lfloor \zeta k \rfloor)}_1} \right]}-1\right|= O\left(\frac{1}{k \left(1-\mathrm{e}^{-(1+\epsilon)\nu}\right)} \right)= O\left(\frac{1}{k} \right),
\end{align*}from which \eqref{unif_goal} follows. This completes the proof of the lemma.

\section{Proof of Corollary~\ref{no_false_negatives}}\label{CorProof} 

Notice that at each step the algorithm either finds a new state $\sigma$ which satisfies a strictly larger number of positive tests, or it terminates. As a consequence, it terminates in at most $n$ steps, as there are at most $n$ positive tests in total.  

Now, if we can show that no $(0,1/100)$-bad local minima exist, then this means that the algorithm will never terminate in a state $\sigma$ which does not contain all $k$ defective items, concluding the proof. To do this, we apply Theorem~\ref{local_search_theorem} with
$\nu =  \ln (5/2) $ and $\lambda = \ln(100/83)$ and $\epsilon = 1/100$. 

We prove the following lemma.
\begin{lemma}\label{monotonicity}
Let $\nu = \ln(5/2)$, $\lambda = \ln(100/83)$ and $\epsilon = 1/100$. The function 
\begin{align*}
Q( \zeta) = -\frac{\ln\left(\mathrm{e}^{-\nu(1+\epsilon) }(\mathrm{e}^{\nu(1-\zeta)  (\mathrm{e}^{-\frac{\lambda (1 + \epsilon -\zeta)}{1 -\zeta }}-1) }-1)  +  (1+ \epsilon-\zeta)\nu \mathrm{e}^{-\nu(1+\epsilon)  }(1-\mathrm{e}^{-\nu(1-\zeta)})(\mathrm{e}^{\lambda}-1)  +1 \right)  }{ (1+\epsilon-\zeta) \ln2}
\end{align*}
is increasing in $\zeta \in [0, 1)$.
\end{lemma}

Theorem~\ref{local_search_theorem} and Lemma~\ref{monotonicity} imply that if
\begin{align}\label{corrrrr} 
R <  \frac{  \ln(5/2) \cdot (2/5)   }{\ln2} + Q(\ln(100/83), 0, \ln(5/2),  1/100),
\end{align} 
then no bad $(0,1/100)$-local bad minima exist. The fact that the righthand side of~\eqref{corrrrr} is at least $0.5468$ concludes the proof.

\subsection{Proof of Lemma~\ref{monotonicity}}\label{proof_corr_monotonicity}
To prove the lemma we  show  that the derivative of $Q(\zeta) $  is positive. We define
\begin{align*}
f(\zeta) &=  \ln\left(\mathrm{e}^{-\nu(1+\epsilon) }(\mathrm{e}^{\nu(1-\zeta)  (\mathrm{e}^{-\frac{\lambda (1 + \epsilon -\zeta)}{1 -\zeta }}-1) }-1)  +  (1+ \epsilon-\zeta)\nu \mathrm{e}^{-\nu(1+\epsilon)  }(1-\mathrm{e}^{-\nu(1-\zeta)})(\mathrm{e}^{\lambda}-1)  +1 \right), \\
g(\zeta) &= (1+\epsilon-\zeta) \ln 2,
\end{align*}
%\begin{align*}
%f(\zeta) &=  \ln\left(\mathrm{e}^{-\nu }(\mathrm{e}^{\nu(1-\zeta)  (\mathrm{e}^{-\frac{\lambda (1 -\zeta)}{1 -\zeta }}-1) }-1)  +  (1-\zeta)\nu \mathrm{e}^{-\nu  }(1-\mathrm{e}^{-\nu(1-\zeta)})(\mathrm{e}^{\lambda}-1)  +1 \right), \\
%g(\zeta) &= (1+\epsilon-\zeta) \ln 2,
%\end{align*}
so that $Q(\zeta) = - f(\zeta) / g(\zeta)$ and, therefore,
\begin{align}\label{deriv_def}
\frac{\partial Q( \zeta) }{\partial \zeta }  =  -\frac{  g(\zeta) \cdot \frac{ \partial f(\zeta) }{ \partial \zeta }  - f(\zeta) \cdot \frac{  \partial g(\zeta) }{\partial \zeta }      }{ g(\zeta)^2 }.
\end{align}

We need to show that the righthand side of~\eqref{deriv_def} is positive and, to that end, we need to calculate $ \frac{ \partial f(\zeta)}{ \partial \zeta } $ and $ \frac{ \partial g(\zeta)}{ \partial \zeta } $. The latter is trivial to calculate:
\begin{align}\label{g_deriv}
\frac{ \partial g(\zeta)}{ \partial \zeta }  = - \ln 2.
\end{align}
For the former, setting 
\begin{align*}
h(\zeta) =  \mathrm{e}^{-\nu(1+\epsilon) }(\mathrm{e}^{\nu(1-\zeta)  (\mathrm{e}^{-\frac{\lambda (1 + \epsilon -\zeta)}{1 -\zeta }}-1) }-1)  +  (1+ \epsilon-\zeta)\nu \mathrm{e}^{-\nu(1+\epsilon)  }(1-\mathrm{e}^{-\nu(1-\zeta)})(\mathrm{e}^{\lambda}-1)  +1
\end{align*}
so that $f(\zeta) = \ln h(\zeta)$, we have 
\begin{align}\label{pra_ksydia}
\frac{\partial f(\zeta)  }{ \partial \zeta}  = \frac{ 1}{ h(\zeta)} \cdot \frac{ \partial h(\zeta) }{ \partial \zeta } 
\end{align}
and so we focus on calculating $\frac{ \partial h(\zeta) }{ \partial \zeta }$.  Observe now that $h(\zeta) = a(\zeta) + b(\zeta) + 1$ where:
\begin{align*}
a(\zeta) &=      \mathrm{e}^{-\nu(1+\epsilon) }(\mathrm{e}^{\nu(1-\zeta)  (\mathrm{e}^{-\frac{\lambda (1 + \epsilon -\zeta)}{1 -\zeta }}-1) }-1),    \\
b(\zeta) & =  (1+ \epsilon-\zeta)\nu \mathrm{e}^{-\nu(1+\epsilon)  }(1-\mathrm{e}^{-\nu(1-\zeta)})(\mathrm{e}^{\lambda}-1) 
\end{align*}
and, thus, $ \frac{  \partial h(\zeta)}{ \partial \zeta } = \frac{  \partial a(\zeta)}{ \partial \zeta } + \frac{  \partial b(\zeta)}{ \partial \zeta } $.

We start by calculating the derivative of $a(\zeta)$, and  by direct calculation we see that
\begin{eqnarray}
 \frac{ \partial   a(\zeta) } { \partial \zeta } & =&  \frac{  \partial     (   \mathrm{e}^{-\nu(1+\epsilon) }  \mathrm{e}^{\nu(1-\zeta)  (\mathrm{e}^{-\frac{\lambda (1 + \epsilon -\zeta)}{1 -\zeta }}-1) } )      }{ \partial \zeta }  \nonumber  \\
 								  &  = &  \mathrm{e}^{-\nu(1+\epsilon) }  \mathrm{e}^{\nu(1-\zeta)  (\mathrm{e}^{-\frac{\lambda (1 + \epsilon -\zeta)}{1 -\zeta }}-1) }   \frac{  \partial  (  \nu(1-\zeta)  (\mathrm{e}^{-\frac{\lambda (1 + \epsilon -\zeta)}{1 -\zeta }}-1)  )}{ \partial \zeta } \nonumber \\
								  & =&   \mathrm{e}^{-\nu(1+\epsilon) }  \mathrm{e}^{\nu(1-\zeta)  (\mathrm{e}^{-\frac{\lambda (1 + \epsilon -\zeta)}{1 -\zeta }}-1) }   \left( 	  \frac{\partial c(\zeta) }{ \partial \zeta }  		 + \nu \right) \label{ksimerwmata} 
\end{eqnarray}
where,
\begin{align}\label{dineis_dikaiwmata}
c(\zeta) = \nu(1-\zeta) \mathrm{e}^{-\frac{\lambda (1 + \epsilon -\zeta)}{1 -\zeta }}
\end{align}
In particular, we have:
\begin{eqnarray}
\frac{  \partial c(\zeta) }{ \partial \zeta }  & =&   \frac{ \partial( \nu(1-\zeta)   )    }{ \partial \zeta }  \cdot \mathrm{e}^{-\frac{\lambda (1 + \epsilon -\zeta)}{1 -\zeta }} + \nu(1-\zeta)  \frac{ \partial  (  \mathrm{e}^{-\frac{\lambda (1 + \epsilon -\zeta)}{1 -\zeta }}   ) }{ \partial \zeta }  \nonumber \\
							  & =&  - \nu  \mathrm{e}^{-\frac{\lambda (1 + \epsilon -\zeta)}{1 -\zeta }}  + \nu(1-\zeta) \mathrm{e}^{-\frac{\lambda (1 + \epsilon -\zeta)}{1 -\zeta }}  \cdot \frac{   -  \left( (1-\zeta)(-\lambda) - \lambda(1+\epsilon-\zeta) (-1)   \right)   }{ (1-\zeta)^2 }  \nonumber \\
							  & = &- \nu  \mathrm{e}^{-\frac{\lambda (1 + \epsilon -\zeta)}{1 -\zeta }}  + \nu(1-\zeta) \mathrm{e}^{-\frac{\lambda (1 + \epsilon -\zeta)}{1 -\zeta }}  \cdot \frac{     (1-\zeta)\lambda - \lambda(1+\epsilon-\zeta)       }{ (1-\zeta)^2 }  \nonumber \\
							  & =& - \nu  \mathrm{e}^{-\frac{\lambda (1 + \epsilon -\zeta)}{1 -\zeta }}  -   \frac{ \epsilon \lambda  \nu \mathrm{e}^{-\frac{\lambda (1 + \epsilon -\zeta)}{1 -\zeta }}}{1- \zeta }   \nonumber \\
							  & =&   -  \nu \mathrm{e}^{-\frac{\lambda (1 + \epsilon -\zeta)}{1 -\zeta }}  \left( 1 +  \frac{ \epsilon \lambda  }{ 1- \zeta }  \right) \label{peaky_blinders}.
\end{eqnarray}
Combining~\eqref{ksimerwmata} and~\eqref{peaky_blinders} we obtain:  
\begin{align}
 \frac{ \partial   a(\zeta) } { \partial \zeta }  = \nu \mathrm{e}^{-\nu(1+\epsilon) }  \mathrm{e}^{\nu(1-\zeta)  (\mathrm{e}^{-\frac{\lambda (1 + \epsilon -\zeta)}{1 -\zeta }}-1) }    \left( 1	-   \mathrm{e}^{-\frac{\lambda (1 + \epsilon -\zeta)}{1 -\zeta }}  \left( 1 +  \frac{ \epsilon \lambda  }{ 1- \zeta }  \right)   		  \right).									\label{a_final} 
\end{align}

We continue by computing  the derivative of $b(\zeta)$:
\begin{eqnarray}
\frac{ \partial b(\zeta) }{ \partial \zeta }  & =&  \frac{  \partial ( ( 1+ \epsilon-\zeta)\nu \mathrm{e}^{-\nu(1+\epsilon)  }(1-\mathrm{e}^{-\nu(1-\zeta)})(\mathrm{e}^{\lambda}-1) )    }{ \partial \zeta } \nonumber  \\
							  & =& \nu \mathrm{e}^{-\nu(1+\epsilon)  } (\mathrm{e}^{\lambda}-1)  \left( (-1)  ( 1- \mathrm{e}^{ - \nu(1-\zeta) }  ) 		- \nu \mathrm{e}^{ - \nu(1-\zeta) } 		\right)  \nonumber \\
							  & =& \nu \mathrm{e}^{-\nu(1+\epsilon)  } (\mathrm{e}^{\lambda}-1) \left(   (1-\nu) \mathrm{e}^{ - \nu(1-\zeta) }    - 1\right).
\end{eqnarray}

We are now ready to prove that  the righthand side~\eqref{deriv_def} is negative.  Using the notation we have introduced, this amounts to showing that
\begin{align}\label{goalaki_mwre}
(1+\epsilon - \zeta) \frac{  \partial h(\zeta) }{ \partial \zeta } \frac{1}{h(\zeta) }  + \ln \left(  h(\zeta)	\right) < 0,
\end{align}
for every $\zeta \in [0,1)$.

Since $ \nu = \ln(5/2) ,\lambda = \ln(100/83)$ and $\epsilon = 1/100$ we have:
\begin{eqnarray*}
h(\zeta) 	    & = &\left( \frac{2}{5}\right)^{1 + \frac{1}{100} }  \left(  \left(  \frac{5}{2} \right)^{ (1-\zeta) (  (\frac{83 }{100})^{ \frac{1 + \frac{1}{100} - \zeta }{1-\zeta }  }   -1)   } -1 	\right)   \\
		  & &  +   \left(\frac{2}{5}\right)^{1 +\frac{1}{100}}  \left(1+\frac{1}{100}-\zeta\right) \ln\left( \frac{5}{2} \right)   \left(1 - \left( \frac{2}{5} \right)^{ 1-\zeta}  \right) \frac{17}{83}    +1, \\
\frac{ \partial a(\zeta) }{ \partial \zeta }  &=&  \left( \frac{2}{5}\right)^{1 + \frac{1}{100} }  \ln\left( \frac{5}{2} \right)   \left(  \frac{5}{2} \right)^{ (1-\zeta) (  (\frac{83 }{100})^{ \frac{1 + \frac{1}{100} - \zeta }{1-\zeta }  }   -1)   } 	  \left(  1 -  \left(\frac{83 }{100} \right)^{ \frac{1 + \frac{1}{100} - \zeta }{1-\zeta }  }  \left(  1+ \frac{\ln(100/83) }{100( 1-\zeta)} \right)  \right),    \\
\frac{ \partial b(\zeta) }{ \partial \zeta }  &=&  \left( \frac{2}{5}\right)^{1 + \frac{1}{100} }  \ln\left( \frac{5}{2} \right)  \frac{17}{83} \left(   \left(1-\ln\left(  \frac{5}{2}\right) \right)   \left(  \frac{2}{5}  \right) ^{(1-\zeta) }    - 1\right).
\end{eqnarray*}
It can be verified computationally that, for every $\zeta \in [0,1)$, we have  (i)  $h(\zeta) \in ( 0.97,1)$; (ii) $\frac{ \partial h(\zeta) }{ \partial \zeta }=\frac{ \partial a(\zeta) }{ \partial \zeta } + \frac{ \partial b(\zeta) }{ \partial \zeta } < 0 $. Therefore, each summand in~\eqref{goalaki_mwre} is strictly negative, concluding the proof.
\end{document}